\documentclass[11pt]{article}

\usepackage[utf8]{inputenc}
\usepackage{amsmath}
\usepackage{amssymb}
\usepackage{amsfonts}
\usepackage{amsthm}
\usepackage{thmtools}
\usepackage{enumerate}
\usepackage[top=3.5cm, bottom=3.5cm, left=3.5cm, right=3.5cm]{geometry}
\usepackage{graphicx}
\usepackage[all]{xy}
\usepackage{titlesec}
\usepackage[backref]{hyperref}
\usepackage{rotating}
\usepackage{xcolor}
\usepackage{setspace}
\usepackage[T1]{fontenc}
\usepackage{tikz}
\usepackage{comment}
\usepackage[all]{xy}
\usepackage{stmaryrd}
\usepackage[nameinlink,capitalize,noabbrev]{cleveref}
\usepackage{mathrsfs}
\usepackage{imakeidx}
\usepackage{authblk}

\titleformat{\section}
{\filcenter\large\bf} 
{\thesection.\ } 
{1pt} 
{} %

\titleformat{\subsection}[hang]
{\filcenter\bf}
{\thesubsection.}
{1pt}
{}

\titleformat{\subsubsection}[hang]
{\filcenter\bf}
{\thesubsubsection.}
{1pt}
{}

\declaretheoremstyle[bodyfont=\normalfont]{normalbody}

\declaretheorem[numberwithin=section,name=Theorem]{theorem}
\declaretheorem[sibling=theorem,style=normalbody,name=Definition]{definition}
\declaretheorem[sibling=theorem,name=Corollary]{corollary}
\declaretheorem[sibling=theorem,name=Lemma]{lemma}
\declaretheorem[sibling=theorem,name=Proposition]{proposition}

\declaretheorem[sibling=theorem,style=normalbody,name=Example]{example}
\declaretheorem[sibling=theorem,style=normalbody,name=Remark]{remark}

\newcommand{\Z}{\mathbb{Z}}
\newcommand{\N}{\mathbb{N}}
\newcommand{\Q}{\mathbb{Q}}
\newcommand{\R}{\mathbb{R}}
\newcommand{\C}{\mathbb{C}}

\renewcommand{\P}{\mathbb{P}}
\newcommand{\G}{\mathbb{G}}

\newcommand{\End}{\operatorname{End}}
\newcommand{\Hom}{\operatorname{Hom}}
\newcommand{\Aut}{\operatorname{Aut}}
\newcommand{\Gal}{\operatorname{Gal}}

\newcommand{\Div}{\mathrm{Div}}
\newcommand{\divr}{\mathrm{div}}

\newcommand{\id}{\mathrm{id}}

\newcommand{\Spec}{\mathrm{Spec}}
\newcommand{\A}{\mathbb{A}}
\newcommand{\PPDiv}{\mathrm{PPDiv}}
\newcommand{\D}{\mathfrak{D}}

\newcommand{\cadiv}{\mathrm{CaDiv}}

\newcommand{\cone}{\mathrm{cone}}
\newcommand{\relint}{\mathrm{relint}}
\newcommand{\SAut}{\mathrm{SAut}}
\newcommand{\face}{\mathrm{face}}
\newcommand{\Xss}{X^{\mathrm{ss}}}

\newcommand{\sep}[1]{#1^{\mathrm{sep}}}
\newcommand{\aff}[1]{#1_{\mathrm{aff}}}

\newcommand{\can}{\mathrm{can}}
\renewcommand{\H}{\mathrm{H}}

\newcommand{\Addresses}{{
  \bigskip
  \footnotesize

  G.~Martínez-Núñez, \textsc{Departamento de Matemáticas, Facultad de Ciencias, Universidad de Chile}\par\nopagebreak \textit{e-mail address}, G.~Martínez-Núñez: \texttt{gary.martinez@ug.uchile.cl}

}}

\DeclareFontFamily{U}{wncy}{}
\DeclareFontShape{U}{wncy}{m}{n}{<->wncyr10}{}
\DeclareSymbolFont{mcy}{U}{wncy}{m}{n}
\DeclareMathSymbol{\Sh}{\mathord}{mcy}{"58}

\DeclareFontFamily{U}{wncy}{}
\DeclareFontShape{U}{wncy}{m}{n}{<->wncyr10}{}
\DeclareSymbolFont{mcy}{U}{wncy}{m}{n}
\DeclareMathSymbol{\Ch}{\mathord}{mcy}{"51}

\title{\textsc{Polyhedral divisors and algebraic torus actions over arbitrary fields}}

\author[]{\textsc{Gary Martínez-Núnez}}

\begin{document}

\maketitle

\begin{abstract}
We provide an algebro-geometric combinatorial description of affine normal varieties endowed with an effective action of an algebraic torus over arbitrary fields. This description is achieved in terms of proper polyhedral divisors endowed with a Galois semilinear action.
\end{abstract}

		\mbox{\hspace{1.6em} \textbf{ Keywords:} affine varieties, torus actions, Galois descent.}

\vspace{1em}
		\mbox{\hspace{1.6em} \textbf{ MSC codes (2020):} 14L24, 14L30.} 
		\mbox{\hspace{1.3em} }

\tableofcontents

%
%
%
%

\section{Introduction}

\indent 

Since the work of Demazure in \cite{Dem70}, normal varieties endowed with an effective torus action have been extensively studied. In that work, toric varieties naturally emerged, and the author provided a combinatorial description of the smooth ones. At that time, these varieties were referred to as \textit{toroidal embeddings}, as seen in \cite{KKMSD73} and \cite{Oda78}. A foundational survey based on earlier works was presented by Danilov in \cite{Dan78}, where these varieties were referred to as \textit{toric varieties} for the first time.\footnote{Danilov called them \textit{toral} in russian, but in the English translation they appeared as \textit{toric}. See \cite[Appendix A]{CLS11} for a brief historical overview of toric varieties.}

In general, normal toric varieties can be understood in terms of \textit{cones} or \textit{fans} (for modern references, see \cite{CLS11} or \cite{Ful93}). In the words of Fulton, \textit{toric varieties have provided a remarkably fertile testing ground for general theories}. Furthermore, toric varieties have found numerous applications in physics and computational fields of study.

Throughout this century, new results have emerged regarding toric varieties. Almost all the works mentioned above were developed over algebraically closed fields, as all algebraic tori are \textit{split} in that context. For non-split toric varieties, i.e. when the algebraic torus acting is non-split, obtaining a combinatorial
description is not possible in terms of cocharacters only. However, since every algebraic torus splits over a finite Galois extension, it is possible to obtain a combinatorial description accompanied by a Galois action (see \cite{Hur11} and \cite{ELFST14}). Moreover, various taxonomies can be applied to classify non-split toric varieties, depending on the definition of toric varieties and the types of morphisms considered \cite{Dun16}.

A toric variety contains a \emph{torsor} of an algebraic torus as a dense open subvariety, and their dimensions coincide. For a variety endowed with an effective torus action (referred to as a $T$-variety for brevity), the \textit{complexity} is defined as the difference between the dimensions of the variety and the torus. Thus, a toric variety is a $T$-variety of complexity zero.

For normal $T$-varieties of complexity one, Mumford \cite{KKMSD73} provided a description in terms of \textit{toroidal fans}\footnote{This is modern terminology. In \cite{KKMSD73}, what we now call fans were referred to as \textit{finite rational partial polyhedral decompositions}.}. Unfortunately, such a combinatorial description does not extend to higher complexities, even for complexity two. Furthermore, the works of Pinkham \cite{Pin77} and Flenner and Zaidenberg \cite{FZ03}, both focused on complexity one surfaces and restricted to the complex numbers.

It was not until 2006 that an \emph{algebro-geometric} combinatorial description for affine normal $T$-varieties over algebraically closed fields of characteristic zero was achieved for arbitrary complexity. The object encoding the data of an affine normal $T$-variety was called a \textit{proper polyhedral divisor} by Altmann and Hausen \cite{AH06}.

Let $k$ be a field and $Y$ be a normal semi-projective variety over $k$ (see \cref{section: semi-projective varieties} for the notion of semi-projectivity). Let $N$ be a lattice and $\omega \subset N_{\Q}$ be a pointed cone. Denote $M := \Hom_{\Z}(N, \Z)$. A proper polyhedral divisor (abbreviated as pp-divisor) is a finite sum
\[
\D := \sum \Delta_{D} \otimes D,
\]
where the $\Delta_{D}$'s are polyhedra in $N_{\Q}$ with tail cone $\omega$, and the $D$'s are prime divisors in $\Div(Y)$, the group of Weil divisors, satisfying suitable conditions (cf. \cref{definitionppdiv}).

Given a pp-divisor $\D$ over a normal semi-projective variety $Y$ over $k$, we can associate with it a \textit{piecewise linear map} $\mathfrak{h}_{\D}: \omega^{\vee} \cap M \to \cadiv_{\Q}(Y)$, where $\cadiv_{\Q}(Y)$ stands for the group of rational Weil divisors $D$ such that $nD$ is Cartier for some $n\in \N$. Altmann and Hausen associated with the pp-divisor $\D$, via the piecewise linear map $\mathfrak{h}_{\D}$, the following $M$-graded $k$-algebra:
\[
A[Y, \D] := \bigoplus_{m \in \omega^{\vee} \cap M} \H^{0}(Y, \mathscr{O}_{Y}(\mathfrak{h}_{\D}(m))) \subset k(Y)[M],
\]
and proved that it is finitely generated. Besides, they also proved that the scheme $X(\D) := \Spec(A[Y, \D])$ is an affine normal variety over $k$ endowed with an effective action of $T := \Spec(k[M])$. Moreover, they showed that every affine normal $T$-variety arises in this manner.

    \begin{theorem}\cite[Theorems 3.1 and 3.4]{AH06}\label{theoremmainaltmannhausen}
        Let $k$ be an algebraically closed field of characteristic zero.
        \begin{enumerate}[i)]
            \item The scheme $X(\D)$ is a normal $k$-variety with an effective action of $T$.
            \item Let $X$ be an affine normal $k$-variety with an effective $T$-action. Then, there exists a pp-divisor $\D$ such that $X\cong X(\D)$ as $T$-varieties.
        \end{enumerate}
    \end{theorem}
    
     Vollmert \cite{Vol10} makes a correspondence between Mumford's toroidal fans and pp-divisors for complexity one affine normal $T$-varieties.
   
          When $k$ is no longer algebraically closed, the combinatorial framework disappears for non-split algebraic tori over $k$, similar to what happens in toric geometry. However, when the algebraic torus is split, the combinatorial structure reappears. Specifically, \cref{theoremmainaltmannhausen} holds over any field $k$ of characteristic zero and affine normal $T$-varieties, provided that $T$ is split over $k$, as demonstrated by Gillard in \cite[Proposition 4.10]{Gil22b}. Moreover, it is also valid for complexity one normal $T$-varieties over arbitrary fields, provided that $T$ is a split torus, as proved by Langlois in \cite[Theorem 0.2]{Lan15}.

\subsection*{Main results}
\indent

When $k$ is any field and $T$ is a split algebraic torus over $k$, we prove that every affine normal $T$-variety over $k$ arises from a pp-divisor over $k$ by adapting the arguments of Altmann and Hausen. The following result generalizes \cref{theoremmainaltmannhausen}, \cite[Proposition 4.10]{Gil22b} and \cite[Theorem 0.2]{Lan15}.
     
      \begin{theorem}\label{theoremmainaltmannhausensplitintro}
        Let $k$ be a field.
        \begin{enumerate}[i)]
            \item\label{theoremmainaltmannhausensplitintro part a} The scheme $X(\D)=\Spec(A[Y,\D])$ is a normal variety over $k$ with an effective action of $T:=\Spec(k[M])$.
            \item\label{theoremmainaltmannhausensplitintro part b} Let $X$ be an affine normal variety over $k$ with an effective action of a split algebraic torus $T$. Then, there exists a pp-divisor $\D$ such that $X\cong X(\D)$ as $T$-varieties.
        \end{enumerate}
    \end{theorem}

Recall that every algebraic torus over a field $k$ splits over a finite Galois extension $L/k$. Consequently, the combinatorial framework exists over such extensions, and Galois descent theory provides a mechanism to \emph{bring it back} to the ground field. In other words, with additional data describing the combinatorial structure over the extension, it is possible to describe the variety over the ground field. To classify normal $T$-varieties over a nonalgebraically closed field, we need to develop an appropriate language: the notion of \textit{Galois semilinear action on a pp-divisor} (see \cref{section: semilinear pp-div}). Let us denote by $\mathfrak{PPDiv}(\Gamma_{L})$ the category of pairs $(\D,\varrho)$, where $\D$ is a pp-divisor defined over $L$ and $\varrho$ is a Galois semilinear action on it. Thus, we prove the following result, which is the main theorem of this work, and generalizes \cite[Theorem A]{Gil22b} and \cite[Theorem 5.10]{Lan15}.

    \begin{theorem}\label{maintheoremofpaper}
Let $k$ be a field and $L/k$ be a finite Galois extension with Galois group $\Gamma_{L}$.
	\begin{enumerate}[a)]
            \item\label{theorem main affine minimal part i} Let $(\D_{L},\varrho)$ be an object in $\mathfrak{PPDiv}(\Gamma_{L})$. Then, $X(\D_{L},\varrho)$ is an affine normal variety endowed with an effective action of an algebraic torus $T$ over $k$ such that $T$ splits over $L$ and $X(\D_{L},\varrho)_{L}\cong X(\D_{L})$ as $T_{\D_{L}}$-varieties over $L$.
            \item\label{theorem main affine minimal part ii} Let $X$ be an affine normal variety over $k$ endowed with an effective $T$-action such that $T_{L}$ is split. Then, there exists an object $(\D_{L},\varrho)$ in $\mathfrak{PPDiv}(\Gamma_{L})$ such that $X \cong X(\D_{L},\varrho)$ as $T$-varieties.
        \end{enumerate}
	\end{theorem}

\paragraph{About nonaffine normal $T$-varieties.}Altmann, Hausen and Süß \cite{AHS08} give a complete description in terms of \emph{divisorial fans} when the ground field is algebraically closed and is of characteristic zero. It is worth mentioning that, for a complexity zero normal $T$-variety, such a divisorial fan is the usual fan from toric geometry. The author, in a subsequent work \cite{MN25b}, generalizes \cite{AHS08} to arbitrary fields and arbitrary torus actions.

\paragraph{Representability of $\underline{\Aut}^{T}(X)$ and counting real forms.} Once the existence of a $k$-form has been established, a natural question arises: How many pairwise non-isomorphic $k$-forms does an affine normal $T$-variety $X$ admit? 

An answer to this question was given in a recent work by Lucchini-Arteche and Terpereau \cite{LAT26a} in a more general framework. They study the number of $k$-forms of a complexity one normal $G$-variety, where $G$ is a reductive algebraic group. Recall that the complexity for a $G$-variety is the codimension of a general $B$-orbit, where $B$ is any Borel subgroup of $G$ over an algebraic closure.

As a first step, the authors study the representability of the group sheaf $\underline{\Aut}^{G}(X)$ defined, over the small site over $k$, by
\begin{align*} \{\textrm{smooth schemes over $k$}\} &\to \{\textrm{groups}\} \\ S &\mapsto \Aut_{S}^{G_{S}}(X_{S}), \end{align*}
where $X$ is a normal $G$-variety. In the particular case when $G=T$ is an algebraic torus and $X$ is an affine normal variety, the authors apply some results from our work to prove the representability of $\underline{\Aut}^{T}(X)$ by a smooth $k$-group scheme $\Aut^{T}(X)$ that is locally of finite type (see \cite[Theorem 1.3]{LAT26a}).

Regarding the number of $(\R,T)$-forms of $X$ when $k=\R$, i.e. the number of real $T$-varieties $X'$ such that $X_{\C}\cong X_{\C}$ as $T_{\C}$-varieties, the authors prove that $X$ admits only finitely many of them (see \cite[Theorem 1.6]{LAT26a}).

It is worth mentioning that, in a subsequent version of their work \cite{LAT26b}, the authors provide global arguments that apply to a wider class of smooth algebraic groups, but over fields of characteristic zero (see \cite[Theorem 1.3]{LAT26b}). Moreover, \cite[Theorem 1.5]{LAT26b} establishes the finiteness of the number of forms for any connected reductive algebraic group.

\paragraph{Conventions and settings.}Throughout this paper, $k$ stands for an arbitrary field. We fix an algebraic closure $\bar{k}$ and, therefore, we set $\sep{k}\subset\bar{k}$ for the separable closure of $k$ inside of the algebraic closure. Besides, $k\subset L\subset \sep{k}$ stands for a finite Galois extension and we denote by $\Gamma_{L}$ its Galois group.

By a \emph{variety} over $k$ (or simply $k$-variety) we mean a geometrically integral separated scheme of finite type over $k$ and $T$ stands for an algebraic torus over $k$. Besides, by a \emph{$T$-variety} over $k$ we mean a variety endowed with an effective action of $T$. 

For any extension $K/k$, we denote by $X_{K}:=X\times_{k} K$ the respective base change.

\paragraph{Layout of the paper.} In \cref{chapter convex geometry and toric varieties}, some basic notions concerning algebraic tori are provided. Besides, a separate section on convex geometry is also included, where the notions of \emph{cones}, \emph{fans}, and \emph{polyhedra} are recalled and also some of their properties.

 In \cref{section Polyhedral Divisors}, the category of \emph{proper polyhedral divisors} over arbitrary fields is presented. This category was originally introduced by Altmann and Hausen in \cite{AH06} over algebraically closed fields of characteristic zero. In this section, the behavior of such a category under base change of field extensions is also studied.

 \cref{Section Normal varieties and split tori} is devoted to the proof of \cref{theoremmainaltmannhausensplitintro}. It begins with a brief exposition about the \emph{affinization} of a scheme in \cref{section: affinization}, a notion that appears in the proof of \cref{theoremmainaltmannhausensplitintro}~\eqref{theoremmainaltmannhausensplitintro part a} in \cref{section: from pp-divisors to varieties} that involves the construction of a non-affine normal $T$-variety whose affinization is $X(\D)$. Hence, in \cref{section: semi-projective varieties} a discussion about \emph{semi-projective} varieties is founded, where several properties of them are exposed. Semi-projective varieties are one of the central notions in this work, because affine normal $T$-varieties arise from pp-divisors defined over semi-projective varieties. Finally, \cref{theoremmainaltmannhausensplitintro}~\eqref{theoremmainaltmannhausensplitintro part b} is proved in \cref{section: from affine to pp-divisors}. The section begins with the construction of a quotient given an affine normal $T$-variety. This construction is followed by the proof of \cref{theoremmainaltmannhausensplitintro}~\eqref{theoremmainaltmannhausensplitintro part b} by constructing a pp-divisor over such a quotient. The algebro-geometric combinatorial data is not unique and the notion of \emph{minimal} pp-divisor is presented, which plays an important role in the subsequent sections. At the end of \cref{Section Normal varieties and split tori}, there is a brief discussion on $G$-normal varieties and how the theory of pp-divisors may help to understand them. 

 \cref{chapter semilinear morphisms} begins with the study of the \emph{functoriality} of the category of pp-divisors in \cref{section: functoriality} towards the category of affine normal varieties endowed with an effective action of a split torus. In \cref{section: semilinear varieties} and \cref{section: semilinear pp-div} \emph{semilinear} morphisms are introduced and a larger category is presented. This category has exactly the same objects but contains more morphisms; it englobes also the semilinear ones. It is also proved that functoriality remains valid. However, it does not yield an equivalence of categories. In \cref{section: semilinear equivariant morphisms}, the main result of the entire section shows up. It is possible, up to a modification of one of the normal semi-projective varieties, to translate any dominant equivariant morphism of affine normal $T$-varieties into a morphism of pp-divisors (cf \cref{theorem semilinear automorphisms}). In the particular case of minimal pp-divisors, no modification is needed at all. \cref{theorem semilinear automorphisms} plays an important role in the last section.

 \cref{chapter equivariant automorphisms affine} starts with the translation of the classical language of Galois descent into the language of semilinear morphisms in \cref{section: Galois descent tools}, all this in the category of varieties. In \cref{section: affine minimal}, a proof of \cref{maintheoremofpaper} is given and it begins with introducing the notion of Galois semilinear action on a pp-divisor. Also, a brief explanation of how \cite[Theorem A]{Gil22b} is recovered from ours is given. Finally, in \cref{section: The other T-variety}, it can be found a discussion on the \emph{other $T$-variety} arising from a pp-divisor, which has been used to study the resolution of singularities in characteristic zero.

\subsection*{Acknowledgements}

I would like to thank Giancarlo Lucchini-Arteche, Alvaro Liendo, Adrien Dubouloz, Ronan Terpereau and Michel Brion for many enriching conversations. I would also like to thank the anonymous referee for carefully reading this work and for their helpful comments. This work was partially supported by ANID via Beca de Doctorado Nacional 2021 Folio 21211482 and ECOS-ANID Project ECOS230044.

\newpage

\section{Convex geometry and toric varieties}\label{chapter convex geometry and toric varieties}
\indent

This section is devoted to summarizing some known facts about convex geometry and toric varieties. We start with \textit{algebraic tori} and some of their properties. We continue with convex geometry, recalling the definitions of \textit{cones} and \textit{fans}. We present also the notion of \textit{polyhedra}. This section is split into two subparts. The first one is about \textit{split} toric varieties and the last one is about non split toric varieties.

\subsection{ Algebraic tori}\label{Section algebraic torus}
\indent

 As a reference for this topic see for instance \cite{Wat79} or \cite{DG70}. An \textit{algebraic torus over $k$} is a linear algebraic group $T$ over $k$ such that 
\[T_{\sep{k}}:=T\times_{\Spec(k)}\Spec(\sep{k})\cong \left(\G_{\mathrm{m},\sep{k}}\right)^{n},\] where $n=\dim(T)$. If this isomorphism holds over $k$, we say that the algebraic $k$-torus is \textit{split}. It is known that there exists a finite Galois extension $k\subset L\subset \sep{k}$ such that $T_{L}$ is split.  A way to construct a split algebraic torus of dimension $n$ over an arbitrary field $k$ is the following: Let $M$ be a free $\Z$-module of rank $n$. The group algebra $k[M]$ is a finitely generated $k$-algebra, which is isomorphic to 
\[k[x_{1},y_{1},x_{2},y_{2},\dots,x_{n},y_{n}]/(x_{1}y_{1}-1,x_{2}y_{2}-1,\dots,x_{n}y_{n}-1)\]
as $k$-algebras. Then, we have the following $k$-algebra isomorphism
\[k[M]\cong k[x_{1},y_{1}]/(x_{1}y_{1}-1)\otimes k[x_{2},y_{2}]/(x_{2}y_{2}-1)\otimes \cdots \otimes k[x_{n},y_{n}]/(x_{n}y_{n}-1).\]
Hence, by taking the spectrum it follows that
\[\Spec(k[M])\cong \G_{\mathrm{m},k}\times \G_{\mathrm{m},k} \times \cdots \times \G_{\mathrm{m},k}\cong \G_{\mathrm{m},k}^{n}.\]
This gives the variety structure.

The group structure of $\Spec(k[M])$ is induced by the \emph{Hopf algebra structure} given by the following morphisms of $k$-algebras:
	\begin{itemize}
		\item \textbf{Comultiplication:} 
			\begin{align*}
				k[M] &\to k[M]\otimes_{k} k[M] \\
				x_{i} &\mapsto x_{i}\otimes x_{i}.
			\end{align*}
		\item \textbf{Counit:} 
			\begin{align*}
				k[M] &\to k \\
				x_{i} &\mapsto 1.
			\end{align*}
		\item \textbf{Antipode:} 
			\begin{align*}
				k[M] &\to k[M] \\
				x_{i} &\mapsto y_{i}.
			\end{align*}
	\end{itemize}

Any split algebraic torus over $k$ arises this way. Let us introduce some notation: the \emph{group of characters} of a split torus $T$ is defined as  
\[\chi^{*}\left( T \right):=\{ \chi:T\to \G_{\mathrm{m},k} \mid \chi\textrm{ is a $k$-group homomorphism} \},\] 
which will be denoted as $M$, and its group of cocharacters is defined as
\[\chi_{*}\left( T \right):=\{ \lambda: \G_{\mathrm{m},k}\to T \mid \lambda\textrm{ is a $k$-group homomorphism} \},\] 
which will be denoted by $N$. Both, the group of characters and the \emph{group of cocharacters} of a split algebraic $k$-torus, are free $\Z$-modules of finite rank. Notice that if we compose $\chi\in M$ and $\lambda\in N$, we get a $k$-group morphism $\chi\circ\lambda:\G_{\mathrm{m},k}\to \G_{\mathrm{m},k}$. Given that $\End_{gr}(\G_{\mathrm{m},k})\cong \Z$, we have a map 
\begin{align*}
\langle,\rangle:M\times N &\to \Z, \\ 
(\chi,\lambda) &\mapsto \chi\circ\lambda,
\end{align*}
which defines a perfect pairing, as stated in the following result.

\begin{proposition}
Let $T$ be a split algebraic torus over $k$ of dimension $n$. Then,
\begin{enumerate}
	\item $M:=\chi^{*}(T)\cong \Z^{n}$,
	\item $N:=\chi_{*}(T)\cong \Hom_{\Z}(\chi^{*}(T),\Z)\cong \Z^{n}$ and
	\item $T\cong \Spec(k[M])$ as algebraic groups.
\end{enumerate}
\end{proposition}
 It is worth mentioning that the morphisms appearing in the latter proposition are canonical, up to those involving $\mathbb{Z}^{n}$.

\subsection{ Preliminaries on Convex geometry}
\indent

Let $N$ be a lattice and $N_{\Q}:=N\otimes_{\Z}\Q$ be the $\Q$-vector space associated to $N$ by scalar extension. Let $M:=\Hom_{\Z}(N,\Z)$ be the dual lattice of $N$, which has the same rank as $N$. The vector space $M_{\Q}$ is canonically isomorphic to $\Hom_{\Q}(N_{\Q},\Q)$, the dual of $N_{\Q}$ as a $\Q$-vector space. The lattices $N$ and $M$ can be considered contained in $N_{\Q}$ and $M_{\Q}$ respectively.

The natural morphism $\langle ,\rangle:M\times N\to \Z$, given by $\langle m,u\rangle:=m(u)$, defines a perfect pairing between $N$ and $M$.  This morphism extends to a perfect pairing $\langle,\rangle:M_{\Q}\times N_{\Q}\to \Q$.

\subsubsection{Cones and fans}\label{Section cones and fans}
\indent

The definition and results presented in this section can be found in \cite{Ful93} and \cite{CLS11}, for instance.

\begin{definition}
Let $N$ be a lattice and $M$ be its dual lattice. A \textit{convex polyhedral cone in $N_{\Q}$} is a subset $\omega$ of $N_{\Q}$ of the form 
\[ \omega=\cone(v_{1},\dots,v_{l})=\left\{ \sum_{i=1}^{l}r_{i}v_{i} \mid r_{i}\in \Q_{\geq 0}  \right\}, \]
for some $v_{1},\dots,v_{l}\in N_{\Q}$.
\end{definition}

Notice that convex polyhedral cones are closed convex sets. The \textit{dimension of $\omega$}, denoted by $\dim(\omega)$, is the dimension of the smallest subspace $V\subset N_{\Q}$ containing $\omega$. The \emph{dual cone} $\omega^{\vee}$ of $\omega$ is the cone 
\[\omega^{\vee}:=\{m\in M_{\Q} \mid \langle m,v \rangle \geq 0\,\, \textrm{for all }v\in \omega\},\]
which is also a convex polyhedral cone.

\begin{definition}
Let $N$ be a lattice and $M:=\Hom_{\Z}(N,\Z)$ be its dual lattice. A \textit{face of a convex polyhedral cone} $\omega\subset N_{\Q}$ is a subset $\tau$ of $\omega$ of the form 
\[ \tau= \omega\cap m^{\perp} = \{ v\in\omega  \mid \langle m,v \rangle =0 \} ,\]
with $m\in \omega^{\vee}\cap M_{\Q}$. The face relation is denoted by $\tau\preceq\omega$.
\end{definition}
Notice that for any convex polyhedral cone $\omega\subset N_{\Q}$ we have $\omega\preceq \omega$. A face $\tau$ of $\omega$ is called \textit{proper} when $\tau\neq \omega$. Every face of a convex polyhedral cone is a convex polyhedral cone and the intersection of two faces of a convex polyhedral cone is also a face. A further important
property is that the face relation is transitive.

\begin{definition}
Let $N$ be a lattice and $M:=\Hom_{\Z}(N,\Z)$ be its dual lattice. A polyhedral cone in $N_{\Q}$ is said to be \emph{pointed} if it contains no line.
\end{definition}

From now on, by a \emph{cone} in $N$ we mean a convex polyhedral cone in $N_{\Q}$.\footnote{In classical references, we mean \cite{Ful93} and \cite{CLS11}, we ask for rationality on the cones, but this is due to the definition being given over real vector spaces.}

\begin{definition}
Let $N$ be a lattice. A \textit{fan} in $N_{\Q}$ is a finite set $\Sigma$ of pointed cones in $N_{\Q}$ such that, for any $\omega\in\Sigma$, if $\tau\preceq\omega$ we have $\tau\in\Sigma$ and, for any pair $\omega,\omega'\in \Sigma$, the intersection $\omega\cap\omega'$ is in $\Sigma$ and $\omega\cap\omega'\preceq\omega,\omega'$. If the cones on $\Sigma$ are not necessarily pointed, then we say that $\Sigma$ is a \textit{quasifan}.
\end{definition}

\subsubsection{Polyhedra}\label{Section polyhedra}
\indent

This section is based in \cite[Section 1]{AH06}. A convex polyhedron in $N_{\Q}$ is the intersection of finitely many closed affine half spaces in $N_{\Q}$. In particular, a cone is a convex polyhedron. The set of all polyhedra in $N_{\Q}$ comes with a natural monoidal structure under the \textit{Minkowski sum}: for any pair of polyhedra $\Delta_{1}$ and $\Delta_{2}$ in $N_{\Q}$ 
\[\Delta_{1}+\Delta_{2}:=\{v_{1}+v_{2} \mid v_{i}\in\Delta_{i}\}.\]
 A \textit{polytope} $\Pi\subset N_{\Q}$ is the convex hull of finitely many points. Every polyhedron $\Delta$ in $N_{\Q}$ has a Minkowski decomposition $\Delta=\Pi+\omega$, with $\Pi$ a polytope in $N_{\Q}$ and $\omega$ a cone in $N_{\Q}$. This cone is called the \textit{tail cone} of $\Delta$, or \textit{recession cone} of $\Delta$, and is given by 
 \[\omega=\{v\in N_{\Q} \mid v'+tv\in\Delta \textrm{ for all }v'\in\Delta\textrm{ and }t\in\Q_{\geq 0}\}.\] 

\begin{definition}
    Let $\omega$ be a cone in $N_{\Q}$.
    \begin{enumerate}
        \item A $\omega$-\textit{tailed polyhedron} (or $\omega$-\textit{polyhedron} for short) in $N_{\Q}$, is a polyhedron $\Delta$ in $N_{\Q}$ having tail cone $\omega$. The set of all $\omega$-polyhedra in $N_{\Q}$ is denoted by $\mathrm{Pol}_{\omega}^{+}(N_{\Q})$.
        \item $\Delta\in\mathrm{Pol}_{\omega}^{+}(N_{\Q})$ is called \textit{integral} if $\Delta=\Pi+\omega$ holds with a polytope $\Pi\subset N_{\Q}$ having its vertices in $N$. The set of all integral $\omega$-polyhedra in $N_{\Q}$ in denoted by $\mathrm{Pol}_{\omega}^{+}(N)$. 
    \end{enumerate}
\end{definition}

The Minkowski sum of two $\omega$-polyhedra is also an $\omega$-polyhedron, then $\mathrm{Pol}_{\omega}^{+}(N_{\Q})$ is a monoid having $\omega\in\mathrm{Pol}_{\omega}^{+}(N_{\Q})$ as neutral element. This holds also for $\mathrm{Pol}_{\omega}^{+}(N_{})$, because the sum of two integral $\omega$-polyhedra is an integral $\omega$-polyhedron. Denote by $\mathrm{Pol}_{\omega}(N_{\Q})$ and $\mathrm{Pol}_{\omega}^{}(N_{})$ their respective Grothendieck groups.

Recall that the \textit{support function} associated to a convex set $\Delta\subset N_{\Q}$ is given by 
\begin{align*}
h_{\Delta}:M_{\Q} &\to \Q\cup\{-\infty\}, \\
m &\mapsto \inf_{v\in\Delta}\langle m,v \rangle 
\end{align*} 
and its support is $\mathrm{Supp}(h_{\Delta}):=\{m\in M_{\Q} \mid h_{\Delta}(m)>-\infty \}$. For an $\omega$-polyhedron $\Delta$ and $m\in M_{\Q}$, we define
\[\lambda_{m}:=\{m'\in M_{\Q}\mid h_{\Delta}(m+m')=h_{\Delta}(m)+h_{\Delta}(m')\},\]
which is a cone. The set $\lambda_{\Delta}:=\{\lambda_{m} \mid m\in M_{\Q}\}$ is finite. Define $\Lambda(\Delta)$ as the set generated by all the finite intersections of elements in $\lambda_{\Delta}$. Each element in $\Lambda(\Delta)$ is a cone, not necessarily pointed. The set $\Lambda(\Delta)$ is called \textit{the normal quasifan of $\Delta$}.

In the following we present some properties that can be found in \cite[Section 1]{AH06}.

\begin{lemma}\label{Lemma 1.4 AH06}
Let $\omega\in N_{\Q}$ a pointed cone, $\Delta\in\mathrm{Pol}_{\omega}^{+}(N_{\Q})$ and $h_{\Delta}:M_{\Q}\to \Q\cup\{-\infty\}$ its respective support function. Then, the following hold.
\begin{enumerate}[i)]
	\item The support of $h_{\Delta}$ is $\omega^{\vee}$ and $h_{\Delta}$ is linear on each cone of the normal quasifan $\Lambda(\Delta)$.
	\item The function $h_{\Delta}$ is convex, i.e. for every $m_{1}$ and $m_{2}$ in $M_{\Q}$ we have 
	\[h_{\Delta}(m_{1}+m_{2})\leq h_{\Delta}(m_{1})+h_{\Delta}(m_{2}).\]
	Moreover, the strict inequality holds if and only if  $m_{1}$ and $m_{2}$ do not belong to the same maximal cone of $\Lambda(\Delta)$.
\end{enumerate}
\end{lemma}

A function $h:M_{\Q}\to \Q\cup\{-\infty\}$, having $\omega^{\vee}$ as support, is said to be \textit{piecewise linear} if there exists a quasifan $\Lambda$ with support $\omega^{\vee}$ such that $h$ is linear on each $\lambda\in\Lambda$. Denote $\mathrm{CPL}_{\Q}(\omega)$ the set of convex piecewise linear functions $h:M_{\Q}\to \Q\cup\{-\infty\}$ having $\omega^{\vee}$ as its support. For any $\Delta\in \mathrm{Pol}_{\omega}^{+}(N_{\Q})$, the support function $h_{\Delta}$ is piecewise linear and has $\omega^{\vee}$ as its support. The assignment $\Delta\mapsto h_{\Delta}$ defines a map $\mathrm{Pol}_{\omega}^{+}(N_{\Q}) \to \mathrm{CPL}_{\Q}(\omega)$, which turns to be a monoid isomorphism.

\begin{proposition}\label{Proposition 1.5 AH06}
Let $\omega\subset N_{\Q}$ a pointed cone. The set $\mathrm{CPL}_{\Q}(\omega)$ is a monoid and the map
\begin{align*}
\mathrm{Pol}_{\omega}^{+}(N_{\Q}) &\to \mathrm{CPL}_{\Q}(\omega), \\ 
\Delta &\mapsto h_{\Delta}
\end{align*}
is a monoid isomorphism.
\end{proposition}

\begin{proposition}\label{Proposition 1.7 AH06}
Let $\omega\in N_{\Q}$ a pointed cone. Then, the following statements hold.
\begin{enumerate}[i)]
	\item\label{Proposition 1.7 AH06 part i} There is a commutative diagram of canonical injective homomorphisms of monoids
	\[\xymatrix{ \mathrm{Pol}_{\omega}^{+}(N) \ar[r] \ar[d] & \mathrm{Pol}_{\omega}^{+}(N_{\Q}) \ar[d] \\ \mathrm{Pol}_{\omega}^{}(N) \ar[r] & \mathrm{Pol}_{\omega}^{}(N_{\Q}). }\]
	\item\label{Proposition 1.7 AH06 part ii} The multiplication of an element $\Delta\in\mathrm{Pol}_{\omega}^{+}(N_{\Q})$ by a positive rational number $\alpha\in\Q^{+}$ is defined as 
	\[\alpha\cdot \Delta:=\{\alpha v\mid v\in\Delta\}\]
	extends to a unique $\Q$-action on $\mathrm{Pol}_{\omega}(N_{\Q})$.
	\item\label{Proposition 1.7 AH06 part iii} The group $\mathrm{Pol}_{\omega}(N)$ of integral $\omega$-polyhedra is a free abelian group and we have a canonical isomorphism 
	\[\mathrm{Pol}_{\omega}^{}(N_{\Q})\cong \Q\otimes_{\Z}\mathrm{Pol}_{\omega}^{}(N).\]
	\item\label{Proposition 1.7 AH06 part iv} For every element $m\in \omega^{\vee}$, there is a unique linear evaluation functional $\mathrm{eval}_{m}:\mathrm{Pol}_{\omega}^{}(N_{\Q})\to \Q$ satisfying 
	\[ \mathrm{eval}_{m}(\Delta)=\min_{v\in\Delta}\langle m,v \rangle, \]
	for $\Delta\in\mathrm{Pol}_{\omega}^{+}(N)$.
	\item\label{Proposition 1.7 AH06 part v} Two elements $\Delta_{1}$ and $\Delta_{2}$ in $\mathrm{Pol}_{\omega}^{}(N_{\Q})$ coincide if and only if $\mathrm{eval}_{m}(\Delta_{1})=\mathrm{eval}_{m}(\Delta_{2})$ holds for every $m\in\omega^{\vee}$.
	\item\label{Proposition 1.7 AH06 part vi} An element $\Delta\in \mathrm{Pol}_{\omega}^{}(N_{\Q})$ is integral if and only if $\mathrm{eval}_{m}(\Delta)\in\Z$ for every $m\in\omega^{\vee}\cap M$.
\end{enumerate}
\end{proposition}

\section{The category of proper polyhedral divisors}\label{section Polyhedral Divisors}
\indent

It is known that split affine toric varieties over $k$ arise from cones in $N_{\Q}$. The main goal of this section is to present the objects that generalize cones for any affine normal variety over $k$ endowed with an effective action of a split algebraic torus over $k$. These objects were introduced by Altmann and Hausen \cite{AH06} for algebraically closed fields of characteristic zero. However, the definitions work over any field.

\subsection{ Proper polyhedral divisors}
\indent

Let $N$ be a lattice and $\omega\subset N_{\Q}$ be a pointed cone. As stated in \cref{Section polyhedra}, the set of all $\omega$-tailed polyhedra $\mathrm{Pol}_{\omega}^{+}(N_{\Q})$ is a monoid, whose neutral element is $\omega$. The same holds for the set of integral $\omega$-tailed polyhedra $\mathrm{Pol}_{\omega}^{+}(N)\subset \mathrm{Pol}_{\omega}^{+}(N_{\Q})$. Moreover, both admit the construction of a Grothendieck group, denoted by $\mathrm{Pol}_{\omega}(N_{\Q})$ and $\mathrm{Pol}_{\omega}(N)$ respectively. These groups are abelian. 

Let $Y$ be a variety over $k$. Given that $\mathrm{Pol}_{\omega}(N_{\Q})$ and $\mathrm{Pol}_{\omega}(N)$ are abelian groups, we can take the tensor products 
\[\mathrm{Pol}_{\omega}(N_{\Q})\otimes_{\Z}\cadiv(Y)\quad\mathrm{ and }\quad\mathrm{Pol}_{\omega}(N_{\Q})\otimes_{\Z}\cadiv(Y),\]
where $\cadiv(Y)$ stands for the group of Cartier divisors of $Y$. 
 Besides, if $Y$ is normal, we can also consider $\mathrm{Pol}_{\omega}(N_{\Q})\otimes_{\Z}\Div(Y)$ and $\mathrm{Pol}_{\omega}(N_{\Q})\otimes_{\Z}\Div(Y)$, where $\Div(Y)$ stands for the group of Weil divisors of $Y$. These groups are called the group of \textit{rational (resp. integral) polyhedral Cartier divisors} and the group of \textit{rational (resp. integral) Weil divisors}.

\begin{definition}
    Let $Y$ be a normal variety over $k$, $N$ be a lattice and $\omega\subset N_{\Q}$ be a pointed cone:
    \begin{enumerate}
        \item The group of \textit{rational polyhedral Weil divisors} and \textit{rational polyhedral Cartier divisors} of $Y$ with respect to $\omega\subset N_{\Q}$ are 
        \[\Div_{\Q}(Y,\omega):=\mathrm{Pol}_{\omega}(N_{\Q})\otimes_{\Z}\Div(Y),\] \[\cadiv_{\Q}(Y,\omega):=\mathrm{Pol}_{\omega}(N_{\Q})\otimes_{\Z}\cadiv(Y).\]
        \item The group of \textit{integral polyhedral Weil divisors} and \textit{integral polyhedral Cartier divisors} of $Y$ with respect to $\omega\subset N_{\Q}$ are \[\Div_{}(Y,\omega):=\mathrm{Pol}_{\omega}(N_{})\otimes_{\Z}\Div(Y),\] \[\cadiv_{}(Y,\omega):=\mathrm{Pol}_{\omega}(N_{})\otimes_{\Z}\cadiv(Y).\]
    \end{enumerate}
\end{definition}

Recall that, for a normal variety $Y$ over $k$ there is a canonical embedding 
\[\cadiv(Y)\to \mathrm{Div}(Y),\] 
which allows us to consider $\cadiv(Y)\subset\mathrm{Div}(Y)$ and, therefore, 
\[\cadiv_{\Q}(Y,\omega)\subset\Div_{\Q}(Y,\omega)\] 
for any cone $\omega\subset N_{\Q}$. Therefore, we can ask $D\in\cadiv(Y)$ to be prime. This being said, note that we can always write an element in any of these groups as $\D=\sum_D \Delta_{D}\otimes D$, where the sum runs through the prime divisors $D$ of $Y$ and the $\Delta_{D}$'s are elements in $\mathrm{Pol}_{\omega}(N_{})$ or $\mathrm{Pol}_{\omega}(N_{\Q})$.\\

We are now ready to introduce the objects of the category of proper polyhedral divisors. In the following, by a \textit{polyhedral divisor} we mean a rational one. 

The sheaf of sections $\mathscr{O}_{Y}(D)$ of a rational Weil divisor $D$ on a normal variety $Y$ over $k$ is defined as, for every open $U\subset Y$,
\[\mathscr{O}_{Y}(D)(U):=\{f\in k(Y) \mid \mathrm{div}(f|_{U})+\lfloor D\rfloor |_{U}\geq 0\},\]
where $\lfloor D\rfloor$ is the \emph{round down} of $D$. The next definition corresponds to \cite[Definition 2.2]{AH06}.

\begin{definition}\label{definitionppdiv}
    Let $Y$ be a normal $k$-variety, $N$ be a lattice and $\omega\subset N_{\Q}$ a pointed cone. A polyhedral divisor $\D=\sum_D \Delta_{D}\otimes D\in\mathrm{CaDiv}_{\Q}(Y,\omega)$ is called \textit{proper} if
    \begin{enumerate}
        \item all the $D\in\Div(Y)$ are prime divisors and the $\Delta_{D}$ are in $\mathrm{Pol}_{\omega}^{+}(N_{\Q})$;
        \item for every $m\in\mathrm{relint}(\omega^{\vee})\cap M$, i.e. $m$ is not contained in a proper face of $\omega^{\vee}$, the evaluation 
        \[\D(m):=\sum h_{\Delta_{D}}(m)D\in\mathrm{CaDiv}_{\Q}(Y)\] is a big divisor on $Y$, i.e. if for some $n\in\N$ there exists a section 
        \[f\in\H^{0}(Y,\mathscr{O}_{Y}(n\D(m)))\]
         such that $Y_{f}$ is affine; 
        \item for every $m\in\omega^{\vee}\cap M$, the evaluation $\D(m)\in\cadiv_{\Q}(Y)$ is semiample, i.e. it admits a base-point-free positive multiple. Otherwise stated, for some $n\in\N$ the sets $Y_{f}$ cover $Y$, where $f\in\H^{0}(Y,\mathscr{O}_{Y}(n\D(m)))$.
    \end{enumerate}
\end{definition}
    
    \begin{remark}
    	Over a projective variety, the definition of big divisor given in \cref{definitionppdiv} turns to be equivalent to the following one: the divisor $\D(m)$ is big if for some $n\in\N$ $n\D(m)$ is a Cartier divisor and its respective line bundle $\mathscr{O}_{Y}(n\D(m))$ is of maximal Iitaka dimension (see \cite[Lemma 2.60]{KM08}).  
    \end{remark}
    
   Let $Y$ be a normal variety over $k$ and $\omega$ be a cone. The set of proper polyhedral divisors (pp-divisors for short) over $Y$ and with respect to $\omega$ has a monoidal structure having $\D_{\omega}:=\sum_{D} \omega\otimes D$ as neutral element. This monoid is denoted by $\PPDiv_{\Q}(Y,\omega)$. The monoid is partially ordered as follows: if $\D=\sum_D \Delta_{D}\otimes D$ and $\D'=\sum_D \Delta_{D}'\otimes D$, then $\D'\leq \D$ if and only if $\Delta_{D}\subset \Delta_{D}'$ for every $D$.

\begin{definition}
Let $Y$ be a normal variety over $k$. Let $\D\in\PPDiv_{\Q}(Y,\omega)$ be a pp-divisor. For $y\in Y$, we define the \textit{fiber polyhedron at} $y$ as 
        \[\Delta_{y}:=\sum_{y\in D}\Delta_{D}.\]
\end{definition}

As mentioned before, a pp-divisor $\D\in\PPDiv_{\Q}(Y,\omega)$ defines a map $\mathfrak{h}_{\D}:\omega^{\vee}\to \cadiv_{\Q}(Y)$ given by $\mathfrak{h}_{\D}(m):=\D(m)$. This map satisfies certain properties summarized in the following definition.

\begin{definition}
    Let $Y$ be a normal $k$-variety; let $M$ be a lattice, and let $\omega^{\vee}\subset M_{\Q}$ be a cone of full dimension. We say that a map $\mathfrak{h}:\omega^{\vee}\to \cadiv_{\Q}(Y)$ is 
    \begin{enumerate}[i)]
        \item \textit{convex} if $\mathfrak{h}(m)+\mathfrak{h}(m')\leq \mathfrak{h}(m+m')$ holds for any two elements $m,m'\in\omega^{\vee}$,
        \item \textit{piecewise linear} if there is a quasifan $\Lambda$ in $M_{\Q}$ having $\omega^{\vee}$ as its support such that $\mathfrak{h}$ is linear on the cones of $\Lambda$,
        \item \textit{strictly semiample} if $\mathfrak{h}(m)$ is semiample for all $m\in\omega^{\vee}$ and if for all $m\in\mathrm{relint}(\omega^{\vee})$ is big.
    \end{enumerate}
    The set of all convex, piecewise linear and strictly semiample maps $\mathfrak{h}:\omega^{\vee}\to \cadiv_{\Q}(Y)$ is denoted by $\mathrm{CPL}_{\Q}(Y,\omega)$. 
\end{definition}

For each $\D\in\PPDiv_{\Q}(Y,\omega)$ we can associate with it a convex, piecewise linear and strictly semiample map $\mathfrak{h}_{\D}\in\mathrm{CPL}_{\Q}(Y,\omega)$. Thus, we have a natural map 
\begin{align*}
        \PPDiv_{\Q}(Y,\omega) &\to\mathrm{CPL}_{\Q}(Y,\omega), \\ \D &\mapsto \mathfrak{h}_{\D}.
    \end{align*}

The following result corresponds to \cite[Proposition 2.11]{AH06}, which holds over any field. The proof adapts straightforwardly to this context.

\begin{proposition}\label{proposition ppdiv cplmaps}
    Let $Y$ be a normal $k$-variety, $N$ be a lattice, and $\omega\subset N_{\Q}$ be a pointed cone. Then the set $\mathrm{CPL}_{\Q}(Y,\omega)$ is a monoid and the canonical map
    $\PPDiv_{\Q}(Y,\omega) \to\mathrm{CPL}_{\Q}(Y,\omega)$ given by $ \D \mapsto \mathfrak{h}_{\D}$ is an isomorphism. Moreover, the integral polyhedral divisors correspond to the maps $\mathfrak{h}:\omega^{\vee}\to \cadiv_{\Q}(Y)$ such that $\mathfrak{h}(\omega^{\vee}\cap M)\subset\cadiv(Y)$.
\end{proposition}

\subsection{ Morphisms of proper polyhedral divisors}
\indent

In the previous section, we have introduced algebro-geometric combinatorial objects called pp-divisors. In order to construct a category, we need to expose how the objects are related. The morphisms are given by three pieces of data. Among them, there is one called \textit{plurifunction}, whose definition is given below.

\begin{definition}\cite[Definition 8.2]{AH06}\label{defpluri}
    Let $Y$ be a normal $k$-variety, $N$ be a lattice and $\omega\subset N_{\Q}$ be a pointed cone.
    \begin{enumerate}[a)]
        \item A \textit{plurifunction} with respect to the lattice $N$ is an element of 
        \[k(Y,N)^{*}:=N\otimes_{\Z}k(Y)^{*}.\]
        \item \label{defpluri part b} For $m\in M:=\Hom(N,\Z)$, the \textit{evaluation} of a plurifunction $\mathfrak{f}=\sum v_{i}\otimes f_{i}$ with respect to $N$ is 
        \[\mathfrak{f}(m):=\prod f_{i}^{\langle m,v_{i} \rangle}\in k(Y)^{*}.\]
        \item The \textit{polyhedral principal divisor} with respect to $\omega\subset N_{\Q}$ of a plurifunction $\mathfrak{f}=\sum v_{i}\otimes f_{i}$ with respect to $N$ is 
        \[\mathrm{div}(\mathfrak{f}):=\sum(v_{i}+\omega)\otimes\mathrm{div}(f_{i})\in\cadiv(Y,\omega).\]
    \end{enumerate}
\end{definition}

\begin{remark}\label{remark inverse of polyhedral principal divisor}
Notice that the map $k(N,Y)^{*}\to \cadiv_{}(Y,\omega)$, given by $\mathfrak{f}\mapsto\mathrm{div}(\mathfrak{f})$, is a group homomorphism. For a plurifunction $\mathfrak{f}:=\sum v_{i}\otimes f_{i}$, the inverse of $\mathrm{div}(\mathfrak{f})$ is $\mathrm{div}(\sum -v_{i}\otimes f_{i})$.
\end{remark}

A homomorphism $F:N\to N'$ induces a homomorphism $F_{*}:k(N,Y)^{*}\to k(N',Y)^{*}$ given by 
\[F_{*}\left(\sum v_{i}\otimes f_{i}\right):=\sum F(v_{i})\otimes f_{i}.\] 
A morphism $\psi:Y\to Y'$ induces a morphism $\psi^{*}:k(N,Y')^{*}\to k(N,Y)^{*}$ given by 
\[\psi^{*}\left(\sum v_{i}\otimes f_{i}\right):=\sum v_{i}\otimes \psi^{*}(f_{i}).\]

\begin{remark}\label{remark: isomorphisms plurifunction}
The assignment $\mathfrak{f} := \sum v_{i}\otimes f_{i}\mapsto \prod f_{i}^{\langle v_{i},\bullet \rangle}$ defines an isomorphism $k(Y,N)^{*}\cong \Hom_{\Z}(M,k(Y)^{*})$. Thus, a plurifunction can be understood as a homomorphism $\mathfrak{f}:M\to k(Y)^{*}$. We have then the former homomorphism that sends $F$ to $F_{*}$ and induces a homomorphism $\Hom_{\Z}(N,N')\to \Hom_{\Z}(k(Y,N)^{*},k(Y,N')^{*})$.
\end{remark}

Recall that $\PPDiv_{\Q}(Y,\omega)$ is a partially ordered monoid with $\D'\leq \D$ if and only if $\Delta_{D}\subset \Delta_{D}'$ for every $D$.

\begin{definition}\label{def83}\cite[Definition 8.3]{AH06}
    Let $Y$ and $Y'$ be normal $k$-varieties, $N$ and $N'$ be lattices and $\omega\subset N$ and $\omega'\subset N'$ be pointed cones. Let us consider \[\D=\sum \Delta_{i}\otimes D_{i}\in\PPDiv_{\Q}(Y,\omega)\, \textrm{ and }\, \D'=\sum \Delta_{i}'\otimes D_{i}'\in\PPDiv_{\Q}(Y',\omega')\] two pp-divisors.
    \begin{enumerate}[a)]
        \item For morphisms $\psi:Y\to Y'$ such that none of the supports $\mathrm{Supp}(D_{i}')$ contains $\psi(Y)$, we define the (not necessarily proper) $\textit{polyhedral pullback}$ as \[\psi^{*}(\D'):=\sum \Delta_{i}'\otimes \psi^{*}(D_{i}')\in\cadiv_{\Q}(Y,\omega').\]
        \item\label{def83b} For linear maps $F:N\to N'$ with $F(\omega)\subset \omega'$, we define the (not necessarily proper) \textit{polyhedral pushforward} as \[F_{*}(\D):=\sum (F(\Delta_{i})+\omega')\otimes D_{i}'\in \cadiv_{\Q}(Y,\omega').\]
        \item \label{defmappdiv} A \emph{morphism of pp-divisors} is a triple $(\psi,F,\mathfrak{f}):\D\to \D'$, where $\psi:Y\to Y'$ is a \emph{dominant} morphism of varieties, $F$ is a linear map as in \ref{def83b}) and $\mathfrak{f}\in k(Y,N')^{*}$ is a plurifunction, such that \[\psi^{*}(\D')\leq F_{*}(\D)+\divr(\mathfrak{f}).\]
    \end{enumerate}
\end{definition}

 The identity map $\D\to\D$ for a pp-divisor is the triple $(\id,\id_{N},1)$. The composition of two morphisms of pp-divisors $(\psi,F,\mathfrak{f})$ and $(\psi',F',\mathfrak{f}')$ is defined as \[(\psi',F',\mathfrak{f}')\circ(\psi,F,\mathfrak{f})=(\psi'\circ\psi,F'\circ F,F_{*}'(\mathfrak{f})\cdot \psi^{*}(\mathfrak{f}')).\]
 The composition of two morphisms of pp-divisors is a morphism of pp-divisors. Thus, we have the following result.
 
 \begin{proposition}
 The proper polyhedral divisors over normal semi-projective $k$-varieties with the morphisms of pp-divisors form a category denoted by $\mathfrak{PPDiv}$. 
 \end{proposition}
 
 Recall that every proper polyhedral divisor $\D$ in $\mathfrak{PPDiv}$ has a tail cone defined on some $N_{\Q}$, with $N$ a lattice. Furthermore, by fixing a lattice we are fixing a split $k$-torus, as stated in \cref{Section algebraic torus}.
 
 \begin{definition}
 Let $N$ be a lattice. We denote by $\mathfrak{PPDiv}_{N}$ the full subcategory of $\mathfrak{PPDiv}$ whose objects are the proper polyhedral divisors $\D$ such that $\mathrm{Tail}(\D)$ is defined on $N_{\Q}$.
 \end{definition}
 
 \subsection{ Base change for proper polyhedral divisors}\label{Section base change}
 \indent
 
Let $Y$ be a normal variety over $k$. Recall that there is a canonical map $\Div(Y)\to \Div(Y_{\sep{k}})$, which induces a canonical map 
\begin{align*}
    \cadiv_{\Q}(Y,\omega) &\to \cadiv_{\Q}(Y_{\sep{k}},\omega); \\ \D=\sum\Delta_{D}\otimes D &\mapsto \D_{\sep{k}}:=\sum\Delta_{D}\otimes D_{\sep{k}}. 
\end{align*}
Since every divisor $D$ in the left-hand sum is a prime divisor, it follows that the divisor $D_{\sep{k}}$ is a multiplicity-free sum of prime divisors.

This map turns out to be a group monomorphism. In particular, every pp-divisor on $Y$ induces a rational polyhedral divisor on $Y_{\sep{k}}$, which is a pp-divisor.

\begin{lemma}\label{Lemma pp-divisors stable under base change}
    Let $N$ be a lattice, $\omega\subset N_{\Q}$ be a pointed cone, $Y$ be a normal variety over $k$. If $\D\in \PPDiv_{\Q}(Y,\omega)$, then $\D_{\sep{k}}\in \PPDiv_{\Q}(Y_{\sep{k}},\omega)$.
\end{lemma}

\begin{proof}
    Let $\D\in \PPDiv_{\Q}(Y,\omega)$ with $\D=\sum\Delta_{D}\otimes D$ and $\D_{\sep{k}}=\sum\Delta_{D}\otimes D_{\sep{k}}\in\cadiv_{\Q}(Y_{\sep{k}},\omega)$ as above.
    
    Let $m\in\omega^{\vee}\cap M$ and $n\in \N$. The morphisms $Y_{\sep{k}}\to Y$ and $\cadiv_{\Q}(Y,\omega)\to \cadiv_{\Q}(Y_{\sep{k}},\omega)$ define a morphism 
\[\varphi_{n}:\H^{0}(Y,\mathscr{O}(n\D(m)))\to\H^{0}(Y_{\sep{k}},\mathscr{O}(n\D_{\sep{k}}(m))).\]
 This implies that $\D_{\sep{k}}(m)$ is semiample, because $\D(m)$ is semiample. Indeed, there exists $n\in \N$ such that $Y_{f}$ cover $Y$ where $f\in\H^{0}(Y,\mathscr{O}(n\D(m)))$. Thus, the $(Y_{\sep{k}})_{\varphi_{n}(f)}$ cover $Y_{\sep{k}}$. Therefore, the $(Y_{\sep{k}})_{f}$ cover $Y_{\sep{k}}$ for $f\in\H^{0}(Y_{\sep{k}},\mathscr{O}(n\D_{\sep{k}}(m)))$. Hence, $\D_{\sep{k}}(m)$ is semiample for $m\in\omega^{\vee}\cap M$.
    
    If $m\in\relint(\omega^{\vee})$, by definition $\D(m)$ is big. Then, for some $n\in\N$ there exists a section $f\in\H^{0}(Y,\mathscr{O}(n\D(m)))$ such that $Y_{f}$ is affine. Given that $(Y_{\sep{k}})_{\varphi_{n}(f)}=(Y_{f})_{\sep{k}}$, we have that $\varphi_{n}(f)$ has an affine non-vanishing locus. Hence, $\D_{\sep{k}}(m)$ is big for every $m\in\relint(\omega^{\vee})$. This proves that $\D_{\sep{k}}\in\PPDiv_{\Q}(Y_{\sep{k}},\omega)$.
\end{proof}

The group homomorphism $\cadiv_{\Q}(Y,\omega) \to \cadiv_{\Q}(Y_{\sep{k}},\omega)$ induces a monoid homomorphism
\[\PPDiv_{\Q}(Y,\omega) \to \PPDiv_{\Q}(Y_{\sep{k}},\omega).\]
Clearly, unless $k$ is separably closed, this map is not surjective.

First, given that $\Div(Y_{\sep{k}})$ has a natural action of $\Gamma:=\Gal(\sep{k}/k)$, then the monoid $\PPDiv_{\Q}(Y_{\sep{k}},\omega)$ has a natural structure of continuous $\Gamma$-module. Therefore, the image of $\PPDiv_{\Q}(Y,\omega) \to \PPDiv_{\Q}(Y_{\sep{k}},\omega)$ lies in $\PPDiv_{\Q}(Y_{\sep{k}},\omega)^{\Gamma}$. Actually, the image of $\PPDiv_{\Q}(Y,\omega) \to \PPDiv_{\Q}(Y_{\sep{k}},\omega)$ lies in $\PPDiv_{\Q}(Y_{\sep{k}},\omega)^{\Gamma}$. Before proving this claim, we prove the following lemma.

\begin{lemma}\label{lemma: descent of semiampleness}
	Let $Y$ be a variety over $k$ and $D\in \cadiv(Y)$. Then, $D_{\sep{k}}$ is semiample if and only if $D$ is semiample.
\end{lemma}

\begin{proof}
Assume first that $D$ is semiample. Then, there exists $n\in \N$ such that $\mathscr{O}_{Y}(nD)$ is generated by its global sections. This directly implies that $\mathscr{O}_{Y_{\sep{k}}}(D_{\sep{k}})$ is generated by its global sections and, consequently, $D_{\sep{k}}$ is semiample.

Assume now that $D_{\sep{k}}$ is semiample. First, notice that there exists a finite Galois extension $L$ of degree $d$ such that $D_{L}$ is semiample. Denote by $\Gamma_{L}$ the corresponding Galois group and by $\pi_{L}:Y_{L}\to Y$ the base change map. Let $y\in Y_{L}$ be a closed point. Since $D_{L}$ is semiample, there exists $n\in\N$ such that $\mathscr{O}_{Y_{L}}(nD_{L})$ is generated by its global sections. It is possible to find a section $f\in \H^{0}(Y_{L},\mathscr{O}_{Y_{L}}(n D_{L}))$ that does not vanish on the Galois orbit of $y$. Hence, $\Gamma_{L}(f):=\prod_{\gamma\in\Gamma_{L}} \gamma(f)\in\H^{0}(Y_{L},\mathscr{O}_{Y_{L}}(dn D_{L}))$ is a $\Gamma_{L}$-stable section such that $y\in (Y_{L})_{\Gamma_{L}(f)}$. Thus, it descends to a section $g\in\H^{0}(Y,\mathscr{O}_{Y}(dn D))$ such that $\pi_{L}(y)\in Y_{g}$. This holds for every $y\in Y_{L}$. Consequently, it follows that $\mathscr{O}_{Y}(dnD)$ is generated by its global sections, since $\pi_{L}$ is finite and surjective. Thus, the assertion holds.
\end{proof}

\begin{proposition}
Let $Y$ be a normal variety over $k$. Let $N$ be a lattice and $\omega\subset N_{\Q}$ be a pointed cone. Then, the image of $\PPDiv_{\Q}(Y,\omega) \to \PPDiv_{\Q}(Y_{\sep{k}},\omega)$ is $\PPDiv_{\Q}(Y_{\sep{k}},\omega)^{\Gamma}$.
\end{proposition}

\begin{proof}
Clearly, the image of $\PPDiv_{\Q}(Y,\omega) \to \PPDiv_{\Q}(Y_{\sep{k}},\omega)$ is contained in the set of $\Gamma$-stable pp-divisors $\PPDiv_{\Q}(Y_{\sep{k}},\omega)^{\Gamma}$. Let us prove the other inclusion. Let 
\[\tilde{\D}:=\sum \Delta_{\tilde{D}}\otimes \tilde{D}\]
 in $\PPDiv_{\Q}(Y_{\sep{k}},\omega)^{\Gamma}$. Given that the pp-divisor is Galois invariant, we have that $\Delta_{\tilde{D}}=\Delta_{\gamma(\tilde{D})}$ for every $\tilde{D}$ appearing in $\D$ and $\gamma\in\Gamma$. Therefore, for each $\tilde{D}$ appearing in $\tilde{\D}$, we have that 
 \[Z'_{\tilde{D}}:=\bigcup_{\Delta_{\tilde{D}}=\Delta_{\tilde{D}'}} \mathrm{supp}(\tilde{D}')\]
 is a Galois stable closed subvariety of $Y_{\sep{k}}$. Therefore, it descends to a closed subvariety $Z_{\tilde{D}}\subset Y$. Thus, by taking the irreducible components of $Z_{\tilde{D}}$ for every $\tilde{D}$, we can construct a polyhedral divisor 
 \[\D:=\sum \Delta_{D}\otimes D\in\cadiv_{\Q}(Y,\omega)\]
 such that $\Delta_{D}=\Delta_{\tilde{D}}$ when $\mathrm{supp}(D)\subset Z_{\tilde{D}}$. In order to prove that $\D$ is a pp-divisor, we need to prove that the $\D(m)$ is semiample for every $m\in\omega^{\vee}\cap M$ and big for $m\in\relint(\omega^{\vee})\cap M$. First notice that $\D_{\sep{k}}(m)=\tilde{\D}(m)$ and recall that the morphism $Y_{\sep{k}}\to Y$ induces morphisms 
 \[\varphi_{n}:\H^{0}(Y,\mathscr{O}(n\D(m)))\to\H^{0}(Y_{\sep{k}},\mathscr{O}(n\D_{\sep{k}}(m))),\] 
 for every $n\in\N$. 
 
Given that $\D_{\sep{k}}(m)$ is big, for $m\in\relint(\omega^{\vee})\cap M$, there exist $n\in\N$ and $f\in\H^{0}(Y_{\sep{k}},\mathscr{O}(n\D_{\sep{k}}(m)))$ such that $(Y_{\sep{k}})_{f}$ is affine. The Galois group $\Gamma$ acts on the global sections $\H^{0}(Y_{\sep{k}},\mathscr{O}(n\D_{\sep{k}}(m)))$, because the divisor is Galois stable. Hence, we can consider the orbit of $f$ in $\H^{0}(Y_{\sep{k}},\mathscr{O}(n\D_{\sep{k}}(m)))$, which is finite. Denote by $\prod_{\Gamma}(f):=f_{1}\cdots f_{l}$, the product of the elements in the orbit of $f$. Thus, for $n'=l\cdot n$, we have that $\prod_{\Gamma}(f)\in\H^{0}(Y_{\sep{k}},\mathscr{O}(n'\D_{\sep{k}}(m)))$. Given that $\prod_{\Gamma}(f)$ is Galois stable, there exists $g\in\H^{0}(Y,\mathscr{O}(n'\D(m)))$ such that $\varphi_{n'}(g)=\prod_{\Gamma}(f)$. We claim that $Y_{g}$ is affine. On the one hand, for every $i\in\{1,\dots, l\}$, there exists $\gamma_{i}\in\Gamma$ such that $\gamma_{i}((Y_{\sep{k}})_{f})=(Y_{\sep{k}})_{f_{i}}$. This implies that each $(Y_{\sep{k}})_{f_{i}}$ is affine. Thus, the non-zero locus of $\prod_{\Gamma}(f)$ is affine because it is the intersection of finitely many affine open subvarieties over $\sep{k}$
 \[(Y_{\sep{k}})_{\prod_{\Gamma}(f)}= \bigcap_{i=1}^{l}(Y_{\sep{k}})_{f_{i}}.\]
 On the other hand, $(Y_{g})_{\sep{k}}=(Y_{\sep{k}})_{\prod_{\Gamma}(f)}$ is affine. Then, $Y_{g}$ is affine. This implies that $\D(m)$ is big for every $m\in\relint(\omega^{\vee})\cap M$. 
 
Finally, semiampleness of each $\D(m)$ follows from \cref{lemma: descent of semiampleness}, since $n\in \N$ can be chosen big enough such that $\D(nm)\in\cadiv(Y)$. This proves the assertion.
 \end{proof}
     
Denote by $\mathfrak{PPDiv}(k)$ (resp. $\mathfrak{PPDiv}(\sep{k})$) the category of pp-divisors over $k$ (resp. $\sep{k}$). Let $\D$ and $\D'$ be objects in $\mathfrak{PPDiv}(k)$ and $(\psi,F,\mathfrak{f}):\D'\to\D$ a morphism in $\mathfrak{PPDiv}(k)$. By base change we have a morphism of pp-divisors $(\psi_{\sep{k}},F,\mathfrak{f}_{\sep{k}}):\D_{\sep{k}}'\to\D_{\sep{k}}$ in $\mathfrak{PPDiv}(\sep{k})$. This construction is compatible with the composition law defined above. Thus, this data and the one given by $\D\mapsto\D_{\sep{k}}$ define a covariant functor $\mathfrak{PPDiv}(k)\to \mathfrak{PPDiv}(\sep{k})$. 
   
\begin{proposition}\label{proposition covariant base change}
The functor $\mathfrak{PPDiv}(k)\to\mathfrak{PPDiv}(\sep{k})$ is faithful.
\end{proposition}

\begin{proof}
Let $\D$ and $\D'$ be objects in $\mathfrak{PPDiv}_{}(k)$. Let $(\psi,F,\mathfrak{f})$ and $(\psi',F',\mathfrak{f}')$ be morphisms in $\mathrm{Mor}_{\mathfrak{PPDiv}_{}(k)}(\D',\D)$ such that $(\psi_{\sep{k}},F_{\sep{k}},\mathfrak{f}_{\sep{k}})=(\psi'_{\sep{k}},F'_{\sep{k}},\mathfrak{f}'_{\sep{k}})$. After the base change, we have $F=F_{\sep{k}}$ and $F'=F'_{\sep{k}}$. Then $F=F'$. Given that $\psi_{\sep{k}}=\psi_{\sep{k}}'$, they coincide in a Galois stable open subvariety of $Y_{\sep{k}}$ and therefore $\psi=\psi'$. If $\mathfrak{f}_{\sep{k}}=\mathfrak{f}_{\sep{k}}'$, then $\mathfrak{f}_{\sep{k}}(m)=\mathfrak{f}_{\sep{k}}'(m)$ for every $m\in M$. This implies that $\mathfrak{f}(m)=\mathfrak{f}'(m)$, for every $m\in M$ and, therefore, $\mathfrak{f}=\mathfrak{f}'$ (see \cref{remark: isomorphisms plurifunction}). Then, we have that the functor $\mathfrak{PPDiv}_{}(k)\to \mathfrak{PPDiv}_{}(\sep{k})$ is faithful.
\end{proof}

For a fixed lattice $N$, we denote by $\mathfrak{PPDiv}_{N}(k)$ (resp. $\mathfrak{PPDiv}_{N}(\sep{k})$) the full subcategory of $\mathfrak{PPDiv}(k)$ (resp. $\mathfrak{PPDiv}_{N}(\sep{k})$) consisting of those pp-divisors whose tail cone is defined in $N_{\Q}$.

\begin{corollary}\label{Corollary covariant base change}
The induced functor $\mathfrak{PPDiv}_{N}(k)\to\mathfrak{PPDiv}_{N}(\sep{k})$ is faithful.
\end{corollary}

\section{Split torus actions: affine normal varieties and pp-divisors}\label{Section Normal varieties and split tori}
\indent

When $k=\sep{k}$ and $\mathrm{char}(k)=0$, Altmann and Hausen proved that any affine normal variety endowed with an effective action of an algebraic torus over $k$ arises from a pp-divisor over a normal semi-projective variety over $k$ (cf. \cref{theoremmainaltmannhausen}). In the first part of this section we generalize such a result by proving the following.

\begin{theorem}\label{theorem Altmann hausen split}
    Let $T$ be a split $k$-torus and $N$ be its cocharacter lattice.
        \begin{enumerate}[i)]
            \item \label{theorem Altmann hausen split part a} Let $\D\in\mathfrak{PPDiv}_{N}(k)$ be a pp-divisor over a normal semi-projective variety $Y$ over $k$, then the scheme $X[Y,\D]:=\Spec(A[Y,\D])$ is a normal $k$-variety with an effective $T$-action.
            \item \label{theorem Altmann hausen split part b} Let $X$ be an affine normal $k$-variety with an effective $T$-action. Then, there exists $\D\in\mathfrak{PPDiv}_{N}(k)$ over a normal semi-projective variety $Y$ over $k$ such that $X\cong X[Y,\D]$ as $T$-varieties.
        \end{enumerate}
\end{theorem}

We introduce first the notions of \textit{affinization of a scheme} and its respective \textit{affinization morphism} (see \cite[Chapter III Section 3.8]{DG70} and $\mathrm{SGAIII\, VI_{B},\, Section\, 11}$). This is because \cref{theorem Altmann hausen split}~\eqref{theorem Altmann hausen split part a} is a consequence of a more extensive result involving another $T$-variety, which is not affine and whose respective affinization is the required $T$-variety.

\subsection{ Affinization}\label{section: affinization}
\indent

 Let $S$ be a scheme, its \emph{affinization} is defined as $S_{\mathrm{aff}}:=\Spec(\H^{0}(S,\mathscr{O}_{S}))$. This scheme comes with a natural morphism called the \emph{affinization morphism} $r_{S}:S\to S_{\mathrm{aff}}$, which is defined by gluing the morphisms $U\overset{\cong}{\to}\Spec(\H^{0}(U,\mathscr{O}_{S}))\to S_{\mathrm{aff}}$ for $U\in \mathcal{U}$, an affine open covering of $S$. The affinization morphism $r_{S}:S\to S_{\mathrm{aff}}$ does not depend on the affine open covering. Recall that, for any affine scheme $S'$, there is a bijection $\Hom_{\mathrm{Sch}}(S,S')\cong\Hom_{}(\mathscr{O}(S'),\mathscr{O}(S))$ (see for instance \cite[Chapter I Section 1 Corollary 4.3]{DG70}). Hence, for $S'=\aff{S}$, the affinization morphism $r_{S}$ corresponds to the identity $\mathscr{O}(\aff{S})=\mathscr{O}(S)\to \mathscr{O}(S)$. This construction has some nice functorial properties.

\begin{lemma}
Let $S$ and $S'$ be two schemes. If $f:S\to S'$ is a morphism of schemes, then there exists a canonical morphism $f_{\mathrm{aff}}:S_{\mathrm{aff}}\to S'_{\mathrm{aff}}$ that fits into the following commutative diagram
\[\xymatrix{ S \ar[r]^{r_{S}} \ar[d]_{f} & S_{\mathrm{aff}} \ar[d]^{f_{\mathrm{aff}}} \\ S' \ar[r]_{r_{S'}} & S'_{\mathrm{aff}}  .}\]
\end{lemma}

\begin{proof}
From $f:S\to S'$ we have a canonical map $f^{*}:\H^{0}(S,\mathscr{O}_{S})\to\H^{0}(S',\mathscr{O}_{S'})$, which induces a morphism $f_{\mathrm{aff}}:S_{\mathrm{aff}}\to S'_{\mathrm{aff}}$ that fits into the following commutative diagram
\[\xymatrix{ S \ar[r]^{r_{S}} \ar[d]_{f} & S_{\mathrm{aff}} \ar[d]^{f_{\mathrm{aff}}} \\ S' \ar[r]_{r_{S'}} & S_{\mathrm{aff}}  .}\]
\end{proof}

Also, for quasi-compact and quasi-separated schemes over a field $k$, the affinization of a finite product is isomorphic to the product of the respective affinizations. The following result corresponds to \cite[Chapter I Section 2 Proposition 2.6]{DG70}.

\begin{proposition}\label{proposition affinization of a finite product}
	Let $S_{1},\dots S_{r}$ be quasi-compact and quasi-separated schemes over a field $k$. Then, there is a canonical isomorphism 
	\[\left(\prod S_{i}\right)_{\mathrm{aff}}\cong \prod \left( S_{i} \right)_{\mathrm{aff}}.\]
\end{proposition}

In these terms, a scheme over $k$ is said to be \emph{semi-projective} if its affinization morphism is projective and its affinization is of finite type over $k$. 

\begin{lemma}\label{Proposition properties of affinization}
Let $S$ be a scheme over $k$. Then, the following hold
\begin{enumerate}[a)]
	\item \label{Proposition properties of affinization part i} If $S$ is integral, then $S_{\mathrm{aff}}$ is integral.
	\item \label{Proposition properties of affinization part ii} If $S$ is normal, then $S_{\mathrm{aff}}$ is normal.
	\item \label{Proposition properties of affinization part iii} If $S$ is semi-projective, then $S_{}$ is a separated scheme of finite type over $k$.
\end{enumerate}
\end{lemma}

\begin{proof}
Given that $S$ is integral, then $\mathscr{O}_{S}(S)$ is an integral domain. This implies that the affinization $\Spec(\mathscr{O}_{S}(S))$ is integral, which proves \eqref{Proposition properties of affinization part i}. Now, by \cite[Proposition 4.1.5]{Liu02}, $\mathscr{O}_{S}(S)$ is a normal domain. Thus, the affinization $\Spec(\mathscr{O}_{S}(S))$ is a normal integral scheme. This proves \eqref{Proposition properties of affinization part ii}. Finally, if $S$ is semi-projective, then $r_{S}:S\to S_{\mathrm{aff}}$ is of finite type and $\aff{S}$ is of finite type by definition (see \cref{section: semi-projective varieties}). Since $S_{\mathrm{aff}}\to\Spec(k)$ is separated, we have that $S$ is separated. This proves \eqref{Proposition properties of affinization part iii}. Thus, the assertion holds.
\end{proof}

\begin{remark}\label{proposition S finite type if Saff is}
Let $S$ be a scheme over $k$. Notice that whenever $S_{\mathrm{aff}}$ is of finite type and $r_{S}:S\to S_{\mathrm{aff}}$ is of finite type, then $S$ is of finite type.
\end{remark}

\begin{proposition}\label{proposition affinization separated and finite type}
Let $S$ be a semi-projective scheme over $k$. If $X$ is an affine scheme over $S$, then $X$ is quasi-compact and the affinization morphism $r_{X}:X\to X_{\mathrm{aff}}$ is separated and quasi-compact. Moreover, if $X$ is of finite type over $S$, then $r_{X}$ is of finite type and $X_{\mathrm{aff}}$ is of finite type.
\end{proposition}

\begin{proof}
By \cref{Proposition properties of affinization}, we have that $S$ is of finite type. Then, given that $X$ is affine over $S$, we have that $X$ is quasi-compact. This implies that $r_{X}$ is quasi-compact. Now, we have the canonical morphism $f_{\mathrm{aff}}:X_{\mathrm{aff}}\to S_{\mathrm{aff}}$ that fits into the following commutative diagram
\[\xymatrix{ X \ar[r]^{r_{X}} \ar[d]_{f} & X_{\mathrm{aff}} \ar[d]^{f_{\mathrm{aff}}} \\ S \ar[r]_{r_{S}} & S_{\mathrm{aff}}  .}\]
Thus, given that $r_{S}$, $f_{\mathrm{aff}}$ and $f$ are separated, we have that $r_{X}$ is separated. 

If $f$ is of finite type, then $r_{S}\circ f=\aff{f}\circ r_{X}$ is of finite type. Then, by \cite[Proposition 3.2.4]{Liu02}, we have that $r_{X}$ is of finite type. 
\end{proof}

By Nagata's compactification Theorem \cite{Nag62}, a scheme of finite type over a noetherian ring has a compactification, i.e. a quasi-compact open immersion into a proper scheme. This result allows us to construct schemes with proper affinization morphisms. Notice that the affinization of a scheme and of a compactification are not necessarily isomorphic. For example, the affinization of the affine space $\mathbb{A}_{k}^{n}$ is itself, and the affinization of $\mathbb{P}_{k}^{n}$ is $\Spec(k)$.

\begin{proposition}
Let $S$ be an integral scheme of finite type over a noetherian ring $A$ and $r_{S}:S\to S_{\mathrm{aff}}$ be its affinization morphism. If $\bar{S}$ is a compactification of $S$ over $A$, then $r_{\bar{S}}:\bar{S}\to \bar{S}_{\mathrm{aff}}$ is proper. 
\end{proposition}

\begin{proof}
From the hypothesis, we have the commutative diagram 
\[\xymatrix{ S \ar[d]_{r_{S}} \ar[r]^{\iota} & \bar{S} \ar[d]_{r_{\bar{S}}} \ar[r]^{p}  & \Spec(A) \ar[d]^{\id} \\ S_{\mathrm{aff}} \ar[r]_{\alpha} & \bar{S}_{\mathrm{aff}} \ar[r]_{\beta} & \Spec(A)\, . }\]
Given that $p=\beta\circ r_{\bar{S}}$ is proper and $\beta$ is separated, we have that $r_{\bar{S}}$ is proper.
\end{proof}

A case where the affinization is preserved is under blow-ups on normal schemes.

\begin{proposition}\label{proposition affinization blow up}
Let $S$ be a normal noetherian scheme and $\mathscr{I}$ be a coherent sheaf of ideals of $S$. Let $\mathscr{S}:=\oplus_{d\geq 0}\mathscr{I}^{d}$, where $\mathscr{I}^{d}$ is the $d$th power of the ideal $\mathscr{I}$ and $\mathscr{I}^{0}=\mathscr{O}_{S}$. If $S':=\mathrm{Proj} \mathscr{S}$ is the blow-up of $S$ with respect to the coherent sheaf of ideals $\mathscr{I}$, then $S'_{\mathrm{aff}}=S_{\mathrm{aff}}$.
\end{proposition}

\begin{proof}
Let $\pi:S'\to S$ be the canonical morphism. Hence, we have the following commutative diagram induced by the functoriality of the affinization
\[ \xymatrix{ S' \ar[r]^{\pi} \ar[d]_{r_{S'}} & S \ar[d]^{r_{S}} \\ S'_{\mathrm{aff}} \ar[r]_{\aff{\pi}} & S_{\mathrm{aff}} \, .} \] 
Since $S$ is normal, we have that $\pi_{*}\mathscr{O}_{S'}=\mathscr{O}_{S}$ by Zariski's main Theorem (see \cite[Corollary 4.3.12]{EGAIII-I}). This implies that $\aff{\pi}$ corresponds to the identity and, therefore,
\[\H^{0}(S',\mathscr{O}_{S'})\cong\H^{0}(S,\pi_{*}\mathscr{O}_{S'})=\H^{0}(S,\mathscr{O}_{S}).\]
Thus, $\aff{\pi}$ is an isomorphism.
\end{proof}

In the following section, we will focus on semi-projective varieties and some of their properties.

\subsection{ Semi-projective varieties}\label{section: semi-projective varieties}
\indent

A variety $Y$ over $k$ whose affinization morphism $r_Y:Y\to Y_{\mathrm{aff}}$ is proper is called \textit{semi-affine} (cf.~\cite{GL73}). A variety $Y$ over $k$ is said to be \emph{semi-projective} if its ring of global sections $\H^{0}(Y,\mathscr{O}_{Y})$ is a finitely generated $k$-algebra and $Y$ is projective over $\aff{Y}$. Notice that this definition coincides with the previous one for schemes. In other words, a variety $Y$ is semi-projective if and only if it is a semi-affine variety, $r_Y:Y\to \aff{Y}$ is quasi-projective, and its ring of global sections $\H^{0}(Y,\mathscr{O}_{Y})$ is a finitely generated $k$-algebra. Notice that a semi-projective variety is quasi-projective. 
 
Through this section, some results regarding semi-projective varieties are exposed. We start by proving that semi-affine varieties are stable under finite products.

\begin{proposition}\label{proposition stability on semi-affineness}
If $\{Y_{i}\}_{i\in I}$ is a finite set of semi-affine varieties over $k$, then the product $\prod_{i\in I} Y_{i}$ is semi-affine.
\end{proposition}

\begin{proof}
Denote $Y:=\prod_{i\in I}Y_{i}$. By \cref{proposition affinization of a finite product}, we have that $Y_{\mathrm{aff}}\cong \prod (Y_{i})_{\mathrm{aff}}$. Since the affinization morphism $r_{Y}$ corresponds to the product of all the $r_{Y_{i}}$, which is finite product, and all of them are proper, it follows that $r_{Y}$ is proper.
\end{proof}

The same holds for semi-projective varieties.

\begin{proposition}\label{proposition product of semi-projective}
If $\{Y_{i}\}_{i\in I}$ a finite set of semi-projective varieties over $k$, then the product $\prod_{i\in I} Y_{i}$ is semi-projective.
\end{proposition}

\begin{proof}
Denote $Y:=\prod_{i\in I}Y_{i}$. By \cref{proposition stability on semi-affineness}, we now have that the affinization map $r_{Y}:Y\to\aff{Y}$ is proper. Notice that $\H^{0}(Y,\mathscr{O}_{Y})\cong \bigotimes\H^{0}(Y_{i},\mathscr{O}_{Y_{i}})$, which is a finite product, then $\H^{0}(Y,\mathscr{O}_{Y})$ is a finitely generated $k$-algebra. Finally, given that a finite product of quasi-projective morphisms is quasi-projective, we have that $r_{Y}:Y\to\aff{Y}$ is quasi-projective. Thus, the assertion holds.
\end{proof}

The following results are useful properties of semi-projective varieties. In particular, in the understanding of how equivariant morphism of affine normal $T$-varieties in terms of the combinatorial data (see for instance \cref{lemma semilinear resolution}).

\begin{lemma}\label{lemma semi-projectiveness stable under projective morphisms}
Let $Y$ be a semi-projective $k$-variety and $Y'$ be a $k$-variety with $f:Y'\to Y$ a projective morphism. Then $Y'$ is semi-projective.
\end{lemma}

\begin{proof}
We have the following commutative diagram 
\[\xymatrix{Y' \ar[r]^{f} \ar[d]_{r_{Y'}} & Y \ar[d]^{r_{Y}} \\ \aff{Y'} \ar[r]_{\aff{f}} & \aff{Y} .}\] 

Given that $\aff{Y}$ is a $k$-variety, by \cite[\href{https://stacks.math.columbia.edu/tag/0C4P}{Tag 0C4P}]{stacks-project}, we have that $r_{Y}\circ f:Y'\to \aff{Y}$ is projective. Given that $\aff{f}:\aff{Y'}\to \aff{Y}$ is separated and $\aff{f}\circ r_{Y'}=r_{Y}\circ f$ is projective, by \cite[\href{https://stacks.math.columbia.edu/tag/0C4Q}{Tag 0C4Q}]{stacks-project}, we have that $r_{Y'}$ is projective. Then, $Y'$ is semi-projective.
\end{proof}

\begin{proposition}\label{propositiondominantresolution0}
    Let $W,\,Y$ and $Z$ be normal semi-projective varieties over $k$ with birational maps
    \[\xymatrix{ W & Y \ar@{-->}[l]_{\alpha} \ar@{-->}[r]^{\beta} & Z.}\]
     Then, there exists a normal semi-projective variety $\tilde{Y}$ with birational morphisms $\tilde{Y}\to W,Y,Z$ such that the diagram 
     \[\xymatrix{ & \tilde{Y} \ar[ld]_{\kappa_{W}} \ar[d]_{\kappa_{Y}} \ar[rd]^{\kappa_{Z}} & \\ W & Y \ar@{-->}[l]^{\alpha} \ar@{-->}[r]_{\beta} & Z }\]
      commutes.
\end{proposition}

\begin{proof}
Let $U_{W}\subset Y$ be the largest open subvariety where $\alpha|_{U_{W}}:U_{W}\to W$ is defined and $U_{Z}\subset Y$ be the largest open subvariety where $\beta|_{U_{Z}}:U_{Z}\to Z$ is defined. Denote $U:=U_{W}\cap U_{Z}$. Let $Y_{1}$ be the normalization of the closure of the graph of $\beta|_{U}:U\to Y$ on $Y\times Z$. Then, we have the following diagram 
\[\xymatrix{ & Y_{1} \ar[d]_{\kappa_{1}} \ar[rd]^{\kappa_{2}} & \\ W & Y \ar@{-->}[l]^{\alpha} \ar@{-->}[r]_{\beta} & Z ,}\]
where $\kappa_{1}$ and $\kappa_{2}$ are the projections, which are also birational. Now, consider the rational map $\alpha\circ\kappa_{1}:Y_{1}\dasharrow W$. Notice that this map is defined over $\kappa_{1}^{-1}(U)$. Then, as before, let $\tilde{Y}$ be the normalization of the closure of the graph of $\alpha\circ\kappa_{1}:\kappa_{1}^{-1}(U)\to W$ on $W\times Y$. Thus, we have the following commutative diagram 
\[\xymatrix{ & \tilde{Y} \ar[d]^{\kappa_{3}} \ar[ddl]_{\kappa_{W}} & \\ & Y_{1} \ar[d]_{\kappa_{1}} \ar[rd]^{\kappa_{2}} \ar@{-->}[dl]_{} & \\ W & Y \ar@{-->}[l]^{\alpha} \ar@{-->}[r]_{\beta} & Z ,}\]
where $\kappa_{W}$ and $\kappa_{3}$ are the projections which are also birational. Then, $\kappa_{W}$, $\kappa_{Y}:=\kappa_{3}\circ\kappa_{1}$ and $\kappa_{Z}:=\kappa_{3}\circ\kappa_{2}$ are the desired morphisms.

Let us prove now the semi-projectiveness of $Y_{1}$. By \cref{proposition product of semi-projective}, $Y\times Z$ is semi-projective and, therefore, $\overline{\mathrm{Graph}(\beta_{U})}$ is semi-projective. Hence, given that the normalization is a finite morphism, we conclude that $Y_{1}$ is semi-projective by \cref{lemma semi-projectiveness stable under projective morphisms}. Now, by \cref{proposition product of semi-projective}, $Y_{1}\times W$ is semi-projective and, therefore, $\overline{\mathrm{Graph}(\alpha\circ\kappa_{1}(k^{-1}(U))}$ is semi-projective. Hence, we conclude that $\tilde{Y}$ is semi-projective by arguing as for the semi-projectivity of $Y_{1}$.
\end{proof}

The morphism $\kappa_{Y}:\tilde{Y}\to Y$ need not be projective. For example, let $\beta:\A^{2}_{k}\dasharrow \A^{2}_{k}$ be given by $(x,y)\to(x,y/x)$. The graph of $\beta$ in $\A_{k}^{2}\times \A_{k}^{2}=\Spec(k[x,y,v,w])$ is $\{x=v,\,xw=y\}$, which is isomorphic to $\A_{k}^{2}$. In this case, we have the following resolution of indeterminacies:
\[\xymatrix{ & \A_{k}^{2} \ar[ld]_{\kappa_{Y}} \ar[rd]^{\kappa_{Z}}& \\ \A_{k}^{2} \ar@{-->}[rr] & & \A_{k}^{2} ,}\]
where $\kappa_{Y}$ is given by $(x,w) \to (x,xw)$. This morphism is not projective. However, under some extra hypothesis, we can ensure the projectiveness of $\kappa_{Y}$.

\begin{proposition}\label{propositiondominantresolution}
    Let $W,\,Y$ and $Z$ be normal semi-projective varieties over $k$ with birational maps $\alpha$ and $\beta$
    \[\xymatrix{ W & Y \ar@{-->}[l]_{\alpha} \ar@{-->}[r]^{\beta} & Z }\]
     such that there exists regular morphisms $\alpha':\aff{Y}\to \aff{W}$ and $\beta':\aff{Y}\to \aff{Z}$ with $r_{W}\circ \alpha=\alpha'\circ r_{Y}$ and $r_{Z}\circ\beta=\beta'\circ r_{Y}$. Then, there exists a normal semi-projective variety $\tilde{Y}$ with birational morphisms $\kappa_{W},\kappa_{Y},\kappa_{Z}:\tilde{Y}\to W,Y,Z$ such that the diagram 
     \[\xymatrix{ & \tilde{Y} \ar[ld]_{\kappa_{W}} \ar[d]_{\kappa_{Y}} \ar[rd]^{\kappa_{Z}} & \\ W & Y \ar@{-->}[l]^{\alpha} \ar@{-->}[r]_{\beta} & Z }\]
      commutes and $\kappa_{Y}$ is projective.
\end{proposition}

\begin{proof}
By \cref{propositiondominantresolution0}, we have the existence of the normal semi-projective variety $\tilde{Y}$ over $k$ and the birational maps $\kappa_{W}$, $\kappa_{Y}$, and $\kappa_{Z}$. We claim that $\kappa_{Y}$ is projective. We keep the constructions made in the proof of \cref{propositiondominantresolution0} with all its respective notations. Since $r_{Z}\circ\beta=\beta'\circ r_{Y}$, we have the following commutative diagram 
\[\xymatrix{ \overline{\mathrm{Graph}(\beta|_{U})} \ar[r]^{\iota} \ar[d]_{f} & Y\times Z \ar[d]^{r_{Y\times Z}}  \\ \overline{\mathrm{Graph}(\beta')} \ar[r]_{\iota'} & Y_{\mathrm{aff}}\times Z_{\mathrm{aff}},}\]
where $f$ is the restriction of $r_{Y\times Z}$. Since $\iota$ and $\iota'$ are closed embeddings, they are finite and, therefore, projective. Moreover, given that $r_{Y\times Z}$ is projective and $\iota\circ r_{Y\times Z}=f\circ\iota'$, we have that $f$ is projective. Hence, given that $\beta'$ is a regular map, it follows that 
\[\overline{\mathrm{Graph}(\beta')}=\mathrm{Graph}(\beta')\cong \aff{Y},\]
because $\aff{Z}$ is a separated scheme. Thus, $\overline{\mathrm{Graph}(\beta|_{U})}$ is projective over $\aff{Y}$. Then, $\kappa_{1}\circ r_{Y}:Y_{1}\to \aff{Y}$ is projective and, therefore, $\kappa_{1}:Y_{1}\to Y$ is so. This implies that $\kappa_{1}\circ r_{Y}$ is projective.

Recall that $\kappa_{Y}=\kappa_{3}\circ \kappa_{1}$. In this case, we have the following commutative diagram
\[\xymatrixcolsep{4pc}\xymatrix{ \overline{\mathrm{Graph}(\alpha\circ\kappa_{1}|_{\kappa_{1}^{-1}(U)}}) \ar[r]^{i} \ar[d]_{g} & Y_{1}\times W \ar[d]^{(\kappa_{1}\circ r_{Y})\times r_{W}}  \\  \overline{\mathrm{Graph}(\alpha')} \ar[r]_{i'} & \aff{Y}\times \aff{W},}\]
where $g$ is the restriction of $(\kappa_{1}\circ r_{Y})\times r_{W}$. The maps $i$ and $i'$ are projective because they are closed embeddings. Hence, given that $(\kappa_{1}\circ r_{Y})\times r_{W}$ is projective by \cref{proposition product of semi-projective} and $i\circ((\kappa_{1}\circ r_{Y})\times r_{W})=g\circ i'$, we have that $g$ is projective. Since $\alpha'$ is regular and arguing as before, we have that $\aff{Y}\cong \overline{\mathrm{Graph}(\alpha')}$ and a commutative diagram
\[\xymatrixcolsep{4pc}\xymatrix{ \tilde{Y} \ar[r]^{\nu} \ar[d]_{\kappa_{3}\circ\kappa_{1}\circ r_{Y}} & \overline{\mathrm{Graph}(\alpha\circ\kappa_{1}|_{\kappa_{1}^{-1}(U)})} \ar[d]^{g}  \\  \aff{Y} \ar[r]_{\cong} & \overline{\mathrm{Graph}(\alpha')},}\]
where $\nu:\tilde{Y}\to \overline{\mathrm{Graph}(\alpha\circ\kappa_{1}|_{\kappa_{1}^{-1}(U)})}$ is the normalization. Then, $\kappa_{3}\circ\kappa_{1}\circ r_{Y}:\tilde{Y}\to \aff{Y}$ is projective and, therefore, $\kappa_{Y}$ is projective. This proves the assertion.
\end{proof}

\subsection{ From pp-divisors to affine normal varieties}\label{section: from pp-divisors to varieties}
\indent

The main goal of this section is to give a proof of \cref{theorem Altmann hausen split}~\eqref{theorem Altmann hausen split part b}. Before proving the claim, we present the following result, which is a relative version of \cite[Theorem 1.3]{Fuj83} due to Goodman and Landman \cite{GL73}. The original statement is performed over an algebraically closed field, and a generalization to any field holds by descent arguments.

\begin{proposition}\cite[Corollary 9.4]{GL73}\label{proposition D semiample}
Let $Y$ be a semi-affine variety over $k$. Let $D\in\cadiv_{\Q}(Y)$ be a semiample divisor, then 
\[\bigoplus_{n\geq0}\H^{0}(Y,\mathscr{O}_{Y}(nD))\] 
is a finitely generated $k$-algebra.
\end{proposition}

Let $\D$ be a pp-divisor in $\PPDiv_{\Q}(Y,\omega)$. From $\D$ we can construct the following $M$-graded $k$-algebra
\[A[Y,\D]:=\bigoplus_{m\in\omega^{\vee}\cap M}\H^{0}(Y,\mathscr{O}_{Y}(\D(m)))\subset k(Y)[M]\]
and its respective scheme $X[Y,\D]:=\Spec(A[Y,\D])$, just as in \cite[Section 3]{AH06}. We claim that such a scheme is indeed an affine normal variety over $k$ endowed with an effective action of $T=\Spec(k[M])$. 

\begin{proposition}\label{corollary fg algebra}
Let $\D$ be a pp-divisor over a normal semi-projective variety $Y$ over $k$ with tail cone $\omega\subset N_{\Q}$. Then, the $M$-graded $k$-algebra $A[Y,\D]$ is finitely generated and integral.
\end{proposition}

\begin{proof}
Let $\Lambda$ be the quasifan associated to $\mathfrak{h}_{\D}$ with support $|\Lambda|=\omega^{\vee}$ (see: \cref{proposition ppdiv cplmaps}). For every $\lambda\in \Lambda$ the map $\mathfrak{h}_{\D}|_{\lambda}$ is linear and the monoid $\lambda\cap M$ is finitely generated.

Let $\lambda\in\Lambda$ and $\{m_{1},\dots,m_{l}\}\subset \lambda\cap M$ be a set of generators of the monoid. Denote $D_{i}=\mathfrak{h}_{\D}(m_{i})$. For every $m_{i}\in\{m_{1},\dots,m_{l}\}$, by \cref{proposition D semiample}, we know that the $k$-algebra 
\[A[Y,\D](m_{i}):=\bigoplus_{n\geq 0}\H^{0}(Y,\mathscr{O}_{Y}(nD_{i}))\subset k(Y)[M]\]
is finitely generated. Hence, given that $\mathfrak{h}_{\D}$ is linear over $\lambda$, the $k$-algebra
\[A[Y,\D](\lambda):=\bigoplus_{m\in\lambda\cap M}\H^{0}(Y,\mathscr{O}_{Y}(\mathfrak{h}_{\D}(m)))\subset k(Y)[M]\]
coincides with the algebra generated by the algebras $A[Y,\D](m_{i})$ inside of $k(Y)[M]$. Otherwise stated, we have the following equality
\[A[Y,\D](\lambda)=\left\langle A[Y,\D](m_{1}),\dots,A[Y,\D](m_{l}) \right\rangle\subset k(Y)[M].\]
Then, $A[Y,\D](\lambda)$ is a finitely generated $k$-algebra. 

Given that the support of $\Lambda$ is $\omega^{\vee}\cap M$, we have that 
\[ A[Y,\D]=\left\langle A[Y,\D](\lambda) \mid \lambda\in\Lambda \right\rangle\subset k(Y)[M].\]
Thus, given that $\Lambda$ is a finite set, we have that $A[Y,\D]$ is a finitely generated $k$-algebra.

The $k$-algebra $A[Y,\D]$ is integral because is a subalgebra of $k(Y)[M]$.
\end{proof}

We now recall the notion of contraction morphism. 

\begin{definition}
A proper morphism of varieties $f:X\to Z$ is said to be a \emph{contraction} if $f^{\#}:\mathscr{O}_{Z}\to f_{*}\mathscr{O}_{X}$ is bijective.
\end{definition}

The following result is based on \cite[Theorem 3.1]{AH06}. The proof of this proposition is word by word the one given by Altmann and Hausen, with the exception of the integrality and finiteness of the algebra $A[Y,\D]$ that is proved in \cref{corollary fg algebra}. This proposition corresponds to \cref{theorem Altmann hausen split}~\eqref{theorem Altmann hausen split part b}.

\begin{proposition}\label{proposition ppdiv to variety}
Let $Y$ be a normal semi-projective variety over $k$, $N$ be a lattice, $M$ be its dual lattice, $\omega\subset N_{\Q}$ be a cone. Let $\D\in\PPDiv_{\Q}(Y,\omega)$ be a pp-divisor and consider the $\mathscr{O}_{Y}$-algebra
\[\mathscr{A}:=\bigoplus_{m\in\omega^{\vee}\cap M}\mathscr{A}_{m}:=\bigoplus_{m\in\omega^{\vee}\cap M}\mathscr{O}_{Y}(\D(m)).\]
Denote $T:=\Spec(k[M])$ and $\tilde{X}:=\Spec_{Y}(\mathscr{A})$, the relative spectrum. Then, the followings hold:
 \begin{enumerate}[i)]
 	\item\label{Proposition pp-div to variety part i} The scheme $\tilde{X}$ is a normal variety over $k$ of dimension $\dim(Y)+\dim(T)$ and the grading defines an effective action of $T$ over $\tilde{X}$ having the canonical map $\pi:\tilde{X}\to Y$ as a good quotient, i.e. $\pi$ is affine, $T$-invariant and $\pi^{\#}:\mathscr{O}_{Y}\to \pi_{*}(\mathscr{O}_{\tilde{X}})^{T}$ is an isomorphism.
	\item\label{Proposition pp-div to variety part ii} The ring of global sections $\H^{0}(\tilde{X},\mathscr{O}_{\tilde{X}})=\H^{0}(Y,\mathscr{A})=A[Y,\D]$ is a finitely generated $M$-graded, normal $k$-algebra. Moreover, the affinization morphism $r_{\tilde{X}}:\tilde{X}\to X[Y,\D]:=\Spec(A[Y,\D])$ is a $T$-equivariant proper birational contraction.
	\item\label{Proposition pp-div to variety part iii} Let $m\in\omega^{\vee}\cap M$ and $f\in A_{m}$. Then we have $\pi(\tilde{X}_{f})=Y_{f}$. Therefore, when $Y_{f}$
is affine, then so is $\tilde{X}_{f}$, and the canonical map $\tilde{X}_{f} \rightarrow X_{f}$ is an isomorphism. Moreover, even for non-affine $Y_{f}$ , we have
\[\H^{0}(Y_{f}, \mathscr{A})=\bigoplus_{m\in\omega^{\vee}\cap M} (A_{f} )_{m}.\]
\end{enumerate}
\end{proposition}

\begin{remark}\label{remark diagram affinization split actions}
Let $\D\in\PPDiv_{\Q}(Y,\omega)$. In general, we do not have a map $X(\D)\to Y$, but we do have the following commutative diagram
\[\xymatrix{ \tilde{X} \ar[r]^{r_{\tilde{X}}} \ar[d]_{\sslash T}& X(\D) \ar[d] \\ Y \ar[r]_{r_{Y}} & \aff{Y}, }\]
where the horizontal arrows are affinizations. Thus, if $Y$ is affine, we have that $Y=\aff{Y}$ and, therefore, we have a map $X(\D)\to Y$. Moreover, since the morphism $\tilde{X}\to Y$ is affine, we have that $\tilde{X}$ is affine. Thus, we have that $r_{\tilde{X}}:\tilde{X}\to X(\D)$ is an equivariant isomorphism. For further discussions about the $T$-variety $\tilde{X}$ see \cref{section: The other T-variety}.
\end{remark}

\subsection{ From affine normal varieties to pp-divisors}\label{section: from affine to pp-divisors}
\indent

Through out this section we construct a pp-divisor from an affine normal variety endowed with an effective action of a split algebraic torus. The proof of \cref{theorem Altmann hausen split}~\eqref{theorem Altmann hausen split part a} that we present here follows the strategy used by Altmann and Hausen in \cite[Sections 5 and 6]{AH06}. First, we start by building the normal semi-projective variety. The construction of such a variety lies within the Geometric Invariant Theory (GIT) \cite{GIT}.

\subsubsection{Construction of the normal semi-projective variety.} 
\noindent

Let $T$ be a split algebraic $k$-torus and $X:=\Spec(A)$ be an affine normal $k$-variety on which $T$ acts effectively. Let $M$ be the character lattice of $T$ and $N:=\Hom_{\Z}(M,\Z)$ as before. It is known that $A$ has an $M$-graduation from the torus action: 
\[A=\bigoplus_{m\in M}A_{m}.\] 
Since $A$ is a finitely generated $k$-algebra, the set $\{m\in M \mid A_{m}\neq 0\}$ forms a finitely generated monoid and generates a convex polyhedral cone $\omega^{\vee}\subset M_{\Q}$ called \textit{weight cone}. We denote by $\omega\subset N_{\Q}$ the dual cone of the weight cone.

Let $\mathcal{L}$ be a $T$-linearized line bundle over $X$. A $T$-linearization of $\mathcal{L}$ induces an action of $T$ on the space of sections $\H^{0}(X,\mathcal{L})$ as follows: for $s\in\H^{0}(X,\mathcal{L})$ we have
\[(t\cdot s)(x):=t\cdot s(t^{-1}x).\]
By definition, the space of \textit{semistable points} associated to $\mathcal{L}$, denoted by $\Xss(\mathcal{L})$, is the set of $x\in X$ such that for some $n\in\N$ there exists a $T$-invariant section $s\in\H^{0}(X,\mathcal{L}^{n})$ such that $s(x)\neq 0$. Since algebraic tori are linearly reductive over any field \cite{Nag61}, the geometric quotient $\Xss(\mathcal{L})\sslash T$ exists by \cite[Theorem 1.10]{GIT} over any field and, therefore, Altmann and Hausen's strategy adapts to any characteristic. 

Before carrying on with the construction, notice that the space of semistable points $\Xss(\mathcal{L})$ depends on the $T$-linearization. Two $T$-linearized line bundles $\mathcal{L}$ and $\mathcal{L}'$ are called \emph{GIT-equivalent} if $\Xss(\mathcal{L})=\Xss(\mathcal{L}')$.

Let $\mathcal{L}$ be the trivial line bundle. For each $m\in M$ there exists a $T$-linearization of $\mathcal{L}$ given by 
\begin{equation}\label{linearization}
t\cdot(x,r)\to (tx,\chi^{m}(t)r),
\end{equation}
where $\chi^{m}$ denotes the character associated to $m$. Denote by $\Xss(m):=\Xss(\mathcal{L})$ the space of semistable points associated to $\mathcal{L}$ with respect to $m\in M$ and by $Y_{m}:=\Xss(m)\sslash T$ its respective geometric quotient. The main idea of Altmann and Hausen in \cite{AH06} is to glue all these quotients $Y_{m}$ for $m\in\omega^{\vee}\cap M$. But before gluing all these quotients, we need to establish which ones among them are GIT-equivalent. This was studied by Berchtold and Hausen in \cite{BH06} when $k$ is an algebraically closed field of characteristic zero. The main definitions and results in \cite{BH06} can be summarized in the following.

\begin{definition}
    Let $x\in X$ be a closed point.
    \begin{enumerate}[i)]
        \item The \textit{orbit monoid} associated to $x\in X$ is the submonoid $S(x)\subset M$ consisting of all $m\in M$ that admit an $f\in A_{m}$ with $f(x)\neq 0$.
        \item The \textit{orbit cone} associated to $x\in X$ is the convex cone $\omega(x)^{\vee} \subset M_{\Q}$ generated by the orbit monoid.
        \item The \textit{orbit lattice} associated to $x\in X$ is the sublattice $M(x)\subset M$ generated by the orbit monoid.
    \end{enumerate}
\end{definition}

The orbit cones are polyhedral and they are contained in $\omega^{\vee}$. The following corresponds to \cite[Proposition 5.2]{AH06}.

\begin{proposition}\label{proposition orbits}
	Let $k$ be a separably closed field. Let $x\in X$ be a $k$-rational point.
    \begin{enumerate}[i)]
        \item\label{proposition orbits part a} The orbit lattice $M(x)$ consists of all $m \in M$ that admit an $m$-homogeneous function $f \in k(X)$, i.e. $f=g/h$ with $g\in k[X]_{m+m'}$ and $h\in k[X]_{m'}$ for some $m'\in M$, that is defined and invertible near $x$.
        \item\label{proposition orbits part b} The isotropy group scheme $T_{x} \subset T$ of the point $x \in X$ is the diagonalizable group given by $T_{x} = \Spec(k[M/M(x)])$.
        \item\label{proposition orbits part c} The orbit closure $\overline{T\cdot x}$ is isomorphic to $\Spec(k[S(x)])$; it comes along with an equivariant open embedding of the torus $T/T_{x} = \Spec(k[M(x)])$.
        \item\label{proposition orbits part d}  The normalization of the orbit closure $\overline{T\cdot x}$ is the affine toric variety corresponding to the cone $\omega(x)^{}$ in $\Hom(M(x),\Z)$.    
    \end{enumerate}
\end{proposition}

\begin{proof}
	
	Proof of \eqref{proposition orbits part a}: By definition, for any $m\in M(x)$, we have that there exists $f\in k[X]_{m}\subset K(X)$ such that $f$ is invertible near $x$. Conversely, let $m\in M$ such that an $m$-homogeneous function $f\in k(X)$ that is defined and invertible near $x$ exists, then we have elements $g\in k[X]_{m+m'}$ and $h\in k[X]_{m'}$ such that $g(x)\neq 0$ and $h(x)\neq 0$. Then, $m+m'$ and $m'$ belong to $M(x)$. Since $M(x)$ is a group, it follows that $m\in M$. This proves part \eqref{proposition orbits part a}.

    Proof of \eqref{proposition orbits part b}: Since $x$ is a $k$-rational point, the orbit $T\cdot x$ is again a torus and it corresponds to $\Spec(k[M(x)])$. This yields a morphism of algebraic tori $T\to\Spec(k[M(x)])$ that is faithfully flat. Hence, the isotropy group scheme $T_{x}$ corresponds to the diagonalizable group scheme $\Spec(k[M/M(x)])$
	
	The rest of the proposition corresponds to \cite[Proposition 2.3]{BH06}, whose proof adapts to this case.

\end{proof}

\begin{remark}
A difference that arises in positive characteristic with respect to the characteristic zero is that the isotropy groups in \cref{proposition orbits} \eqref{proposition orbits part b} may be non-smooth.
\end{remark}

In terms of the orbit cones, there is a simple description of the set $\Xss_{\sep{k}}(m)$ of semistable points. Namely, by \cite[Lemma 2.7]{BH06}, we have 
\[\Xss_{\sep{k}}(m)=\{x\in X_{\sep{k}}\mid m\in \omega(x)^{\vee}\}.\]

\begin{definition}
    The \textit{GIT}-\textit{cone} associated to $m\in \omega^{\vee}\cap M$ is the intersection of all orbit cones containing $m$: \[\lambda(m)^{\vee}:=\bigcap_{x\in \Xss_{\sep{k}}(m)}\omega(x)^{\vee}.\]
\end{definition}

The main results of \cite[Section 2]{BH06} are the following, whose proofs adapt for any separably closed field and in any characteristic.

\begin{theorem}\label{theorem GIT-equivalence}
 Let $k$ be a separably closed field. Let $T:=\Spec(k[M])$ be a $k$-torus acting on a normal variety $X:=\Spec(A)$ over $k$. Then, the following statements hold:
    \begin{enumerate}[i)]
        \item The GIT-cones $\lambda(m)^{\vee}$, where $m\in M$, form a quasifan $\Lambda$ in $M_{\Q}$.
        \item The support of the quasifan $\Lambda$ is the weight cone $\omega^{\vee}\subset M_{\Q}$.
        \item \label{theorem GIT-equivalence part 3} For any two elements $m_{1},m_{2}\in\omega^{\vee}\cap M$, we have 
        \[ \Xss_{}(m_{1})\subset \Xss_{}(m_{2}) \Longleftrightarrow \lambda(m_{2})^{\vee}\subset\lambda(m_{1})^{\vee}.\]
        In particular, the equality holds if and only if $\lambda(m_{2})^{\vee}=\lambda(m_{1})^{\vee}$.
    \end{enumerate}
\end{theorem}

We prove that this theorem also holds over any field whenever $T$ is a split torus.

\begin{proposition}\label{proposition GIT-equivalence}
    Let $T:=\Spec(k[M])$ be a split $k$-torus acting on a normal variety $X:=\Spec(A)$ over $k$. Then, for any two elements $m_{1},m_{2}\in\omega^{\vee}\cap M$, we have 
    \[ \Xss(m_{1})\subset \Xss(m_{2}) \Longleftrightarrow \lambda(m_{2})^{\vee}\subset\lambda(m_{1})^{\vee}.\] 
    In particular, the equality holds if and only if $\lambda(m_{2})^{\vee}=\lambda(m_{1})^{\vee}$.
\end{proposition}

\begin{proof}
    By \cite[Proposition 1.14]{GIT}, we have that 
    \[(\Xss(m_{i}))\times_{\Spec(k)}\Spec(\sep{k})=\Xss_{\sep{k}}(m_{i}).\]
    On the one hand, by \cref{theorem GIT-equivalence}, if $\lambda(m_{2})^{\vee}\subset\lambda(m_{2})^{\vee}$ we have $\Xss_{\sep{k}}(m_{1})\subset \Xss_{\sep{k}}(m_{2})$. Then, $\Xss(m_{1})\subset \Xss(m_{2})$. On the other hand, if $\Xss(m_{1})\subset \Xss(m_{2})$, then $\Xss_{\sep{k}}(m_{1})\subset \Xss_{\sep{k}}(m_{2})$ and, by \cref{theorem GIT-equivalence}, we have that $\lambda(m_{2})^{\vee}\subset\lambda(m_{1})^{\vee}$. 
\end{proof}

The sets of semistable points of a $T$-linearization are $T$-stable open subvarieties of $X=\Spec(A)$ that admit a geometric quotient for the $T$-action. As in \cite[Section 5]{AH06}, for the $T$-linearization \eqref{linearization}, we have that 
\[Y_{m}=\mathrm{Proj}\left(\bigoplus_{n\geq 0} A_{nm}\right)\]
and $Y_{m}$ is projective over $Y_{0}:=\Spec(A_{0})$.

Let us see how the normal semi-projective variety $Y$ is constructed from the action of $T$ over $X$. Let $\Lambda$ be the quasifan in $M_{\Q}$ of \cref{theorem GIT-equivalence}. For every $\lambda\in\Lambda$ and any $m_{1},m_{2}\in\mathrm{relint}(\lambda)$, the sets of semistable points $\Xss(m_{1})$ and $\Xss(m_{2})$ are equal by \cref{proposition GIT-equivalence}. Now, denote by $W_{\lambda}$ the set of semistable points of any $m\in\mathrm{relint}(\lambda)$ and denote by $q_{\lambda}:W_{\lambda}\to Y_{\lambda}$ the corresponding geometric quotient, which is normal by \cite[Section 0.2]{GIT}. Notice that $W_{0}=X$ and it comes with a natural morphism $q_{0}:W_{0}\to Y_{0}=\Spec(A_{0})$. Given that for $\gamma\preceq\lambda$ we have an open embedding $W_{\lambda}\subset W_{\gamma}$, the open subschemes $W_{\lambda}$, with $\lambda\in\Lambda\cup \{0\}$, form a filtered inverse system. Let us denote by 
\[W:=\varprojlim W_{\lambda}=\bigcap_{\lambda\in\Lambda}W_{\lambda},\] 
the inverse limit of the sets of semistable points, which is an open subvariety of $X$. The open embeddings $W_{\lambda}\subset W_{\gamma}$ induce morphisms $p_{\lambda\gamma}:Y_{\lambda}\to Y_{\gamma}$. Denote by $Y'$ the inverse limit of the $Y_{\lambda}$ with respect to the morphisms $p_{\lambda\gamma}$, which exists as a scheme because it is a finite system. There is a canonical map $q':W\to Y'$. The scheme $Y'$ might not be reduced, but it has a canonical reduced component, which is the schematic closure of $q'(W)$ in $Y'_{\mathrm{red}}$. Moreover, this canonical reduced component is geometrically integral and, therefore, geometrically reduced. This holds because $W$ is geometrically integral. Hence, by taking the normalization of $\overline{q'(W)}$, we obtain a normal variety 
\begin{equation}\label{equation: construction Y}Y:=\overline{q'(W)}^{\nu}.\end{equation} 
Moreover, by the universal property of the normalization, there exists a morphism $q:W\to Y$. We claim that $Y$ is projective over $Y_{0}$. Given that the quasifan $\Lambda$ is a finite set, we have that $\prod_{\lambda\in\Lambda} Y_{\lambda}$ is semi-projective by \cref{proposition product of semi-projective}. The inverse limit $\varprojlim Y_{\lambda}\subset \prod_{\lambda\in\Lambda} Y_{\lambda}$ is a closed subscheme and therefore projective over $Y_{0}$, because of the following commutative diagram
\[ \xymatrix{ \varprojlim Y_{\lambda} \ar[r] \ar[d]& \prod_{\lambda\in\Lambda} Y_{\lambda} \ar[d] \\ 
Y_{0} \ar[r] & \prod_{\lambda\in\Lambda} Y_{0}} \]
and by \cite[\href{https://stacks.math.columbia.edu/tag/0C4Q}{Tag 0C4Q}]{stacks-project}. Hence, $\overline{q(W)}$ is also projective over $Y_{0}$. Given that $\nu:Y\to \overline{q(W)}$ is finite, it is projective by \cite[\href{https://stacks.math.columbia.edu/tag/0B3I}{Tag 0B3I}]{stacks-project}. This implies that $Y$ is projective over $Y_{0}$.

\begin{remark}
It is not true that the inverse limit of a finite inverse system of normal varieties is normal, even for a filtered system. For example, consider the filtered inverse system induced by
\[\xymatrix@=1pc{ & k[z] \ar[rd]^{z\mapsto x^{2}} \\ k[u] \ar[ru]^{u\mapsto z} \ar[rd]_{u\mapsto w} & & k[x,y] \\ & k[w] \ar[ru]_{w\mapsto y^{3}} }\]
whose inverse limit is the cuspidal curve, which is not normal.
\end{remark}

Let us study the morphisms $p_{\lambda}$ and $p_{0}$. Consider the following commutative diagram 

\begin{equation}\label{diagram Y projective over Y_{0}}
\xymatrix{W \ar[rr]^{\iota_{\lambda}} \ar[d]_{q} & & W_{\lambda} \ar[rr]^{\iota_{\lambda\gamma}} \ar[d]_{q_{\lambda}} & & W_{\gamma} \ar[rr]^{\iota_{\gamma 0}} \ar[d]_{q_{\gamma}} & & W_{0} \ar[dd]_{q_{0}} \\ 
Y \ar[rr]^{p_{\lambda}} \ar[rrrrrrd]_{p_{0}} & & Y_{\lambda} \ar[rr]^{p_{\lambda\gamma}} \ar[rrrrd]^{p_{\lambda 0}} & & Y_{\gamma} \ar[rrd]^{p_{\gamma 0}} \\
& &  & & & & Y_{0}}
\end{equation}

\begin{proposition}\label{Proposition lemma 6.1 AH06}
The morphisms $p_{\lambda}: Y \to Y_{\lambda}$ and $p_{\lambda\gamma}: Y_{\lambda} \to Y_{\gamma}$ are contractions. Moreover, if $\mathrm{dim}(Y_{\lambda})=\mathrm{dim}(X)-\mathrm{dim}(T)$, for example if $\lambda$ intersects $\mathrm{relint}(\omega^{\vee})$, then the morphism $p_{\lambda} : Y \to Y_{\lambda}$ is birational.
\end{proposition}

\begin{proof}
	Recall that the morphisms $p_{\lambda 0}:Y_{\lambda}\to Y_{0}$ are projective, because 
\[Y_{\lambda}=\mathrm{Proj}\left(\bigoplus_{n\geq 0} A_{nm}\right)\] 
for any $m\in\relint(\lambda)\cap M$. Hence, given that $p_{\lambda 0}=p_{\gamma 0}\circ p_{\lambda\gamma}$ is projective and $p_{\gamma 0}$ is separated, we have that $p_{\lambda\gamma}$ is projective by \cite[\href{https://stacks.math.columbia.edu/tag/0C4Q}{Tag 0C4Q}]{stacks-project}.  

By \cite[\href{https://stacks.math.columbia.edu/tag/0C4Q}{Tag 0C4Q}]{stacks-project}, the morphisms $p_{\lambda}:Y\to Y_{\lambda}$ are projective. Since every $Y_{\lambda}$ is dominated by $W$, all morphisms $p_{\lambda} : Y \to Y_{\lambda}$ are dominant. Together with properness, this implies surjectivity of each $p_{\lambda}$. The same holds for the morphisms $p_{\lambda\gamma}$.

The morphisms $p_{\lambda}:Y\to Y_{\lambda}$ are proper and surjective, then the generic point of $Y$ goes to the generic point of $Y_{\lambda}$. Let us take the Stein factorization 
\[ \xymatrix{ Y \ar[r]_{f} \ar@/^2pc/[rr]^{p_{\lambda}} & Y'_{\lambda} \ar[r]_{g} & Y_{\lambda},}\]
where $Y_{\lambda}'$ is the relative normalization of $Y_{\lambda}$ in $Y$, $g$ is a finite morphism and $f$ is a contraction. Given that $p_{\lambda}$ is surjective, we have that $g$ is also surjective. Let $\nu:Y_{\lambda}'^{\nu}\to Y_{\lambda}$ be the normalization morphism. The morphism $h:=g\circ\nu:Y_{\lambda}'^{\nu}\to Y_{\lambda}$ is finite, because it is the composition of two finite morphisms. By \cite[\href{https://stacks.math.columbia.edu/tag/035I}{Tag 035l}]{stacks-project}, there exists a morphism $r:Y'_{\lambda}\to Y_{\lambda}'^{\nu}$ that fits into the following commutative diagram
\[\xymatrix{ Y \ar[r]^{f^{\nu}} \ar[d]_{f} & Y_{\lambda}'^{\nu} \ar[d]^{h:=g\circ \nu} \ar@/^/[ld]^{\nu}  \\ Y_{\lambda}' \ar[r]_{g} \ar@/^/[ur]^{r} & Y_{\lambda} }\]
and is the normalization of $Y'_{\lambda}$ in $Y_{\lambda}$. Thus, $Y_{\lambda}'=Y_{\lambda}'^{\nu}$ is normal and $g:Y_{\lambda}'\to Y_{\lambda}$ is a finite (integral) morphism. Given that $p_{\lambda}$ is birational and surjective, then $g$ is birational. By \cite[\href{https://stacks.math.columbia.edu/tag/0AB1}{Tag 0AB1}]{stacks-project}, we have that $g$ is an isomorphism. Thus, it follows that $p_{\lambda}$ has geometrically connected fibers. Since the morphism $p_{\lambda}:Y\to Y_{\lambda}$ is a projective surjection with geometrically connected fibers and $Y_{\lambda}$ is a normal variety, it follows that it is a contraction by Stein factorization. The same holds for the morphisms $p_{\lambda\gamma}:Y_{\lambda}\to Y_{\gamma}$, since $p_{\gamma}=p_{\lambda}\circ p_{\lambda\gamma}$ and $Y_{\gamma}$ are normal varieties.
    
    Let $\lambda$ and $\gamma$ in $\Lambda$ such that $\dim(Y_{\lambda})=\dim(Y_{\gamma})=\dim(X)-\dim(T)$. If $\gamma\preceq\lambda$, then $p_{\lambda\gamma}:Y_{\lambda}\to Y_{\gamma}$ is birational and, therefore, induces the identity between the fields of rational functions $k(Y_{\lambda})=k(X)^{T}=k(Y_{\gamma})$. Given that $Y$ can be constructed just by taking the subsystem $Y_{\lambda}$ with $\lambda\cap \relint(\omega^{\vee})\neq \emptyset$, where all the morphisms $p_{\lambda\gamma}$ are birational, we have that $p_{\lambda}$ is birational.

\end{proof}

\subsubsection{Construction of a pp-divisor and proof of \cref{theorem Altmann hausen split}~\eqref{theorem Altmann hausen split part a}.} 
\noindent

The construction above tells us how to construct a normal semi-projective $k$-variety from an affine normal $k$-variety $X$ endowed with an effective action of a split $k$-torus $T$. In the following, we present some results that will help us to construct a pp-divisor $\D\in\PPDiv_{\Q}(Y,\omega)$, where $\omega^{\vee}\subset M_{\Q}$ is the weight cone associated to the $T$-action on $X$. 

Roughly speaking, the strategy is the following: we construct, for every $m \in \omega^{\vee} \cap M$, a sheaf $\mathscr{A}_m$ over $Y$. This allows us to obtain the unique $\Q$-Cartier divisor $\mathfrak{h}(m)$ such that $\mathscr{O}_Y(\mathfrak{h}(m)) = \frac{1}{s(m)}\mathscr{A}_m$, where $s(m)\in k(Y)^{*}$. Consequently, we prove that $\mathfrak{h}$ defines a convex piecewise linear strictly semiample and we conclude by \cref{proposition ppdiv cplmaps}.

Let us give some context before. Recall that, for $\Lambda$ the quasifan associated to $\omega^{\vee}$ as in \cref{theorem GIT-equivalence} and for $\lambda\in\Lambda$, we have  
\[Y_{\lambda}=\mathrm{Proj}(A_{(m)})\textrm{, where }A_{(m)}:=\bigoplus_{n\in\N}A_{nm}\]
and $m$ is any element in $\relint(\lambda)\cap M$. Thus, we can associate to $m$ a sheaf $\mathscr{A}_{\lambda,m}$ on $Y_{\lambda}$ given by 
\[\mathscr{A}_{\lambda,m}:=(q_{\lambda})_{*}(\mathscr{O}_{W_{\lambda}})_{m},\]
where $(\mathscr{O}_{W_{\lambda}})_{m}$ denotes the sheaf of semiinvariants of weight $m$ with respect to the $T$-linearization. The following results are in \cite[Section 6]{AH06} and their proofs adapt directly to this context. 

\begin{lemma}\cite[Lemma 6.3]{AH06}\label{Lemma AH06 63}
Let $\lambda\in\Lambda$ and $m\in\relint(\lambda)\cap M$. For $f\in A_{nm}$, let $Y_{\lambda,f}:=q_{\lambda}(X_{f})$ be the corresponding affine chart of $Y_{\lambda}$.
	\begin{enumerate}[i)]
		\item On $Y_{\lambda,f}$, the sheaf $\mathscr{A}_{\lambda,m}$ is the coherent $\mathscr{O}_{Y_{\lambda}}$-module associated to the $(A_{f})_{0}$-module $(A_{f})_{m}$.
		\item If $m$ is saturated, i.e. the ring $A_{(m)}$ is generated in degree one, then $\mathscr{A}_{\lambda,m}$ is an ample invertible sheaf on $Y_{\lambda}$. On the charts $Y_{\lambda,f}$, where $f\in A_{m}$, we have 
		\[ \mathscr{A}_{\lambda,m} =f\cdot(A_{f})_{0}=f\cdot\mathscr{O}_{Y_{\lambda}}.\]
		\item If $g\in\mathrm{Quot}(A)$ and $n\in\N$, then $g^{n}\in\mathscr{A}_{\lambda,nm}$ implies $g\in\mathscr{A}_{\lambda,m}$.
		\item The space of global sections of $\mathscr{A}_{\lambda,m}$ is $\H^{0}(Y_{\lambda},\mathscr{A}_{\lambda,m})=A_{m}$.
	\end{enumerate}
\end{lemma}

For each $\lambda\in\Lambda$ and $m\in\relint(\lambda)$, we have a coherent sheaf $p_{\lambda}^{*}\mathscr{A}_{\lambda,m}$ with $p_{\lambda}:Y\to Y_{\lambda}$. Since every $m\in \mathrm{Supp}(\Lambda)$ lies in $m\in\relint(\lambda)$ for a unique $\lambda\in \Lambda$, we can set the sheaf as $\mathscr{A}_{m}:=p_{\lambda}^{*}\mathscr{A}_{\lambda,m}$. Thus, for each $m\in\omega^{\vee}\cap M$, we have the coherent sheaf $\mathscr{A}_{m}$ over $Y$. These sheaves satisfy the following.

\begin{lemma}\cite[Lemma 6.4]{AH06}\label{Lemma AH06 lemma 64}
	Let $m,m'\in\omega^{\vee}\cap M$.
	\begin{enumerate}[i)]
		\item We have $k(Y)=\mathrm{Quot}(A)_{0}$, and the natural transformation $p_{\lambda}^{*}q_{\lambda *} \to q_{*}j_{\lambda}^{*}$ sends $\mathscr{A}_{m}$ into $\mathrm{Quot}(A)_{m}$.
		\item Let $m$ be saturated. Then $\mathscr{A}_{m}$ is a globally generated invertible sheaf. On the (not necessarily affine) sets $Y_{f}:=p_{\lambda}^{-1}(Y_{\lambda,f})$ with $f\in A_{m}$, we have 
		\[ \mathscr{A}_{m}=f\cdot \mathscr{O}_{Y}\subset f\cdot k(Y)=\mathrm{Quot}(A)_{m}.  \]
		Moreover, for the global sections of $\mathscr{A}_{m}$, we obtain $\H^{0}(Y,\mathscr{A}_{m})=A_{m}$.
		\item If $m$, $m'$ and $m+m'$ are saturated, then $\mathscr{A}_{m}\mathscr{A}_{m'}\subset \mathscr{A}_{m+m'}$. If, moreover, $m$ and $m'$ lie in a common cone of $\Lambda$, then the equality holds.
	\end{enumerate}
\end{lemma}

\begin{remark}
The proof of the latter lemma relies on \cref{Proposition lemma 6.1 AH06} and \cref{Lemma AH06 63}.
\end{remark}

Now we are ready to prove \cite[Theorem 3.4]{AH06} for every affine normal $k$-variety endowed with an effective action of a split $k$-torus over any field $k$.

We summarize here some of the previous construction: given an affine normal $T$-variety $X:=\Spec(A)$, with $T:=\Spec(k[M])$, its respective $k$-algebra comes with a $M$-graduation
\[
A=\bigoplus_{\omega^{\vee}\cap M} A_{m},
\] 
weight cone $\omega^{\vee}\subset M_{\Q}$ related to the $M$-graduation and . Besides, we have constructed a quotient $Y$ as in \cref{equation: construction Y} and a family of sheaves $\mathscr{A}_{m}$ over $Y$, where $m$ runs over $\omega^{\vee}\cap M$. 

Since $T$ acts effectively on $X$, there exists an homomorphism $s:M\to \mathrm{Quot}(A)^{*}$ such that $s(m)$ is homogeneous of degree $m$ for any $m\in M$. Such a choice is non-canonical and it corresponds to a birational section of the quotient $X\dasharrow Y$.

\begin{proposition}\label{proposition existence ah-datum split}
    Let $T:=\Spec(k[M])$ be a split $k$-torus and $X:=\Spec(A)$ be an affine normal $T$-variety over $k$. Let $s:M\to \mathrm{Quot}(A)^{*}$ be a homomorphism such that $s(m)$ is homogeneous of degree $m$ for any $m\in M$. Then, there exists a pp-divisor $\D$ in $\mathfrak{PPDiv}_{N}(k)$ such that $s(m)\cdot \H^{0}(Y,\mathscr{O}_{Y}(\D(m)))=A_{m}$ and $X\cong X[Y,\D]$ as $T$-varieties.
\end{proposition}

\begin{proof}
    The construction of the pp-divisor, as in \cite[Section 6]{AH06}, follows from a construction of a map $\mathfrak{h}\in\mathrm{CPL}_{\Q}(Y,\omega)$. For each saturated $m\in \omega^{\vee}\cap M$, there exists a unique Cartier divisor $\mathfrak{h}(m)\in\cadiv(Y)$ such that 
    \[\mathscr{O}_{Y}(\mathfrak{h}(m))=\dfrac{1}{s(m)}\cdot\mathscr{A}_{m}\subset k(Y),\]
whose local equation on $Y_{f}$, for $f\in A_{m}$, is $s(m)/f$. If $m\in\omega^{\vee}\cap M$ is not saturated, choose a saturated multiple $nm$ (such a saturated multiple always exists by \cite[Proposition III.1.3]{Bou06}) and define 
\[\mathfrak{h}(m):=\dfrac{1}{n}\cdot\mathfrak{h}(nm)\in\cadiv_{\Q}(Y).\]
This definition does not depend on the choice of $n\in\N$. 

Let $\Lambda$ be the quasifan of \cref{theorem GIT-equivalence}. By \cref{Lemma AH06 lemma 64}, the map $\mathfrak{h}$ is convex and piecewise linear on $\Lambda$. Moreover, given that for $m\in\relint(\lambda)\cap M$ the sheaves $\mathscr{A}_{m}$ are big, then the $\mathfrak{h}(m)$ are big. Therefore, $\mathfrak{h}\in\mathrm{CPL}_{\Q}(Y,\omega)$ and, by \cref{proposition ppdiv cplmaps}, there exists a pp-divisor $\D\in\PPDiv_{\Q}(Y,\omega)$ such that $\mathfrak{h}_{\D}=\mathfrak{h}$. By \cref{Lemma AH06 lemma 64}, we have that $\H^{0}\left(Y,\mathscr{A}_{m} \right)=A_{m}$, therefore
\[s(m)\cdot\H^{0}(Y,\mathscr{O}_{Y}(\D(m)))= \H^{0}\left(Y,\mathscr{A}_{m} \right)=A_{m},\]
if $m\in\omega^{\vee}\cap M$ is saturated. Otherwise, if $m\in\omega^{\vee}\cap M$ is not saturated and $nm$ is a saturated multiple, we have 
\[g\in\H^{0}(Y,\mathscr{O}_{Y}(\D(m)))\Leftrightarrow g^{n}\in\H^{0}(Y,\mathscr{O}_{Y}(\D(nm)))\Leftrightarrow (gs(m))^{n}\in A_{nm}. \]
Given that $A$ is normal, $g\in\H^{0}(Y,\mathscr{O}_{Y}(\D(m)))$ if and only if $gs(m)\in A_{m}$. This defines an isomorphism of $M$-graded $k$-algebras 
\[A[Y,\D]:=\bigoplus_{m\in\omega^{\vee}\cap M}\H^{0}(Y,\mathscr{O}_{Y}(\D(m)))\to \bigoplus_{m\in \omega^{\vee}\cap M}A_{m}=A.\]
Finally we have that there exists a triple $(\omega,Y,\D)$ such that 
\[X=\Spec(A)\cong \Spec(A[Y,\D])=X[Y,\D].\]
This proves the assertion.
\end{proof}

This proposition proves \cref{theorem Altmann hausen split}~\eqref{theorem Altmann hausen split part b}.

\begin{remark}\label{remark choice of section}
A pp-divisor constructed as in \cref{proposition existence ah-datum split} is called \emph{$\mathrm{GIT}$-constructed}. Any $\mathrm{GIT}$-constructed pp-divisor arises from a choice of a homomorphism $s:M\to k(Y)^{*}$ such that $s(m)$ is homogeneous of degree $m$ for any $m\in M$. Such a construction is used in the proof of \cref{lemma 9.2 AH06} and \cref{theorem semilinear automorphisms}.
\end{remark}

\subsubsection{Minimal pp-divisors}
\noindent

 Every affine normal variety endowed with an effective action of a split algebraic torus arises from a pp-divisor over a normal semi-projective variety. There are many pp-divisors encoding the same pair. For example, let $\Delta := [1, +\infty] \subset \Q$, the action 
\begin{align*} 
\G_{\mathrm{m}, k} \times \A_{k}^{2} &\to \A_{k}^{2}, \\
(\lambda, (x, y)) &\mapsto (\lambda x, y)
\end{align*}
is encoded both by $\D_{1} := \Delta \otimes \{0\}$ on $\A^{1}$ and $\D_{2} := \Delta \otimes \{0\} + \emptyset \otimes \{\infty\}$ on $\P^{1}$. However, there is a notion of \textit{minimality} for pp-divisors. Let $\D \in \mathrm{PPDiv}_{\Q}(Y, \omega)$ be a pp-divisor. Given that $\D(m)$ is semiample for every $m \in \omega^{\vee} \cap M$, we have natural morphisms 
\[\vartheta_{m}: Y \to Y_{m} := \mathrm{Proj}\left(\bigoplus_{n \geq 0} \H^{0}(Y, \mathscr{O}_{Y}(\D(nm)))\right)\]
that are contraction maps. Moreover, they are birational if $m \in \relint(\omega^{\vee}) \cap M$.

Denote $X := X(\D)$. We can prove that all the $Y_{m}$ correspond to the GIT-quotients of the semistable subvarieties for the respective linearization of the trivial bundle. Then, all the spaces $Y_{\lambda} := Y_{m}$, with $m \in \relint(\lambda)$ and $\lambda \in \Lambda$ the quasifan in \cref{theorem GIT-equivalence}, can be put into an inverse system compatible with the morphisms $\vartheta_{\lambda}: Y \to Y_{\lambda}$. Hence, we have a projective and birational morphism 
\[\vartheta: Y \to \varprojlim Y_{\lambda}.\]

The scheme $\varprojlim Y_{\lambda}$ comes with a canonical reduced component, which is the schematic image of $q:W\to \varprojlim Y_{\lambda}$ for $W$ the intersection of all semistable subvarieties. The schematic image of $\vartheta:Y\to \varprojlim Y_{\lambda}$ is contained in $q(W)$. 
\begin{definition}
A pp-divisor $\D\in\mathrm{PPDiv}_{\Q}(Y,\omega)$ is said to be minimal if the morphism $\vartheta:Y\to \varprojlim Y_{\lambda}$ is the normalization of the canonical reduced component of $\varprojlim Y_{\lambda}$.
\end{definition}

In particular, a pp-divisor constructed as in \cref{proposition existence ah-datum split} is minimal.

 \begin{proposition}\label{proposition minimality stable under base change}
 Let $Y$ be a normal semi-projective variety over $k$. Let $\D$ be a pp-divisor over $Y$. Then we have that $\D$ is minimal if and only if $\D_{\sep{k}}$ is minimal. 
 \end{proposition}
 
 \begin{proof}
 By definition, $\D\in\PPDiv(Y,\omega)$, with $Y$ a normal semi-projective variety over $k$ and $\omega\subset N_{\Q}$ a cone. The varieties $X(\D)$ and $X(\D_{\sep{k}})$ have the same quasifan decomposition $\Lambda$ for $\omega^{\vee}$. Then, we have the following commutative diagram
 \[\xymatrix{ Y_{\sep{k}} \ar[r]^(0.4){\vartheta_{\sep{k}}} \ar[d] &  \varprojlim (Y_{\sep{k}})_{\lambda} \ar[d] \\ Y \ar[r]_(0.4){\vartheta} &  \varprojlim Y_{\lambda},}\]
 where the vertical arrows correspond to the base change to $\sep{k}$. Denote by $Y'$ (respectively $Y_{\sep{k}}'$) the canonical reduced component of $\varprojlim Y_{\lambda}$ (respectively $\varprojlim (Y_{\sep{k}})_{\lambda}$). They satisfy $(Y_{\sep{k}})'=(Y')_{\sep{k}}$. Besides, the morphism $\vartheta_{\sep{k}}:Y_{\sep{k}}\to Y_{\sep{k}}'$ is the normalization of $Y_{\sep{k}}'$ if and only if $\vartheta:Y\to Y'$ is the normalization of $Y'$. This proves the assertion.
 \end{proof}

\begin{example}\label{exampleA3}
The algebraic torus $\G_{\mathrm{m},k}^{2}$ acts on the three dimensional affine space $\mathbb{A}_{k}^{3}$. Let us consider the action given by 
\[(\lambda,\mu)\cdot (x,y,z)=(\lambda x,\mu y,\lambda\mu z).\]
This action is encoded by the pp-divisor $\D:=\Delta\otimes\{\infty\}$ over $\P_{k}^{1}$, where $\Delta$ is the polyhedron 
\begin{center}
    \begin{tikzpicture}
        \fill[gray] (0,0.7) -- (0,2) -- (2,2) -- (2,0) -- (0.7,0);
        \node[scale=0.6] at (0.7,-0.2) {$(1,0)$};
        \node[scale=0.6] at (-0.3,0.7) {$(0,1)$};
        \draw[->] (0,0) -- (2,0);
        \draw[->] (0,0) -- (0,2);
        \node[scale=0.9] at (0.6,0.6) {$\Delta$};
    \end{tikzpicture}
\end{center} 
\end{example}

\begin{example}\label{exampleSL2ppdivisor}
The algebraic group $\mathrm{SL}_{2,k}$ is a normal variety over $k$ with a $\G_{\mathrm{m},k}^{2}$-structure. Let us consider the action 
\[(\lambda,\mu)\cdot (x,y,z,w)=(\lambda x,\mu y,\mu^{-1}z,\lambda^{-1}w).\]
This action is encoded by the pp-divisor $\D:=\Delta_{1}\otimes [0] +\Delta_{2}\otimes [1]$, where the polyhedra are $\Delta_{1}:=\mathrm{cone}(0,1)$ and $\Delta_{2}:=\cone(1,0)$ as shown in the following picture.
\begin{center}
    \begin{tikzpicture}
        \node[scale=0.6] at (-0.3,0.7) {$\Delta_{1}$};
        \draw[->] (0,0) -- (2,0);
        \draw[gray,->] (0,0) -- (0,2);
        
        \node[scale=0.6] at (3.7,-0.2) {$\Delta_{2}$};
        \draw[gray,->] (3,0) -- (5,0);
        \draw[->] (3,0) -- (3,2);
    \end{tikzpicture}
\end{center} 
\end{example}

\begin{example}\cite[Example 11.1]{AH06}\label{exampleAH06 11 1}
 The affine threefold $X:=\Spec(k[x,y,z,w]/(x^{3}+y^{4}+zw))$ in $\A_{k}^{4}$ with the action of $\G_{\mathrm{m},k}^{2}$ given by 
\[(\lambda,\mu)\cdot (x,y,z,w)=(\lambda^{4} x,\lambda^{3} y,\mu z,\lambda^{12}\mu^{-1}w)\]
is encoded by the pp-divisor $\D:=\Delta_{0}\otimes \{0\}+\Delta_{1}\otimes \{1\}+\Delta_{\infty}\otimes\{\infty\}$, where 
\[\Delta_{0}=\left(\dfrac{1}{3},0\right)+\omega, \quad \Delta_{1}=\left(-\dfrac{1}{4},0\right)+\omega, \quad \Delta_{\infty}=\left(\{0\}\times[0,1]\right)+\omega\]
and $\omega=\cone((1,0),(1,12))$.

\begin{center}
    \begin{tikzpicture}
        \fill[gray] (1/3,0) -- (1/2,2) -- (2,2) -- (2,0) -- (1/3,0);
        \draw[->] (0,0) -- (2,0);
        \draw[->] (0,0) -- (0,2);
        \node[scale=0.9] at (1,0.6) {$\Delta_{0}$};
        
        \fill[gray] (11/4,0) -- (35/12,2) -- (5,2) -- (5,0);
        \node[scale=0.6] at (3.7,-0.2) {$(1,0)$};
        \node[scale=0.6] at (2.7,0.7) {$(0,1)$};
        \draw[->] (3,0) -- (5,0);
        \draw[->] (3,0) -- (3,2);
        \node[scale=0.9] at (3.6,0.6) {$\Delta_{1}$};
        
        \fill[gray] (6,0) -- (6,1) --(73/12,2) -- (8,0) -- (6.7,0);
        \node[scale=0.6] at (6.7,-0.2) {$(1,0)$};
        \node[scale=0.6] at (5.7,0.7) {$(0,1)$};
        \draw[->] (6,0) -- (8,0);
        \draw[->] (6,0) -- (6,2);
        \node[scale=0.9] at (6.6,0.6) {$\Delta_{\infty}$};
    \end{tikzpicture}
\end{center} 

\end{example}

\begin{example}\label{example A3 blow up A2}
 The affine space $\A_{k}^{3}$ with the action of $\G_{\mathrm{m},k}$ given by 
\[\lambda\cdot (x,y,z)=(\lambda x,\lambda y,\lambda^{-1} z)\]
is encoded by the pp-divisor 
\[\D:=\{1\}\otimes D_{(1,0)}+\{0\}\otimes D_{(0,1)}+[0,1]\otimes D_{(1,1)} \in \PPDiv_{\Q}(\mathrm{Bl}_{0}(\A_{k}^{2}),\omega)\]
 where $D_{(1,0)}$, $D_{(0,1)}$ and $D_{(1,1)}$ are the toric invariant divisors of $\mathrm{Bl}_{0}(\A_{k}^{2})$ associated to the rays $\cone(0,1)$ and $\cone(1,1)$, respectively, and $\omega=\cone(0)$.
\end{example}

\begin{remark}
Since all the examples above are computed by following \cite[Section 11]{AH06}, they are all minimal over the separable closure. Thus, they are minimal over the ground field by \cref{proposition minimality stable under base change}. The latter example is of complexity two, so we prove its minimality by following the construction given in \cite[Section 11]{AH06}. As a toric variety, $\A_{k}^{3}$ under the action of $\G_{\mathrm{m},k}^{3}$ coordinatewise is given by the cone 
\[\omega=\cone((1,0,0),(0,1,0),(0,0,1)).\]

The action of $\G_{\mathrm{m},k}$ on $\A_{k}^{3}$ in \cref{example A3 blow up A2} follows from the embedding $\lambda\to(\lambda,\lambda,\lambda^{-1})$ of the respective tori. This embedding, in terms of cocharacter lattices, is equivalent to the morphism $\Z\to \Z^{3}$ given by $a\mapsto(a,a-a)$. This latter morphism fits into the following exact sequence of $\Z$-modules:
\begin{equation}\label{equation: exact lattice remark minimality}\xymatrix{0 \ar[r] & \Z \ar[r]^{F} & \Z^{3} \ar@/^1pc/[l]_{s} \ar[r]^{P} \ar[r] & \Z^{2} \ar[r] & 0},\end{equation} 
where $F(a)=(a,a,-a)$, $P(a,b,c)=(a+c,b+c)$ and $s(a,b,c)=(a)$. This latter map is a section of $F$, which is not unique. Therefore, it is not canonical. Now, we look for the images of the rays of $\omega$ by $P$, which are $P(1,0,0)=(1,0)$, $P(0,1,0)=(0,1)$ and $P(0,0,1)=(1,1)$. The smallest fan in $\Z^{2}$ admitting $(1,0)$, $(0,1)$ and $(1,1)$ as rays is 
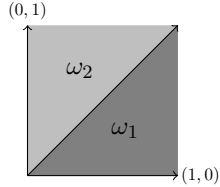
\begin{center}
    \begin{tikzpicture}
        \fill[gray] (0,0) -- (2,0) -- (2,2) -- (0,0);
        \fill[lightgray] (0,0) -- (0,2) -- (2,2) -- (0,0);
        \node[scale=0.6] at (2.3,0) {$(1,0)$};
        \node[scale=0.6] at (0,2.2) {$(0,1)$};
        \draw[->] (0,0) -- (2,0);
        \draw[->] (0,0) -- (0,2);
        \draw[->] (0,0) -- (2,2);
        \node[scale=0.9] at (1.3,0.6) {$\omega_{1}$};
        \node[scale=0.9] at (0.7,1.4) {$\omega_{2}$};
    \end{tikzpicture}
\end{center}
This fan corresponds to the GIT-quotient constructed in \cref{proposition existence ah-datum split} for the $\G_{\mathrm{m},k}$-action. Besides, this fan corresponds to $\mathrm{Blow}_{0}(\A_{k}^{2})$ and each ray corresponds to a toric invariant divisor $D_{(1,0)}$, $D_{(0,1)}$ and $D_{(1,1)}$ of $\mathrm{Blow}_{0}(\A_{k}^{2})$.

Let us now compute the polyhedra. An exact sequence of cocharacter lattice extends to an exact sequence of $\Q$-vector spaces and likewise for the morphisms. Thus, the exact sequence \eqref{equation: exact lattice remark minimality} turns to be
\[\xymatrix{0 \ar[r] & \Q \ar[r]^{F} & \Q^{3} \ar@/^1pc/[l]_{s} \ar[r]^{P} \ar[r] & \Q^{2} \ar[r] & 0}.\]
The polyhedra associated to each toric divisor are computed as 
\[\Delta_{(i,j)}:=s\left(\omega\cap P^{-1}\left(i,j\right)\right),\]
for $i,j\in\{0,1\}$. Thus, 
\[\Delta_{(1,0)}=s\left(\{(1-c,-c,c) \mid -c\geq 0 \textrm{ and } c\geq 0 \}\right)=\{1\},\]
\[\Delta_{(0,1)}=s\left(\{(-c,1-c,c) \mid -c\geq 0 \textrm{ and } c\geq 0 \}\right)=\{0\},\]
and
\[\Delta_{(1,1)}=s\left(\{(1-c,1-c,c) \mid 1-c\geq 0 \textrm{ and } c\geq 0 \}\right)=[0,1].\]
Thus, the corresponding pp-divisor 
\[\D=\{1\}\otimes D_{(1,0)}+\{0\}\otimes D_{(0,1)}+[0,1]\otimes D_{(1,1)} \in \PPDiv_{\Q}(\mathrm{Bl}_{0}(\A_{k}^{2}),\omega)\]
is minimal.
\end{remark}

\subsubsection{Affine $G$-normal varieties}
\noindent

For simplicity, let us assume that $k$ is a perfect field of characteristic $p > 0$. Consider a diagonalizable finite group scheme $G \subset T$, where $T$ is a split algebraic torus over $k$. Let $X$ be an affine $G$-normal variety over $k$ (as defined in \cite[Definition 4.1]{Bri24a}). Brion \cite[Remark 4.3]{Bri24b} provides an alternative perspective on $X$ by considering the affine normal $T$-variety $X^{\#}:=(T \times X)/G$. The action of $G$ on $T$ is given by $g \cdot t = g^{-1} t$, and $G$ acts diagonally on the product. Normality of $X^{\#}$ follows from \cite[Proposition 2.1.3]{Bri26}, and it is accompanied by a $T$-equivariant morphism $\Phi:X^{\#}\to T/G$ induced by the first projection $\mathrm{pr}_{1}:T\times X\to T$. The fiber of $\Phi$ at the base point of $T/G$ is geometrically integral, and it is precisely the $G$-normal curve $X$. This establishes a bijection between the set of $G$-normal varieties $X$ and the set of normal $T$-varieties $Z$ equipped with a $T$-equivariant map $\Phi:Z\to T/G$ such that the fiber at the base point is geometrically integral. Therefore, $G$-normal varieties provide a big family of examples of normal $T$-varieties.

Both $X$ and $X^{\#}$ allow for the construction of good quotients $X\sslash G$ and $X^{\#}\sslash T$, respectively, which are canonically isomorphic. Since an affine $G$-normal variety $X$ gives rise to an affine normal $T$-variety $X^{\#}$, we can express $X$ in terms of a pp-divisor $\D\in\PPDiv_{\Q}(X\sslash G,\omega)$ by \cref{theoremmainaltmannhausensplitintro}. Notably, the tail cone associated to a pp-divisor encoding an affine $G$-normal variety is always $\{0\}$, as the morphism $\Phi$ is $T$-equivariant, surjective, and $T/G$ has the same dimension as $T$. Let us now follow the example provided in \cite[Example 2.1.11]{Bri26}.

\begin{example}\label{example: example 4.13 brion}
 Let $G:=\mu_{q}=\Spec(k[T]/(T^{q}-1))$ and $q=p^{r}$ with $r\geq 1$. Let us consider the action of $\mu_{q}$ on $\A_{k}^{n+1}$ given by $t\cdot (x_{1},\dots,x_{n},y)=(x_{1},\dots,x_{n},ty)$. Thus, the affine variety $X\subset \A_{k}^{n+1}$ given by the polynomial $y^{q}-f(x_{1},\dots,x_{n})$ is $\mu_{q}$-stable. Moreover, $X$ is $\mu_{q}$-normal whenever $f(x_{1},\dots,x_{n})$ is square-free. Hence, its respective affine normal $T$-variety is 
	\[X^{\#}=\Spec\left(\frac{k[x_{1},\dots,x_{n},z,a,b]}{(z^{q}-a f(x_{1},\dots,x_{n}),ab-1)}\right),\]
with the $T$-action given by $\lambda\cdot(x_{1},\dots,x_{n},z,a,b)=(x_{1},\dots,x_{n},\lambda  z,\lambda^{q}a,\lambda^{-q}b)$ and whose quotient by $T$ corresponds to $X^{\#}\sslash T=\Spec(k[x_{1},\dots,x_{n}])$. Besides, the $T$-equivariant morphism $\Phi$ is given by 
    \begin{align*}
		\Phi:X^{\#} &\to T/G \\
		(x_{1},\dots,x_{n},z,a,b) &\mapsto (a,b).
	\end{align*} 
Hence, the $G$-normal variety $X$ arises from the pp-divisor $\D$ in $\PPDiv_{\Q}(\A_{k}^{n},\{0\})$ given by $\left[1/q\right]\otimes \mathrm{div}(f)$.
\end{example}

It remains to translate the $T$-equivariant morphism $\Phi:X^{\#}\to T/G$ into a combinatorial form. Since $\Phi$ is surjective and $T$-equivariant, it induces an injective $M$-graded morphism of $k$-algebras that preserves degrees:
\[
\Phi^{*}: k[M^G] \longrightarrow A[X\sslash G,\D].
\]
This problem is studied in \cite{MNT26} in the specific case where $X$ is a curve and $\mu_{d}=G\subset\G_{\mathrm{m}}$ with $d\geq 2$. In this case, it is proven that the $T$-equivariant morphism $\Phi:X^{\#}\to T/G$ corresponds to the choice of an element $\alpha \in A[X\sslash G,\D]^{*}\cap \H^{0}(X\sslash G,\mathscr{O}_{X\sslash G}(\D(d)))$ known as a $\Phi$-\emph{structure}. Therefore, a $G$-normal curve $X$ corresponds to a $\Phi$-\emph{pair} $(\D,\alpha)$ over $X\sslash G$, which is a pair of a pp-divisor $\D\in\PPDiv_{\Q}(X\sslash G,\{0\})$ and a $\Phi$-structure $\alpha$ on $\D$. However, not every $\Phi$-pair $(\D,\alpha)$ over an affine smooth curve $Y$ gives rise to an affine $G$-normal curve. Indeed, such pairs are called \emph{geometrically integral} $\Phi$-pairs (see \cite[Definition 4.7]{MNT26}). The following corresponds to \cite[Proposition 4.2]{MNT26} and \cite[Proposition 4.8]{MNT26}.

\begin{proposition}\label{lemma invertible element}
\begin{enumerate}[a)]
	Let $G\cong\mu_{d}\subset \G_{\mathrm{m}}$, where $d\geq 2$.
	\item Let $X$ be an affine $G$-normal curve. Then there exists a geometrically integral $\Phi$-pair $(\D,\alpha)$ over $X\sslash G$ such that $\Phi_{\alpha}^{-1}(\{1\})\cong X$.
	\item Conversely, given a geometrically integral $\Phi$-pair $(\D,\alpha)$ over an affine smooth curve $Y$, there exists an affine $G$-normal curve $X$ such that $X\sslash G\cong \mathrm{Loc}(\D)$, $X^{\#}\cong X(\D)$ and $\Phi=\Phi_{\alpha}$.
\end{enumerate}
\end{proposition}

The following result serves as a valuable tool for comprehending the combinatorial structure of $G$-normal varieties. For instance, it is employed in the proof of \cref{lem: description contracted pp-div} presented below.

\begin{lemma}\label{lemma addition of polyhedra}
Let $N$ be a lattice and let $\Delta_{1},\Delta_{2}\in \mathrm{Pol}(N_{\Q},\{0\})$. 
If $\Delta_{1}+\Delta_{2}=\{v\}$ for some $v\in N_{\Q}$, then $\Delta_{1}=\{v_{1}\}$ and $\Delta_{2}=\{v_{2}\}$ with $v_{1}+v_{2}=v$.
\end{lemma}

\begin{proof}
Assume first that $\Delta_{1}$ is not a singleton. 
Then for any pair $v_{1},v_{1}'\in \Delta_{1}$ we have
$[v_{1},v_{1}']:=\{\lambda v_{1}+(1-\lambda)v_{1}' \mid \lambda\in[0,1]\}\subset \Delta_{1}$.
If $v_{1}\neq v_{1}'$, then for any $w\in \Delta_{2}$ we have $[v_{1},v_{1}']+w \subset \Delta_{1}+\Delta_{2}$.
However, since $\Delta_{1}+\Delta_{2}=\{v\}$, this implies $[v_{1},v_{1}']+w\subset \{v\}$, which is impossible. 
Thus $\Delta_{1}$ must be a singleton. By the same argument, $\Delta_{2}$ is also a singleton. 
Finally, if $\Delta_{1}=\{v_{1}\}$ and $\Delta_{2}=\{v_{2}\}$, then $v_{1}+v_{2}=v$.
\end{proof}

\begin{proposition}\label{lem: description contracted pp-div}
Let $X$ be an affine $G$-normal variety, $X^{\#}$ its associated normal $T$-variety, and let $\D\in\PPDiv_{\Q}(X\sslash G,\{0\})$ such that $U^{\#}\simeq X(\D)$ as $T$-varieties. Then
\[
\D = \sum \{v_D\} \otimes D,
\]
where $D$ runs over the prime divisors of $Y$ and $v_D \in N_\Q$.
\end{proposition}

\begin{proof}
Since $k$ is perfect, we can reduce ourselves to the case where $k$ is algebraically closed by \cref{proposition covariant base change}. Write the pp-divisor as 
\[
\Delta = \sum \Delta_{D}\otimes D
\] 
and let $\Phi:U^{\#} \longrightarrow T/G$
be the corresponding $T$-equivariant morphism. Since the morphism $\Phi:X^{\#}\to T/G$ is $T$-equivariant, it follows that the $T$-orbit of every $k$-rational point of $X^{\#}$ is of the same dimension as $T$ and, therefore, $T$-orbits of $k$-rational points $X^{\#}$ are closed. 

Let $y\in X\sslash G$ be a closed point, the set-theoretic and scheme-theoretic fiber $\pi^{-1}(\{y\})$ is a single $T$-orbit. On the other hand, the number of orbits in $\pi^{-1}(\{y\})$ equals the number of faces of the fiber polyhedron $\Delta_{y}$ (see \cite[Corollary 7.11]{AH06}, valid in any characteristic for a $k$-rational point).

Given that the tail cone of $\Delta$ is $\{0\}$, there exist $v_{y}^{-},v_{y}^{+}\in N_{\Q}$ such that 
\[
\Delta_{y} = \{v_{y}^{-}\lambda + v_{y}^{+}(1-\lambda) \mid \lambda \in [0,1]\cap \Q\}.
\] 
If $v_{y}^{-}\neq v_{y}^{+}$, $\Delta_{y}$ has two faces, so $\pi^{-1}(\{y\})$ would contain at least two distinct $T$-orbits, a contradiction. Thus, $v_{y}^{-} = v_{y}^{+}$, so $\Delta_{y}$ is a singleton.

By \cref{lemma addition of polyhedra}, this implies that each $\Delta_{D}$ is a singleton for all $y\in X\sslash G$. Hence, all polyhedra $\Delta_{D}$ in $\Delta$ are singletons, proving the assertion.

\end{proof}

\section{Functoriality and semilinear morphisms}\label{chapter semilinear morphisms}
\indent

In \cref{section: semilinear pp-div}, we present the notion of \textit{semilinear morphisms of pp-divisors}. Then we focus in \cref{section: semilinear equivariant morphisms} on how these morphisms are related to the semilinear equivariant morphisms between their respective varieties.

In order to do this, we study first the functoriality of the Altmann-Hausen construction in \cref{section: functoriality}. And for the convenience of the reader, we recall the definition of semilinear morphisms in \cref{section: semilinear varieties}.

\subsection{ Functoriality of the Altmann-Hausen construction}\label{section: functoriality}
\indent

As stated in \cref{section Polyhedral Divisors}, proper polyhedral divisors form a category. Besides, by \cref{theoremmainaltmannhausensplitintro}, there is an assignment $\D\mapsto X(\D)$ from pp-divisors to affine normal varieties endowed with an effective torus action. This assignment actually defines a functor $X:\mathfrak{PPDiv}(k)\to\mathcal{E}(k)$, where $\mathcal{E}(k)$ stands for the category of affine normal varieties endowed with an effective action of a split algebraic torus over $k$ and whose morphisms are equivariant morphisms of varieties over $k$. In order to prove this statement, we need to explain how the assignment works on morphisms.

Let $\D$ and $\D'$ be two objects in $\mathfrak{PPDiv}(k)$ and $(\psi,F,\mathfrak{f}):\D'\to\D$ be a morphism of pp-divisors over $k$. This morphism induces a morphism of modules given by 
\begin{align*}
   \H^{0}(Y,\mathscr{O}_{Y}(\D(m))) &\to\H^{0}(Y',\mathscr{O}_{Y'}(\D'(F^{*}(m)))), \\ h  &\mapsto \mathfrak{f}(m)\psi^{*}(h),
\end{align*} 
compatible with the $\H^{0}(Y,\mathscr{O}_{Y})$ and $\H^{0}(Y',\mathscr{O}_{Y'})$-module structures. Hence, all these morphisms fit together into a graded morphism 
\[A[Y,\D]=\bigoplus_{m\in \omega^{\vee}\cap M}\H^{0}(Y,\mathscr{O}_{Y}(\D(m)))\to \bigoplus_{m\in \omega'^{\vee}\cap M'}\H^{0}(Y',\mathscr{O}_{Y'}(\D'(m)))=A[Y',\D'],\]
which turns into an equivariant morphism 
\[X(\psi,F,\mathfrak{f}):=(\varphi,f):X(\D)\to X(\D'), \] 
where $\varphi:T'\to T$ is determined by $F:N'\to N$.

\begin{proposition}\label{Proposition X is faithful}
The assignment $\D\mapsto X(\D)$ defines a faithful covariant functor $X:\mathfrak{PPDiv}(k)\to\mathcal{E}(k)$.
\end{proposition}

\begin{proof}
We prove the compatibility with compositions. Let $\D$, $\D'$ and $\D''$ be objects in $\mathfrak{PPDiv}(k)$. Let $(\psi,F,\mathfrak{f}):\D'\to\D$ and $(\psi',F',\mathfrak{f}'):\D''\to\D'$ be morphisms of pp-divisors. By definition, the composition in $\mathfrak{PPDiv}(k)$ is given by
\[(\psi,F,\mathfrak{f})\circ(\psi',F',\mathfrak{f}')=(\psi\circ\psi',F\circ F',F_{*}(\mathfrak{f}')\cdot \psi'^{*}(\mathfrak{f})).\]
The equivariant morphism $X(\psi,F,\mathfrak{f})$ corresponds to the morphism of modules given by 
\begin{align*}
   \H^{0}(Y,\mathscr{O}_{Y}(\D(m))) &\to\H^{0}(Y',\mathscr{O}_{Y'}(\D'(F^{*}(m)))), \\ h  &\mapsto \mathfrak{f}(m)\psi^{*}(h),
\end{align*}
and $X(\psi',F',\mathfrak{f}')$ corresponds to
\begin{align*}
   \H^{0}(Y',\mathscr{O}_{Y'}(\D'(m))) &\to\H^{0}(Y'',\mathscr{O}_{Y''}(\D''(F'^{*}(m)))), \\ h  &\mapsto \mathfrak{f}'(m)\psi'^{*}(h).
\end{align*}
Therefore, the composition induces the following morphism on the modules 
\[\begin{array}{ccccl}
\H^{0}(Y,\mathscr{O}_{Y}(\D(m))) &\to &\H^{0}(Y',\mathscr{O}_{Y'}(\D'(F^{*}(m)))) &\to &\H^{0}(Y'',\mathscr{O}_{Y''}(\D''(F'^{*}(F^{*}(m))))) \\
h  &\mapsto & \mathfrak{f}(m)\psi^{*}(h) &\mapsto & \mathfrak{f}'(F^{*}(m))\psi'^{*}(\mathfrak{f}(m)\psi^{*}(h)) \\
& & & = & [\mathfrak{f}'(F^{*}(m))\psi'^{*}(\mathfrak{f}(m))]\psi'^{*}(\psi^{*}(h)) \\
& & & = & [F_{*}(\mathfrak{f}')\cdot \psi'^{*}(\mathfrak{f})](m)(\psi\circ\psi')^{*}(h) ,
\end{array}\]
which coincides with the morphism induced by $(\psi\circ\psi',F\circ F',F_{*}(\mathfrak{f}')\cdot \psi'^{*}(\mathfrak{f}))$. Hence, both define the same graded morphism between the graded algebras $A[Y'',\D'']$ and $A[Y,\D]$ and, therefore, 
 \begin{align*}
 X((\psi,F,\mathfrak{f})\circ(\psi',F',\mathfrak{f}')) &=X(\psi\circ\psi',F\circ F',F_{*}(\mathfrak{f}')\cdot \psi'^{*}(\mathfrak{f})) \\ &=X(\psi,F,\mathfrak{f})\circ X(\psi',F',\mathfrak{f}').
 \end{align*}  
This proves that the assignment is a covariant functor. It remains to prove that it is faithful.

Let $\D$ and $\D'$ be two objects in $\mathfrak{PPDiv}(k)$. Let $(\psi_{1},F_{1},\mathfrak{f}_{1})$ and $(\psi_{2},F_{2},\mathfrak{f}_{2})$ be two semilinear morphisms of pp-divisors from $\D'\to\D$ such that $X(\psi_{1},F_{1},\mathfrak{f}_{1})=X(\psi_{2},F_{2},\mathfrak{f}_{2})=(\varphi_{},f_{})$.

Notice that if $\psi_{1}^{*},\psi_{2}^{*}:k(Y)\to k(Y')$ are equal, then $\psi_{1}=\psi_{2}$. Given that $k(Y)=k(X)^{T}$, every rational function $f\in k(Y)$ can be written as a quotient of $g$ and $h$ in $\H^{0}(Y,\mathscr{A}_{m})$ for some $m\in M$. Hence,
\[\psi_{1}^{*}(f)=\psi_{1}^{*}\left(\frac{g}{h}\right)=\frac{\mathfrak{f}_{1}(m)\psi_{1}^{*}\left(g\right)}{\mathfrak{f}_{1}(m)\psi_{1}^{*}\left(h\right)}=\frac{\mathfrak{f}_{2}(m)\psi_{2}^{*}\left(g\right)}{\mathfrak{f}_{2}(m)\psi_{2}^{*}\left(h\right)}=\psi_{2}^{*}\left(\frac{g}{h}\right)=\psi_{2}^{*}(f),\]
where the central equality follows from the fact that both morphisms define the same morphism between the graded algebras. Thus, it follows that $\psi_{1}=\psi_{2}=:\psi$. 

Given that $(\psi,F_{1},\mathfrak{f}_{1})$ and $(\psi,F_{2},\mathfrak{f}_{2})$ define the same morphism of graded algebras, we have that $\mathfrak{f}_{1}(m)\psi^{*}(h)=\mathfrak{f}_{2}(m)\psi^{*}(h)$ for every $m\in\omega^{\vee}\cap M$ and $h\in\H^{0}(Y,\mathscr{O}_{Y}(\D(m)))$. Hence, $\mathfrak{f}_{1}=\mathfrak{f}_{2}$.

In order to prove $F_{1}=F_{2}$, it suffices to find a point $x\in X$ such that $f(x)\in X'$ has a trivial isotropy group, i.e. $T'_{f_{}(x)}=\{1_{T'}\}$. By \cref{proposition covariant base change}, we can assume that $k$ is separably closed. Let $x'\in X'$ be a $k$-rational point such that its isotropy group is trivial, for example a general point whose orbit cone is $\omega_{\D'}$. Notice that such a point always exists, since $X'(k)$ is a Zariski dense subset. By \cref{proposition orbits}, we have that $T'_{x'}=\{1_{T'}\}$ is equivalent to $M(x')=M'$, where $M(x')$ is the orbit lattice of $x'$. Let $\{m_{1},\dots,m_{r}\}\subset S(x')$ be a set of generators of the orbit monoid $S(x')$. By definition, for every $i\in\{1,\dots,r\}$, there exists $f_{m_{i}}\in A_{m_{i}}$ such that $f_{m_{i}}(x')\neq 0$. Define 
\[U:=\bigcap_{i=1}^{r}D_{f_{m_{i}}}.\]
Notice that, for every $x''\in U$, we have that $S(x')\subset S(x'')$. Then, we have that $M(x')\subset M(x'')\subset M'$. This implies that $M(x'')=M'$. Otherwise stated, all the elements of $U$ have trivial isotropy group. Finally, given that $f_{}$ is dominant and $U\subset X'$ is open, we have that there exists $x\in X$ such that $f_{}(x)$ has a trivial isotropy group. Then, the assertion holds.
\end{proof}

Let $T$ be a split algebraic torus over $k$ and $N$ be its cocharacter lattice. Recall that $\mathfrak{PPDiv}_{N}(k)$ is the full subcategory of all pp-divisors over $k$ whose tail cone is defined on $N_{\Q}$. Likewise, denote by $\mathcal{E}_{T}(k)$ the full subcategory of all affine normal $T$-varieties. By \cref{theorem Altmann hausen split}, for $\D$ an object in $\mathfrak{PPDiv}_{N}(k)$ we have that $X(\D):=\Spec(A[Y,\D])$ is an affine normal $T$-variety over $k$. Then, the functor $X:\mathfrak{PPDiv}_{}(k)\to\mathcal{E}(k)$ induces a functor
\begin{align*}
X:\mathfrak{PPDiv}_{N}(k) &\to \mathcal{E}_{T}(k),\\ 
\D &\mapsto X(\D).\end{align*}

\begin{corollary}\label{corollary X is faithful}
   The functor $X:\mathfrak{PPDiv}_{N}(k)\to \mathcal{E}_{T}(k)$ is faithful and covariant.
\end{corollary}

As stated in \cref{Proposition X is faithful}, the functor $X:\mathfrak{PPDiv}(k)\to\mathcal{E}(k)$ is faithful, but it is not full. For example, let $\D\in\PPDiv_{\Q}(\P_{k}^{2},\omega)$ be any pp-divisor and $\kappa:F_{1}\to \P_{k}^{2}$ a birational morphism from the Hirzebruch surface to the projective plane. By pulling back, we have $\kappa^{*}\D\in\PPDiv_{\Q}(F_{1},\omega)$. Both pp-divisors define the same normal $T$-variety, then we have the identity map $(\id_{T},\id):X(\D)\to X(\psi^{*}\D)$. However, this map does not arise from a morphism of pp-divisors, because that would imply that there exists a non-constant morphism $\tilde{\kappa}:\P_{k}^{2}\to F_{1}$ such that 
\[(\kappa,\id,\mathfrak{1})\circ(\tilde{\kappa},\id,\mathfrak{1})=(\id_{\P_{k}^{2}},\id,\mathfrak{1}),\]
which gives a contradiction. Thus, not every dominant equivariant morphism between two fixed affine normal varieties endowed with an effective action of a split algebraic torus arises from a morphism of a pair of fixed pp-divisors. Actually, the morphism above arises rather from a pair of morphisms
\[\xymatrixcolsep{4pc}\xymatrix{ \D & \kappa^{*}\D \ar[l]_-{(\kappa,\id_{N},1)} \ar[r]^-{(\id,\id_{N},1)} & \kappa^{*}\D .}\]

Let us call a morphism of pp-divisors $(\psi,F,\mathfrak{f})$ \textit{dominating} if $X(\psi,F,\mathfrak{f})$ is dominant. By \cite[Theorem 8.8]{AH06}, when $k=\bar{k}$ and $k$ is of characteristic zero, every dominant equivariant morphism of affine normal $T$-varieties arises from a pair of dominating morphism of pp-divisors, i.e. from a data 
\[\xymatrixcolsep{4pc}\xymatrix{ \D & \kappa^{*}\D \ar[l]_-{(\kappa,\id_{N},1)} \ar[r]^-{(\psi,F,\mathfrak{f})} & \kappa^{*}\D },\]
where $(\psi,F,\mathfrak{f})$ is a dominating morphism of pp-divisors and $\kappa$ is a projective birational morphism from a normal semi-projective variety.
 In the following we will prove a more general result involving \textit{semilinear} morphisms. These morphisms form a larger family than morphisms of varieties over $k$.

\subsection{ Semilinear morphism of varieties}\label{section: semilinear varieties}
\indent

Recall that $L/k$ stands for a finite Galois extension with Galois group $\Gamma_{L}$. These assumptions are kept for the subsequent sections.

Semilinear morphisms seem to be the right language to deal with Galois descent problems. These morphisms have been used, for instance, by Huruguen \cite{Hur11} and Borovoi \cite{Bor20}.

For any $\gamma\in\Gamma_{L}$, we define the \emph{twist $S^{\gamma}$ of $S$ with respect to $\gamma$} as the pullback 
\[\xymatrix{S^{\gamma} \ar[r]^{\beta_{\gamma}} \ar[d] & S \ar[d] \\ \Spec(L) \ar[r]_{\gamma^{\natural}} & \Spec(L),}\]
where $\gamma^{\natural}:=\Spec(\gamma^{-1})$ and $\beta_{\gamma}$ is the respective projection. Morphisms can also be \emph{twisted}. If $f:S_{1}\to S_{2}$ is a morphism of schemes over $L$, the \emph{twist of $f$ with respect to $\gamma$}, denoted by $f^{\gamma}:S_{1}^{\gamma}\to S_{2}^{\gamma}$, is the pullback of $f$ via $\gamma^{\natural}$.

Whenever $S$ is a group scheme over $L$, then $S^{\gamma}$ is also a group scheme and the morphism $\beta_{\gamma}$ respects the group scheme structure.
\begin{definition}
Let $S_{1}$ and $S_{2}$ be schemes over $L$ and $\gamma\in\Gamma_{L}$. A \textit{semilinear morphism with respect to} $\gamma$ is a morphism of schemes $h_{\gamma}:S_{1}\to S_{2}$ satisfying the following commutative diagram
\[ \xymatrix{ S_{1} \ar[r]^{h_{\gamma}} \ar[d] & S_{2} \ar[d] \\ \Spec(L) \ar[r]_{\gamma^{\natural}} & \Spec(L) ,} \]
where $\gamma^{\natural}:=\Spec(\gamma^{-1})$. Moreover, we say that $h_{\gamma}$ is a \textit{semilinear isomorphism} if $h_{\gamma}$ is an isomorphism of schemes. 
\end{definition}

We say that a semilinear morphism $h_{\gamma}$ is \textit{dominant} if $h_{\gamma}$ is dominant as a morphism of schemes.

\begin{remark}
To give a semilinear morphism $h_{\gamma}:S_{1}\to S_{2}$ with respect to $\gamma$ is equivalent to giving a morphism $h_{\gamma}':S_{1}\to S_{2}^{\gamma}$ of schemes over $L$ such that $h_{\gamma}=h_{\gamma}'\circ\beta_{\gamma}$.
\end{remark}

The main reason to choose the twist $S^{\gamma}$ defined via $\gamma^{\natural}$ instead of $\gamma^{*}:=\Spec(\gamma)$ is to keep the coherence or \emph{covariance} in the following sense: for any pair $\gamma_{1}$ and $\gamma_{2}$ of elements in $\Gamma_{L}$ and any pair of semilinear morphisms $h_{\gamma_{1}}:S_{1}\to S_{2}$ and $g_{\gamma_{2}}:S_{2}\to S_{3}$ of schemes over $L$, the composition 
\[
\xymatrix{ 
S_{1} \ar[r]^{h_{\gamma_{1}}} \ar@/^2pc/[rr]^{g_{\gamma_{2}}\circ h_{\gamma_{1}}} \ar[d] & S_{2} \ar[r]^{g_{\gamma_{2}}} \ar[d] & S_{3} \ar[d] \\
\Spec(L) \ar[r]_{\gamma_{1}^{\natural}} \ar@/_2pc/[rr]_{(\gamma_{2}\cdot\gamma_{1})^{\natural}} & \Spec(L) \ar[r]_{\gamma_{2}^{\natural}} & \Spec(L)  
}
\] 
is indeed a semilinear morphism with respect to $\gamma_{2}\cdot \gamma_{1}$. 

Clearly, any morphism of varieties over $L$ is a semilinear morphism with respect to the neutral element of the Galois group. Then, if we denote by $\SAut(Z)$ the group of semilinear automorphisms of a variety $Z$ over $L$, there is an exact sequence
\begin{equation}\label{Sequence saut aut}
1\to \Aut(Z)\to\SAut(Z)\to \Gamma_{L}.
\end{equation}

Let $G_{1}$ and $G_{2}$ be algebraic groups over $L$ and $\gamma\in\Gamma_{L}$. A \textit{semilinear group homomorphism with respect to} $\gamma$ is a morphism of schemes $\varphi_{\gamma}:G_{1}\to G_{2}$ such that the respective morphism $\varphi_{\gamma}':G_{1}\to G_{2}^{\gamma}$ is a morphism of group schemes over $L$. Moreover, we say that $\varphi_{\gamma}$ is a \textit{semilinear group isomorphism} if $\varphi_{\gamma}'$ is an isomorphism of group schemes. We denote by $\SAut_{\mathrm{gp}}(G)$ the group of such automorphisms for a fixed group-scheme $G$.

Let $Z_{1}$ be a $G$-variety and $Z_{2}$ be a $G_{2}$-variety. Let $\gamma\in\Gamma_{L}$. A \textit{semilinear equivariant morphism with respect to} $\gamma$ is a pair $(\varphi_{\gamma},h_{\gamma})$ such that $\varphi_{\gamma}:G_{1}\to G_{2}$ is a semilinear group homomorphism, $h_{\gamma}:Z_{1}\to Z_{2}$ is a semilinear morphism, both with respect to $\gamma$, and the following diagram of semilinear morphisms commutes 
\[\xymatrix{ G_{1}\times Z_{1} \ar[r]^{\mu_{1}} \ar[d]_{(\varphi_{\gamma},h_{\gamma})} & Z_{1} \ar[d]^{h_{\gamma}} \\ G_{2}\times Z_{2} \ar[r]_{\mu_{2}} & Z_{2},} \]
where $\mu_{1}$ and $\mu_{2}$ are the respective actions of $G_{1}$ on $Z_{1}$ and $G_{2}$ on $Z_{2}$.

For a given $G$-variety $Z$, the group of semilinear equivariant automorphisms over $L$ is denoted by $\SAut_{G}(Z)$, which naturally contains $\Aut_{G}(Z)$. Define $\SAut(G;Z)$ as the subgroup of $\SAut_{\mathrm{gp}}(G)\times\SAut_{G}(Z)$ defined as the preimage of the diagonal inclusion $\Gamma_{L}\to\Gamma_{L}\times\Gamma_{L}$. We have then the following exact sequence
\[1\to\Aut_{\mathrm{gp}}(G)\times\Aut(Z)\to\SAut(G;Z)\to\Gamma_{L},\]
where the map at the right is defined as $(\varphi_{\gamma},h_{\gamma})\mapsto \gamma$.
 
\begin{definition}\label{definition semilinear equivariant action}
Let $G$ be an algebraic group over $L$ and $Z$ be a $G$-variety over $L$. Let $H$ be an abstract group. A \textit{semilinear equivariant action of $H$ on $Z$} is a group homomorphism $\varrho:H\to \SAut(G;Z)$. If $H=\Gamma_{L}$ and $\varrho$ is a section of the exact sequence above, then $\varrho$ is a \textit{Galois semilinear equivariant action}. In the particular case when $G$ is trivial, these are called \textit{semilinear action of $H$ on $Z$}, or a \textit{$H$-semilinear action on $Z$}, and \textit{Galois semilinear action} respectively.
\end{definition}

\subsection{ Semilinear morphisms of pp-divisors}\label{section: semilinear pp-div}
\indent

In \cref{section: functoriality}, we saw that there is a covariant functor $X:\mathfrak{PPDiv}(L)\to\mathcal{E}(L)$, which is faithful but not full. In this section, we consider a \emph{bigger} category. The objects remain the same, but a larger family of morphisms is being considered.

\begin{definition}\label{definition semilinear mpp-div}
 Let $\D$ and $\D'$ be in $\mathfrak{PPDiv}_{}(L)$, the category of pp-divisors over $L$. Let $\gamma\in\Gamma_{L}$. A \textit{semilinear morphism of pp-divisors with respect to $\gamma$} is a triple $(\psi_{\gamma},F,\mathfrak{f}):\D\to \D'$, where $\psi_{\gamma}:Y\to Y'$ is a semilinear dominant morphism, $F:N\to N'$ is a morphism of lattices such that $F(\mathrm{Tail}(\D))\subset \mathrm{Tail}(\D')$ and $\mathfrak{f}\in L(N',Y)^{*}$ is a plurifunction such that 
\[\psi_{\gamma}^{*}(\D')\leq F_{*}(\D)+\mathrm{div}(\mathfrak{f}).\]
\end{definition}

Let $(\psi_{\gamma},F,\mathfrak{f}):\D\to\D'$ be a semilinear morphism of pp-divisors over $L$. For every $m\in M'$, we have a morphism of modules over $\H^{0}(Y',\mathscr{O}_{Y'})$ and $\H^{0}(Y,\mathscr{O}_{Y})$ respectively, (notice that in this case it is only $k$-linear)
\begin{align*}
\H^{0}(Y',\mathscr{O}_{Y'}(\D'(m))) &\to\H^{0}(Y,\mathscr{O}_{Y}(\D(F^{*}(m)))), \\ 
h &\mapsto \mathfrak{f}(m)\psi_{\gamma}^{*}(h)
\end{align*}
that fit together into a morphism of $M$-graded $L$-algebras satisfying the following commutative diagram
\[\xymatrix{A[Y',\D'] \ar[r]& A[Y,\D]  \\ L \ar[r]_{\gamma^{-1}} \ar[u] & L \ar[u] \, ,}\]
which gives a semilinear equivariant morphism 
\[\xymatrixcolsep{12pc}\xymatrix{X(\D) \ar[r]^{X(\psi_{\gamma},F,\mathfrak{f})=(\varphi_{\gamma},f_{\gamma})}& X(\D')  \\ \Spec(L) \ar[r]_{\gamma^{\natural}} \ar[u] & \Spec(L) \ar[u] \, .}\]
Thus, semilinear morphisms of pp-divisors induce semilinear equivariant morphisms of affine normal varieties with a split torus action on $L$. As in the case of morphisms of pp-divisors, let us call \textit{dominating} those semilinear morphisms of pp-divisors inducing dominant semilinear equivariant morphisms. Denote by $\mathfrak{PPDiv}(L/k)$ the category of pp-divisors over $L$ with dominating semilinear morphisms and by $\mathcal{E}(L/k)$ the category of affine normal varieties over $L$ endowed with an effective torus action and whose morphisms are dominant semilinear equivariant morphisms of varieties over $L$. In this setting, there is also a functor $X:\mathfrak{PPDiv}(L/k)\to \mathcal{E}(L/k)$, sending semilinear morphisms of pp-divisors to semilinear equivariant morphisms.

\begin{proposition}\label{proposition faithful functor}
The assignment $\D\mapsto X(\D)$ induces a faithful covariant functor $X:\mathfrak{PPDiv}(L/k)\to \mathcal{E}(L/k)$.
\end{proposition}

\begin{proof}
The proof that the assignment is a functor is analogous to the proof of \cref{Proposition X is faithful} and the functor is covariant by construction.

Let $\D$ and $\D'$ be two objects in $\mathfrak{PPDiv}(L/k)$. Let $(\psi_{\gamma,1},F_{\gamma,1},\mathfrak{f}_{\gamma,1})$ and $(\psi_{\eta,2},F_{\eta,2},\mathfrak{f}_{\eta,2})$ be semilinear morphisms of pp-divisors from $\D'\to\D$ such that 
\[
X(\psi_{\gamma,1},F_{\gamma,1},\mathfrak{f}_{\gamma,1})=X(\psi_{\eta,2},F_{\eta,2},\mathfrak{f}_{\eta,2})=(\varphi_{\gamma},f_{\gamma}).
\] 
First, given that both define the same semilinear equivariant morphism, it follows that $\gamma=\eta$.

The proof of $\psi_{\gamma,1}=\psi_{\gamma,2}$, $F_{\gamma,1}=F_{\gamma,2}$ and $\mathfrak{f}_{\gamma,1}=\mathfrak{f}_{\gamma,2}$ is adapted straightforwardly from the proof of \cref{Proposition X is faithful}. 
\end{proof}

\subsection{ Semilinear equivariant morphisms }\label{section: semilinear equivariant morphisms}
\indent

As morphisms of pp-divisors induce equivariant morphisms of affine normal varieties endowed with effective torus actions, semilinear morphisms of pp-divisors similarly induce semilinear equivariant morphisms of affine normal varieties endowed with effective torus actions. However, not every dominant semilinear equivariant morphism of affine normal varieties arises from a semilinear morphism of pp-divisors. 

In the following, we will prove that dominant semilinear equivariant morphisms between affine normal varieties endowed with an effective torus action arise from a pair of dominating semilinear morphisms of pp-divisors. The next results are intermediary steps that will help us to achieve our goal.

\begin{proposition}\label{lemma semilinear resolution}
Let $Y$ and $Y'$ be normal semi-projective varieties over $L$ and $\gamma\in\Gamma_{L}$.  Let $\xymatrix{h_{\gamma}:Y \ar@{-->}[r] & Y'}$ be a rational semilinear morphism with respect to $\gamma$, then there exists a normal semi-projective variety $\tilde{Y}$ over $L$ such that the following triangle is commutative
\[\xymatrix{ & \tilde{Y} \ar[dl]_{\kappa} \ar[dr]^{\psi_{\gamma}}& \\ Y \ar@{-->}[rr] & & Y' ,}\]
where $\kappa$ is a birational morphism of varieties over $L$ and $\psi_{\gamma}$ is a birational semilinear morphism with respect to $\gamma$. Moreover, if there exists a regular semilinear map $h'_{\gamma}:\aff{Y}\to\aff{Y}'$ such that $h_{\gamma}\circ r_{Y}=r_{Y'}\circ h'_{\gamma}$, then $\kappa$ is projective.
\end{proposition}

\begin{proof}
Consider the diagram corresponding to the semilinear rational map 
\[\xymatrix{ Y \ar[d] \ar@{-->}[r]^{h_{\gamma}} & Y' \ar[d] \\ L \ar[r]_{\gamma^{\natural}} & L .}\]
Denote by $Y'':=\gamma^{-1*}Y'$ the variety over $L$ corresponding to the composition 
\[\xymatrix{Y' \ar[r] & L \ar[r]^{(\gamma^{-1})^{\natural}} & L.}\] 
Then, $h_{\gamma}$ is a rational morphism of varieties over $L$ between $Y$ and $Y''$. By \cref{propositiondominantresolution0}, there exists a normal semi-projective variety $\tilde{Y}$ over $L$ with birational morphisms $\kappa$ and $\psi_{\gamma}$ such that the following diagram commutes
\[\xymatrix{ & \tilde{Y} \ar[dl]_{\kappa} \ar[dr]^{\psi_{\gamma}}& \\ Y \ar@{-->}[rr] & & Y'' .}\] 
Then, we have that the following diagram commutes
\[\xymatrix{ & \tilde{Y} \ar[dl]_{\kappa} \ar[dr]^{\psi_{\gamma}}& \\ Y \ar@{-->}[rr]_{h_{\gamma}} \ar[d] & & Y' \ar[d] \\ L \ar[rr]_{\gamma^{\natural}} & & L .}\]
Given that $\kappa$ is a morphism of varieties over $L$, we have that $\psi_{\gamma}$ is semilinear with respect to $\gamma$. Under the additional assumption of the existence of the semilinear morphism $h_{\gamma}'$, it follows that $\kappa$ is projective by \cref{propositiondominantresolution}. Then, the assertion holds.
\end{proof}

The following two lemmas were proved over an algebraically closed field of characteristic zero \cite[Lemmas 9.1 and 9.2]{AH06}. Nevertheless, both hold over any field.

\begin{lemma}\label{lemma 9.1 AH06}
Let $Y$ be a normal variety over $k$. If $D$ and $D'$ in $\cadiv_{\Q}(Y)$ are semiample and $\H^{0}(Y,\mathscr{O}(nD))\subset\H^{0}(Y,\mathscr{O}(nD'))$, as subspaces of $k(Y)$, holds for infinitely many $n>0$, then $D\leq D'$.
\end{lemma}

\begin{proof}
The divisors $D$ and $D'$ can be written as 
\[D=\sum \alpha_{i}D_{i}\quad\mathrm{and}\quad D'=\sum\alpha_{i}'D',\]
respectively, where the $D_{i}$'s and the $D_{i}'$'s are prime divisors on $Y$. For any $y\in Y$, we define 
\[D_{y}=\sum_{y\in D_{i}} \alpha_{i}D_{i}\quad\mathrm{and}\quad D'_{y}=\sum_{y\in D_{i}'}\alpha_{i}'D'.\]

Given that $D$ is semiample, there exists a section $f\in\H^{0}(Y,\mathscr{O}_{Y}(nD))$, for some $n\in \N$, such that $y\in Y_{f}$. This implies that $\mathrm{div}(f)_{y}+nD_{y}=0$. Since $\H^{0}(Y,\mathscr{O}(\tilde{n}D))\subset\H^{0}(Y,\mathscr{O}(\tilde{n}D'))$ holds for infinitely many $\tilde{n}>0$ and $n$ can be chosen satisfying such a condition, then we have that $f\in\H^{0}(Y,\mathscr{O}(nD'))$. Hence, $0\leq \mathrm{div}(f)_{y}+nD'_{y}$ and, therefore, $D_{y}\leq D'_{y}$ for every $y\in Y$. This implies that $D\leq D'$.
\end{proof}

\begin{lemma}\label{lemma 9.2 AH06}
Let $T$ be a split torus over $k$. Let $\D'$ be an objects in $\mathfrak{PPDiv}(k)$, the category of pp-divisors. Let $s:M\to k(X)^{*}$ such that $s(m)$ is homogeneous of degree $m$ for any $m\in M$ and $\D$ be its respective $\emph{GIT}$-constructed pp-divisor (see \cref{remark choice of section}) such that $\D$ and $\D'$ define the same affine normal $T$-variety, then there exists a plurifunction $\mathfrak{f}\in k(N,Y')^{*}$ such that $\D'=\vartheta^{*}\D+\mathrm{div}(\mathfrak{f})$, where $\vartheta:Y'\to Y$ is the canonical morphism.
\end{lemma}

\begin{proof}
Denote \[\mathscr{A}':=\bigoplus_{m\in\omega_{\D'}^{\vee}\cap M}\mathscr{O}_{Y'}(\D'(m))\] the $\mathscr{O}_{Y'}$-algebra associated to $\D'$, $\tilde{X}':=\Spec_{Y'}(\mathscr{A}')$, $A':=\H^{0}(Y',\mathscr{A}')$ and $X':=\Spec(A')$. 

On the one hand, there is a natural map $r':\tilde{X}'\to X'$, which fits into the following commutative diagram
\[\xymatrix{ 
r'^{-1}(X'^{\mathrm{ss}}(m)) \ar[r] \ar[d]_{} & X'^{\mathrm{ss}}(m) \ar[d]^{} \\ 
Y' \ar[rd]_{\vartheta} \ar[r]^{\vartheta_{m}} & Y'_{m} \\
& Y \ar[u]_{p_{m}} .} \]

On the other hand, by \cref{proposition existence ah-datum split}, the divisor $\D(m)$ is constructed as 
\[\mathscr{O}_{Y}(\D(m)):=\dfrac{1}{s(m)}\mathscr{A}_{m}\subset k(Y)^{*},\]
where $s:M\to k(X)^{*}$ is a section of the degree map and $\mathscr{A}_{m}$ is a sheaf such that $\H^{0}(Y,\mathscr{A}_{m})=A'_{m}$, the elements of degree $m$ of $A'$.

After pulling back $\D(m)$ by $\vartheta:Y'\to Y$, we have that 
\[\H^{0}(Y',\mathcal{O}_{Y'}(\vartheta^{*}\D(m)))=\dfrac{1}{s(m)}A'_{m}\subset k(Y').\]

Given that $X'=X(\D')$, we have that $\H^{0}(Y',\mathscr{O}_{Y'}(\D'(m)))\subset k(Y')$. Hence, by forgetting the grading, we have a multiplicative map
\begin{align*}
\bigcup_{m\in\omega_{\D'}^{\vee}\cap M}\H^{0}(Y',\mathscr{O}_{Y'}(\D'(m))) &\to k(Y') \\
f_{m} &\mapsto f_{m}.
\end{align*}
This map extends to the multiplicative system of homogeneous rational functions on $X'$. This allows us to see the rational maps $s(m)\in k(X)^{*}$ as elements in $k(Y')$ and therefore we can consider $\mathrm{div}(s(m))\in\cadiv(Y')$. Thus, 
\begin{align*}
\H^{0}(Y',\mathscr{O}_{Y'}(\vartheta^{*}\D(m))) &=\dfrac{1}{s(m)}A_{m'} \\
&=\dfrac{1}{s(m)}\H^{0}(Y',\mathscr{O}_{Y'}(\D'(m))) \\
&=\H^{0}(Y',\mathscr{O}_{Y'}(\D'(m)-\mathrm{div}(s(m)))).
\end{align*}
Besides, since $\vartheta^{*}\D(nm)=n\vartheta^{*}\D(m)$, $\D(nm)=n\D(m)$ and $\mathrm{div}(s(nm))=n\,\mathrm{div}(s(m))$ for every $n\in\N$, it follows that 
\[
\H^{0}(Y',\mathscr{O}_{Y'}(n\,\vartheta^{*}\D(m))) =\H^{0}(Y',\mathscr{O}_{Y'}(n(\D'(m)-\mathrm{div}(s(m)))))
\]
 also holds for every $n\in\N$. Then, by \cref{lemma 9.1 AH06}, we have that $\vartheta^{*}\D(m)=\D'-\mathrm{div}(s(m))$. Hence, defining $\mathfrak{f}\in k(N,Y')^{*}$ as the plurifunction such that $\mathrm{div}(\mathfrak{f})(m)=s(m)$, we have that $\vartheta^{*}\D+\mathrm{div}(\mathfrak{f})=\D'$. Then, the assertion holds.
\end{proof}

Now, we present one of the main results of this section.

\begin{theorem}\label{theorem semilinear automorphisms}
Let $\D$ and $\D'$ be two objects in $\mathfrak{PPDiv}(L/k)$ and $\gamma\in\Gamma_{L}$. Let $(\varphi_{\gamma},f_{\gamma}):X(\D)\to X(\D')$ be a dominant semilinear equivariant morphism. Then, there exists a normal semi-projective variety $\tilde{Y}$ over $L$, a projective birational morphism $\kappa:\tilde{Y}\to Y$ of varieties over $L$ and a semilinear morphism of pp-divisors $(\psi_{\gamma},F,\mathfrak{f}):\kappa^{*}\D\to\D'$ such that the following diagram commutes
\[\xymatrix{ & X(\kappa^{*}\D) \ar[dl]_{X(\kappa,\id_{N},1)}^{\cong} \ar[rd]^{X(\psi_{\gamma},F,\mathfrak{f})} & \\ X(\D) \ar[rr]_{(\varphi_{\gamma},f_{\gamma})} & & X(\D') .}\]
Besides, if $(\varphi_{\gamma},f_{\gamma})$ is a semilinear equivariant isomorphism and $\D'$ is minimal, then $\kappa$ can be taken as the identity and $F:N\to N'$ is an isomorphism such that $F(\omega_{\D})= \omega_{\D'}$. Moreover, if $\D$ is also minimal, then $\psi_{\gamma}:Y\to Y'$ is a semilinear equivariant isomorphism.
\end{theorem}

\begin{proof}
Denote $X:=X(\D)=\Spec(A)$ and $X':=X(\D')=\Spec(A')$. Since $\aff{Y}=\Spec(A_{0})$ and $\aff{Y}'=\Spec(A'_{0})$, the semilinear equivariant morphism $(\varphi_{\gamma},f_{\gamma}):X\to X'$ induces a semilinear morphism $h_{\gamma}':\aff{Y}\to \aff{Y}'$.

 Let $F:N\to N'$ be the lattice morphism corresponding to $\varphi_{\gamma}:T\to T'$ and $F^{*}:M'\to M$ its dual morphism. Let us consider the case where the pp-divisors $\D$ and $\D'$ are $\mathrm{GIT}$-constructed and, therefore, they are minimal. Otherwise stated, they arise from homomorphisms $s:M\to L(X)$ and $s':M'\to L(X')$ such that $\mathscr{O}_{Y}(\D(m))=\frac{1}{s(m)}\mathscr{A}_{m}\subset L(Y)$ and $\mathscr{O}_{Y'}(\D'(m))=\frac{1}{s'(m)}\mathscr{A}'_{m}\subset L(Y')$, respectively, as in \cref{proposition existence ah-datum split}. 

Since $(\varphi_{\gamma},f_{\gamma}):X\to X'$ is dominant, we have that $f_{\gamma}^{-1}(X'^{ss}(m))\subset \Xss(F^{*}(m))$ is not empty for every $m\in\omega_{\D'}^{\vee}\cap M'$. Therefore, we have the following data 
\[\xymatrixcolsep{6pc}\xymatrix{ \Xss(F^{*}(m)) & f_{\gamma}^{-1}(X'^{ss}(m)) \ar[r]^{(\varphi_{\gamma},f_{\gamma})} \ar[l]_{\iota} & X'^{ss}(m), }\] where $\iota$ is the natural embedding. Now, we can take the respective quotients and we get the following commutative diagram
\[\xymatrixcolsep{5pc}\xymatrix{ \Xss(F^{*}(m)) \ar[d] & f_{\gamma}^{-1}(X'^{ss}(m)) \ar[d] \ar[r]^{(\varphi_{\gamma},f_{\gamma})} \ar[l]_{\iota} & X'^{ss}(m) \ar[d] \\ Y_{F^{*}(m)} & f_{\gamma}^{-1}(X'^{ss}(m))\sslash T \ar[l]^{} \ar[r]_{(h_{\gamma})_{m}} & Y'_{m},}\]
where $(h_{\gamma})_{m}$ is a $\gamma$-semilinear morphism. For every $m\in \omega_{\D'}^{\vee}\cap M'$, the latter $\gamma$-semilinear morphism defines a rational $\gamma$-semilinear morphism $(h_{\gamma})_{m}:Y_{F^{*}(m)}\dasharrow Y'_{m}$. Notice that $(h_{\gamma})_{m_{1}}=(h_{\gamma})_{m_{2}}$, whenever $\lambda(m_{1})^{\vee}=\lambda(m_{2})^{\vee}$ by \cref{theorem GIT-equivalence}. Thus, there are just finitely many different rational $\gamma$-semilinear morphisms $(h_{\gamma})_{m}$, one for each $\lambda'\in\Lambda'$, and we thus set them by
\[\xymatrix{(h_{\gamma})_{\lambda'}:Y_{F^{*}(\lambda')} \ar@{-->}[r] & Y'_{\lambda'},}\]
where $F^{*}(\lambda')\in\Lambda$. Moreover, these rational $\gamma$-semilinear morphisms fit into the following commutative diagram
\[\xymatrix{Y_{F^{*}(\lambda')} \ar@{-->}[r]^{(h_{\gamma})_{\lambda'}} \ar[d]_{p_{F^{*}(\lambda')0}}& Y'_{\lambda'} \ar[d]^{p_{\lambda'0}} \\  \aff{Y} \ar[r]_{h_{\gamma}'} & \aff{Y}',}\]
where the morphisms $p_{F^{*}(\lambda')0}$ and $p_{\lambda'0}$ are those in \cref{Proposition lemma 6.1 AH06}. Hence, we have a rational $\gamma$-semilinear morphism between the limits $h_{\gamma}:Y \dasharrow Y'$ that fits into the following commutative diagram
\[\xymatrix{Y \ar@{-->}[r]^{h_{\gamma}} \ar[d]_{r_{Y}}& Y' \ar[d]^{r_{Y'}} \\  \aff{Y} \ar[r]_{h_{\gamma}'} & \aff{Y}',}\]
Then, by \cref{lemma semilinear resolution}, there exists a semilinear resolution of indeterminacies 
\[\xymatrix{ & \tilde{Y} \ar[dl]_{\kappa} \ar[dr]^{\psi_{\gamma}}& \\ Y \ar@{-->}[rr] & & Y' .}\]
such that $\tilde{Y}$ is normal and semi-projective and $\psi_{\gamma}$ and $\kappa$ are birational and $\kappa$ is also projective. Then, by using the homomorphisms $s:M\to L(X)$ and $s':M'\to L(X')$, we have the following commutative diagram 
\[\xymatrix{ A_{F^{*}(m)} & & A'_{m} \ar[ll]_{f_{\gamma}^{*}} \\ \H^{0}(\tilde{Y},\mathscr{O}_{\tilde{Y}}(\kappa^{*}\D(F^{*}(m)))) \ar[u]^{\cdot s(F^{*}(m))} & &\H^{0}(Y',\mathscr{O}_{Y'}(\D'(m))) \ar[u]_{\cdot s(m)} \ar[dl]^{\psi_{\gamma}^{*}} \\ &\H^{0}(\tilde{Y},\mathscr{O}_{\tilde{Y}}(\psi_{\gamma}^{*}\D'(m))) \ar[ul]^{\cdot \frac{f_{\gamma}^{*}(s'(m))}{s(F^{*}(m))}} .& }\]
From this commutative diagram we have a group homomorphism  
\begin{align*}
M' &\to L(\tilde{Y})^{*} \\
m &\mapsto f_{\gamma}^{*}(s'(m))/s(F^{*}(m)),
\end{align*} 
which, by \eqref{defpluri part b} of \cref{defpluri}, defines a plurifunction $\mathfrak{f}\in L(N',\tilde{Y})^{*}$ such that
\[\mathfrak{f}(m)= f_{\gamma}^{*}(s'(m))/s(F^{*}(m)),\] 
for every $m\in M'$ (consider a $\Z$-basis of $M$ and take the $f^{i}$ as the image of the elements of such a base, for instance). Notice that if $(\varphi_{\gamma},f_{\gamma})$ is an isomorphism, then no resolution of indeterminacies is needed and, therefore, $\psi_{\gamma}:Y\to Y'$ is a semilinear isomorphism. We claim that the triple $(\psi_{\gamma},F,\mathfrak{f}):\kappa^{*}\D\to\D'$ is a semilinear morphism of pp-divisors with respect to $\gamma$ that fits into the commutative triangle of the statement. In order to do this, it suffices to prove that 
\[\psi_{\gamma}^{*}\D'(m)\leq \kappa^{*}\D(F^{*}(m))+\mathrm{div}(\mathfrak{f})(m),\]
for every $m\in \omega_{\D'}^{\vee}\cap M'$. Since the morphism 
\[\xymatrixcolsep{3pc}\xymatrix{\H^{0}(\tilde{Y},\mathscr{O}_{\tilde{Y}}(\psi_{\gamma}^{*}\D'(m))) \ar[r]^{\frac{f_{\gamma}^{*}(s'(m))}{s(F^{*}(m))}} &\H^{0}(\tilde{Y},\mathscr{O}_{\tilde{Y}}(\kappa^{*}\D(F^{*}(m))))}\]
defines an inclusion
\[\H^{0}(\tilde{Y},\mathscr{O}_{\tilde{Y}}(\psi_{\gamma}^{*}\D'(m)-\mathrm{div}(\mathfrak{f})(m))) \subset\H^{0}(\tilde{Y},\mathscr{O}_{\tilde{Y}}(\kappa^{*}\D(F^{*}(m)))),\]
the claim holds by \cref{lemma 9.1 AH06}. Therefore, the assertion holds for $\D$ and $\D'$ minimal pp-divisors.

Suppose now that only $\D'$ is minimal and $\D$ is not. Let $\D_{1}$ be a minimal pp-divisor such that $X(\D)\cong X(\D_{1})$, which exists by the construction made in the proof of \cref{proposition existence ah-datum split} (see \cref{remark choice of section}). On the one hand, by \cref{lemma 9.2 AH06}, there exists a plurifunction $\mathfrak{f}_{1}\in L(N,Y)$ such that $\D=\vartheta^{*}\D_{1}+\mathrm{div}(\mathfrak{f}_{1})$, where $\vartheta:Y\to Y_{1}$ is the canonical morphism. On the other hand, given that $\D_{1}$ and $\D'$ are minimal pp-divisors, the theorem holds for them. Hence, there exists a normal semi-projective variety $\tilde{Y}_{1}$ over $L$, a projective birational morphism $\kappa_{1}:\tilde{Y}_{1}\to Y_{1}$ and a semilinear morphism of pp-divisors $(\psi_{\gamma},F,\mathfrak{f}):\kappa_{1}^{*}\D_{1}\to\D'$ such that the following diagram commutes
\[\xymatrix{ 
& X(\kappa_{1}^{*}\D_{1}) \ar[dl]_{X(\kappa_{1},\id_{N},1)}^{\cong} \ar[rd]^{(\psi_{\gamma},F,\mathfrak{f})} & \\
 X(\D_{1}) \ar[rr]_{(\varphi_{\gamma},f_{\gamma})} & & X(\D') 
 .}\]
Now, consider the fiber product 
\[\xymatrix{ Y\times_{Y_{1}} \tilde{Y}_{1} \ar[r]^{\pi_{1}} \ar[d]_{\pi_{2}} & \tilde{Y}_{1} \ar[d]^{\kappa_{1}} \\
Y \ar[r]_{\vartheta} & Y_{1} .}\]
The morphisms $\vartheta$ et $\kappa_{1}$ are birational, then there exist open subvarieties of $Y$ and $\tilde{Y}_{1}$, respectively, isomorphic to open subvarieties of $Y_{1}$. Hence, there exists an open subvariety $U\subset \tilde{Y}_{1}\times_{Y_{1}}Y$ isomorphic to open subvarieties of $\tilde{Y}_{1}$ and $Y_{1}$ under the canonical projections $\pi_{1}$ and $\pi_{2}$. Let $\tilde{Y}:=\overline{U}^{\nu}$ be the normalization of the closure of $U$, $p_{1}:\tilde{Y}\to \tilde{Y}_{1}$ the restriction of $\pi_{1}$ and $\kappa_{2}:\tilde{Y}\to Y$ the restriction of $\pi_{2}$. Then, we have the following commutative diagram
\[\xymatrix{ \tilde{Y} \ar[r]^{p_{1}} \ar[d]_{\kappa_{2}} & \tilde{Y}_{1} \ar[d]^{\kappa_{1}} \ar[rd]^{\psi_{\gamma}} \\
Y \ar[r]_{\vartheta} & Y_{1} & Y'.}\]
Notice that the morphisms of the square are morphisms of varieties over $L$. Then $\psi_{\gamma}\circ p_{1}$ is $\gamma$-semilinear. 

We need to construct a morphism of pp-divisors $\kappa_{2}^{*}\D\to\D$ from the data above. From the fact that $(\psi_{\gamma},F,\mathfrak{f}):\kappa_{1}^{*}\D_{1}\to\D$ is a semilinear morphism of pp-divisors and applying $p_{1}^{*}$ we have
\begin{align*}
(\psi_{\gamma}\circ p_{1})^{*}\D' &=p_{1}^{*}\psi_{\gamma}^{*}\D' \\
&\leq p_{1}^{*}F_{*}\kappa_{1}^{*}\D_{1}+\mathrm{div}(p_{1}^{*}\mathfrak{f}) \\
&=F_{*}p_{1}^{*}\kappa_{1}^{*}\D_{1}+\mathrm{div}(p_{1}^{*}\mathfrak{f}),
\end{align*}
and by the commutativity of the diagram above and the identity $\D=\vartheta^{*}\D_{1}+\mathrm{div}(\mathfrak{f}_{1})$, 
\begin{align*}
(\psi_{\gamma}\circ p_{1})^{*}\D' &\leq F_{*}p_{1}^{*}\kappa_{1}^{*}\D_{1}+\mathrm{div}(p_{1}^{*}\mathfrak{f}) \\
&=F_{*}(\kappa_{1}p_{1})^{*}\D_{1}+\mathrm{div}(p_{1}^{*}\mathfrak{f}) \\
&=F_{*}(\vartheta\kappa_{2})^{*}\D_{1}+\mathrm{div}(p_{1}^{*}\mathfrak{f})\\
&=F_{*}\kappa_{2}^{*}\vartheta^{*}\D_{1}+\mathrm{div}(p_{1}^{*}\mathfrak{f}) \\
&=F_{*}\kappa_{2}^{*}\D-\mathrm{div}(F_{*}\kappa_{2}^{*}\mathfrak{f}_{1})+\mathrm{div}(p_{1}^{*}\mathfrak{f}).
\end{align*}
By \cref{remark inverse of polyhedral principal divisor}, there exists a plurifunction $\mathfrak{f}_{2}$ such that $\mathrm{div}(\mathfrak{f}_{2})=-\mathrm{div}(F_{*}\kappa_{2}^{*}\mathfrak{f}_{1})$. Then, if we denote $\tilde{\mathfrak{f}}=\mathfrak{f}_{2}\cdot p_{1}^{*}\mathfrak{f}$, we have 
\[(\psi_{\gamma}\circ p_{1})^{*}\D'\leq F_{*}\kappa_{2}^{*}\D+\mathrm{div}(\tilde{\mathfrak{f}}).\]
This implies that $(\psi_{\gamma}\circ p_{1},F,\tilde{\mathfrak{f}}):\kappa_{2}^{*}\D\to\D'$ is a semilinear morphism of pp-divisors that fits by construction into the commutative triangle of the statement
\[\xymatrix{
& X(\kappa_{2}^{*}\D) \ar[dl]_{X(\kappa_{2},\id_{N},\mathfrak{1})} \ar[d]^{X(p_{1},\id_{N},\mathfrak{f}_{2})} \ar@/^3pc/[rdd]^{X(\psi_{\gamma}\circ p_{1},F,\tilde{\mathfrak{f}})} & \\ 
X(\D) \ar[d]_{X(\vartheta,\id_{N},\mathfrak{f}_{1})} & X(\kappa_{1}^{*}\D_{1}) \ar[dl]^{X(\kappa_{1},\id_{N},\mathfrak{1})}_{\cong} \ar[rd]^{X(\psi_{\gamma},F,\mathfrak{f})} & \\
 X(\D_{1}) \ar[rr]_{(\varphi_{\gamma},f_{\gamma})} & & X(\D') 
 ,}\]
where $X(\vartheta,\id_{N},\mathfrak{f}_{1})$ is the identity map. Now, If $(\psi_{\gamma},f_{\gamma})$ is a semilinear isomorphism with respect to $\gamma$, then $\kappa_{1}$ can be considered as the identity map and, therefore, $\tilde{Y}_{1}\times_{Y_{1}} Y=Y$. Then, in this case $\tilde{Y}=\overline{U}=Y$, which implies that $\kappa_{2}$ is the identity. This proves the theorem in the case where $\D$ is not minimal and $\D'$ is minimal.

Suppose now that we are in the most general case. The strategy is the same as the previous case, but we have to be careful with the fiber product part. Let $\D'_{2}$ be a minimal pp-divisor such that $X(\D')=X(\D_{2}')$. On the one hand, by \cref{lemma 9.2 AH06}, there exists a plurifunction $\mathfrak{f}_{2}\in L(N',Y')$ such that $\D'=\vartheta^{*}\D'_{2}+\mathrm{div}(\mathfrak{f}_{2})$, where $\vartheta:Y'\to Y_{2}'$ is the canonical morphism. On the other hand, by what we have done so far, we know that the theorem holds for $\D$ and $\D'_{2}$. Then, there exists a normal semi-projective variety $\tilde{Y}_{2}$ over $L$, a projective birational morphism $\kappa_{2}:\tilde{Y}_{2}\to Y$ and a semilinear morphism of pp-divisors $(\psi_{\gamma},F,\mathfrak{f}):\kappa_{2}^{*}\D\to\D'_{2}$ such that 
\[\xymatrix{ & X(\kappa_{2}^{*}\D) \ar[dl]_{X(\kappa_{2},\id_{N},1)}^{\cong} \ar[rd]^{(\psi_{\gamma},F,\mathfrak{f})} & \\ X(\D) \ar[rr]_{(\varphi_{\gamma},f_{\gamma})} & & X(\D'_{2}) .}\]
In this case we have the following commutative diagram 
\[\xymatrix{ \tilde{Y}_{2} \ar[r]^{\psi_{\gamma}} \ar[d] & Y_{2}' \ar[d] & Y' \ar[l]_{\vartheta} \ar[d] \\ 
L \ar[r]_{\gamma^{\natural}} & L & L \ar[l]^{\id},}\]
then we cannot just take the fiber product because $\psi_{\gamma}$ is not a morphism of $L$-varieties. Denote by $\tilde{Y}_{2}''$ the $L$-variety given by the composition 
\[\xymatrix{\tilde{Y}_{2} \ar[r]^{\psi_{\gamma}} & L \ar[r]^{\gamma^{\natural}} & L}\] 
and by $h:\tilde{Y}_{2}''\to Y$ the corresponding morphism of varieties over $L$. Note that $\tilde{Y}_{2}=\tilde{Y}_{2}''$ as schemes. Consider the fiber product $\tilde{Y}''_{2}\times_{Y_{2}'}Y'$. By following the arguments above, let $\tilde{Y}$ be the normalization of the closure of an open subvariety of $\tilde{Y}''_{2}\times_{Y_{2}'}Y'$  isomorphic to some open subvarieties of each of the factors. Then, we have the following commutative diagram of varieties over $L$.
\[\xymatrix{ \tilde{Y} \ar[r]^{p_{1}} \ar[d]_{p_{2}} & Y'  \ar[d]^{\vartheta} \\ 
\tilde{Y}_{2}'' \ar[r]_{h} & Y_{2}' ,}\]
where the morphisms $p_{1}$ and $p_{2}$ are induced by the canonical projections of the fiber product. Then, we have the following diagram
\[\xymatrix{ & \tilde{Y} \ar[r]^{p_{1}} \ar[d]_{p_{2}} & Y'  \ar[d]^{\vartheta} \\ 
Y & \tilde{Y}_{2} \ar[l]^{\kappa_{2}} \ar[r]_{\psi_{\gamma}} & Y_{2}' ,}\]
where $p_{1}$ is a projective dominant semilinear morphism with respect to $\gamma$ and $p_{2}$ is a morphism of varieties over $L$. We denote $\kappa:=\kappa_{2}\circ p_{2}$. We claim that the triple $(p_{1},F,p_{2}^{*}\mathfrak{f}\cdot p_{1}^{*}\mathfrak{f}_{2})$ is a morphism of pp-divisors $\kappa^{*}\D\to\D'$. Indeed, since $(\psi_{\gamma},F,\mathfrak{f})$ is a semilinear morphism of pp-divisors, we have 
\begin{align*}
p_{1}^{*}\D' &= p_{1}^{*}(\vartheta^{*}\D'_{2}+\mathrm{div}(\mathfrak{f}_{2})) \\
 &=p_{1}^{*}\vartheta^{*}\D'_{2}+\mathrm{div}(p_{1}^{*}\mathfrak{f}_{2}) \\
 &=p_{2}^{*}\psi_{\gamma}^{*}\D'_{2}+\mathrm{div}(p_{1}^{*}\mathfrak{f}_{2}) \\
 &\leq p_{2}^{*} F_{*}\kappa_{2}^{*}\D+p_{2}^{*}\mathrm{div}(\mathfrak{f})+\mathrm{div}(p_{1}^{*}\mathfrak{f}_{2}) \\
 &= F_{*}p_{2}^{*}\kappa_{2}^{*}\D+\mathrm{div}(p_{2}^{*}\mathfrak{f})+\mathrm{div}(p_{1}^{*}\mathfrak{f}_{2}) \\
 &=F_{*}\kappa^{*}\D+\mathrm{div}(p_{2}^{*}\mathfrak{f}\cdot p_{1}^{*}\mathfrak{f}_{2}).
\end{align*}

The triples $(\kappa,\id,\mathfrak{1}):\kappa^{*}\D\to\D$ and $(p_{1},F,p_{2}^{*}\mathfrak{f}\cdot p_{1}^{*}\mathfrak{f}_{2}):\kappa^{*}\D\to\D'$ are the semilinear morphisms of pp-divisors that satisfy the commutative triangle of the assertion
\[\xymatrix{
& X(\kappa_{}^{*}\D) \ar[ddl]_{X(\kappa,\id_{N},\mathfrak{1})} \ar[d]^{X(p_{2},\id_{N},\mathfrak{f}_{2})} \ar@/^2pc/[rd]^{X(p_{1},F,p_{2}^{*}\mathfrak{f}\cdot p_{1}^{*}\mathfrak{f}_{2})} & \\ 
 & X(\kappa_{2}^{*}\D) \ar[dl]^{X(\kappa_{2},\id_{N},\mathfrak{1})}_{\cong} \ar[rd]^{X(\psi_{\gamma},F,\mathfrak{f})} & X(\D') \ar[d]^{X(\vartheta,\id_{N},\mathfrak{f}_{2})} \\
 X(\D) \ar[rr]_{(\varphi_{\gamma},f_{\gamma})} & & X(\D'_{2}) 
 ,}\]
where $X(\vartheta,\id_{N},\mathfrak{f}_{2})$ is the identity map. This proves the assertion.
\end{proof}

\begin{remark}
Notice that \cref{theorem semilinear automorphisms} generalizes \cite[Theorem 8.8]{AH06}. It suffices to consider the semilinear morphisms with $\gamma$ the neutral element of the Galois group.
\end{remark}

 Let $T$ be a split algebraic torus over $L$ and $N$ be its cocharacter lattice. Let $\D$ be an object in $\mathfrak{PPDiv}_{N}(L/k)$. Consider the set
\[S(\D):=\{ (\psi_{\gamma},F,\mathfrak{f}):\D\to\D \mid X(\psi_{\gamma},F,\mathfrak{f}) \textrm{ in } \SAut(T;X(\D))  \}.\]
For a general $\D$, the set $S(\D)$ has a structure of monoid, having $(\id,\id,1)$ as the neutral element, but not necessarily a group structure because of the discussion given in \cref{section: functoriality} after \cref{corollary X is faithful}. However, for a minimal pp-divisor, $S(\D)$ has a group structure by \cref{theorem semilinear automorphisms}. In such a case, we denote by $\SAut(\D):=S(\D)$ the group of semilinear automorphisms of pp-divisors of $\D$. Thus, a direct consequence of \cref{theorem semilinear automorphisms} is the following.

\begin{corollary}\label{corollary saut D saut XD}
Let $\D$ be an object in $\mathfrak{PPDiv}(L/k)$ that is minimal. Then, 
\[\SAut(\D)\cong \SAut(T;X(\D))\] as groups, where $T:=T(\D)$ is the corresponding split $L$-torus acting on $X(\D)$ and $\SAut(T;X(\D))$ stands for the semilinear equivariant automorphisms of $X(\D)$. 
\end{corollary}

A more precise statement over the semilinear equivariant automorphisms of a minimal pp-divisor is the following.

\begin{corollary}\label{corollaryequivariantautaffine}
   Let $\D$ be a minimal pp-divisor in $\mathfrak{PPDiv}(L/k)$ and $\gamma\in\Gamma_{L}$. Then the semilinear equivariant automorphisms $(\varphi_{\gamma},f_{\gamma}):X(\D)\to X(\D)$ correspond bijectively to the semilinear morphisms of pp-divisors $(\psi_{\gamma},F,\mathfrak{f})$ such that $\psi_{\gamma}^{*}(\D)=F_{*}(\D)+\divr(\mathfrak{f})$. In particular, if $\varphi_{\gamma}=\id_{T}$ we have $X(\psi_{\gamma},\id_{N},\mathfrak{f})=(\id_{T},f)$ and $\psi_{\gamma}^{*}(\D)=\D+\divr(\mathfrak{f})$.
\end{corollary}

And in the toric case, since the only basis for pp-divisors turns out to be $Y=\Spec(L)$, \cref{theorem semilinear automorphisms} yields the following.

\begin{corollary}\label{corollaryequivariantaffinetoric}
    Let $X_{\omega}$ and $X_{\omega'}$ be two affine normal toric varieties over $L$, with respective cones $\omega$ and $\omega'$, and $\gamma\in\Gamma_{L}$. Let $(\varphi_{\gamma},f_{\gamma}):X_{\omega}\to X_{\omega'}$ be a semilinear equivariant isomorphism. Then, there exists a triple $(\psi_{\gamma},F,\mathfrak{f})$, where $\psi_{\gamma}=\gamma^{\natural}:\Spec(L)\to \Spec(L)$, $F:N\to N'$ is an isomorphism of lattices such that $F(\omega)=\omega'$ and $\mathfrak{f}\in N\otimes L^{*}$, such that $(\varphi_{\gamma},f_{\gamma})=X(\psi_{\gamma},F,\mathfrak{f})$.
\end{corollary}

\begin{remark}
    Notice that in the toric case the plurifunction $\mathfrak{f}$ can be identified with an $L$-point of $T$, because there is an identification $T(L)\cong N\otimes_{\Z} \G_{\mathrm{m}}(L)$. 
\end{remark}

We can always assume that the pp-divisors are defined over complete varieties, by Nagata's compactification Theorem. If we restrict the functor $X(\bullet)$ to the full subcategory of $\mathfrak{PPDiv}_{N}(L/k)$ whose objects are pp-divisors over regular complete curves, denoted by $\mathfrak{PPDiv}_{N}^{\mathrm{reg}}(L/k)$, then we get an equivalence of categories with the category of complexity one affine normal $T$-varieties. 

\begin{corollary}
    The functor $X:\mathfrak{PPDiv}_{N}(L/k)\to \mathcal{E}_{T}(L/k)$ restricts to an equivalence from the full subcategory of $\mathfrak{PPDiv}_{N}(L/k)$ whose objects are pp-divisors over regular complete curves to the full subcategory of $\mathcal{E}_{T}$ whose objects are complexity one affine normal $T$-varieties.
\end{corollary}

\section{Nonsplit affine normal $T$-varieties}\label{chapter equivariant automorphisms affine}
\indent

This section is devoted to the proof of \cref{maintheoremofpaper}, which we recall below for the convenience of the reader.

We start with \cref{section: Galois descent tools}, where we establish a parallelism between Galois semilinear actions and Galois equivariant descent data. 

In \cref{section: affine minimal}, we introduce the category $\mathfrak{PPDiv}(\Gamma_{L})$ whose objects are pairs $(\D,\varrho)$, where $\D$ is a pp-divisor over $L$ and $\varrho$ is a Galois semilinear action on $\D$. In \cref{section: affine minimal}, we prove that a pair $(\D,\varrho)$ encodes the information of an affine normal $T$-variety $X(\D,\varrho)$ over $k$ such that $T_{L}$ is split. Moreover, we prove that any affine normal $T$-variety over $k$ such that $T_{L}$ is split arises this way.

\begin{theorem}
Let $k$ be a field, $L/k$ be a finite Galois extension with Galois group $\Gamma_{L}$.
	\begin{enumerate}[a)]
            \item\label{theorem main affine minimal part i} Let $(\D_{L},\varrho)$ be an object in $\mathfrak{PPDiv}(\Gamma_{L})$. Then, $X(\D_{L},\varrho)$ is an affine normal variety endowed with an effective action of an algebraic torus $T$ over $k$ such that $T$ splits over $L$ and $X(\D_{L},\varrho)_{L}\cong X(\D_{L})$ as $T_{\D_{L}}$-varieties over $L$.
            \item\label{theorem main affine minimal part ii} Let $X$ be an affine normal variety over $k$ endowed with an effective $T$-action such that $T_{L}$ is split. Then, there exists an object $(\D_{L},\varrho)$ in $\mathfrak{PPDiv}(\Gamma_{L})$ such that $X \cong X(\D_{L},\varrho)$ as $T$-varieties.
        \end{enumerate}
\end{theorem}
    
\subsection{ Galois descent via semilinear morphisms}\label{section: Galois descent tools}
\indent

In this section, we establish a correspondence between Galois descent data and Galois semilinear equivariant actions, a notion developed in \cref{section: semilinear varieties}. This allows us to translate the Galois descent data to the algebro-geometric combinatorial description (see \cref{section: semilinear pp-div}). This section is based on \cite[Section 14.20]{GW20} and \cite{Bor20}. A further reference on this topic is \cite[Section 6.2]{BLR90}.

\paragraph{Classical Galois descent.} Given a scheme $S$ over $k$, the scheme $S_{L}$ is a scheme over $L$, which arises from $S$ under a base change. However, not every scheme over $L$ arises from a scheme over $k$, and Galois descent allows us to describe when this phenomenon holds.

Let $S$ be a scheme over $k$. Let us denote $L':=L\otimes_{k} L$ and $L'':=L\otimes_{k} L\otimes_{k} L$. These algebras come with canonical projections $p_{i}:L'\to L$, for $i\in\{1,2\}$, and $p_{ij}:L''\to L$, for $i,j\in\{1,2,3\}$ and $i<j$. The scheme $S_{L}:=S\otimes_{k} L$ comes with a canonical map
\[h_{\mathrm{can}}:S_{L}\times_{L,p_{1}} \Spec(L')\to S_{L}\times_{L,p_{1}} \Spec(L'), \]
which satisfies the cocycle condition $p_{23}^{*} h_{\can}\circ p_{12}^{*}h_{\can}=p_{13}^{*}h_{\can}$. This piece of data is called a \emph{Galois descent datum}.

\begin{definition}
Let $S$ be a scheme over $L$. A \emph{Galois descent datum on $S$} is an isomorphism
\[h:S\times_{L,p_{1}} \Spec(L') \to S\times_{L,p_{2}} \Spec(L') \]
of schemes over $\Spec(L\otimes_{k} L)$ satisfying the cocycle condition $p_{23}^{*} h\circ p_{12}^{*}h=p_{13}^{*}h$.
\end{definition}

Let $Z_{1}$ and $Z_{2}$ be schemes over $L$. Let $h_{1}$ and $h_{2}$ be Galois descent data on $Z_{1}$ and $Z_{2}$, respectively. A morphism between the pairs $(Z_{1},h_{1})$ and $(Z_{2},h_{2})$ is a morphism $f:Z_{1}\to Z_{2}$ of schemes over $L$ that is compatible with the Galois descent data, i.e. $p_{2}^{*}f \circ h_{1}=h_{2}\circ p_{1}^{*}f$. The pairs $(Z,h)$ form a category that will be denoted by $\mathbf{Sch}_{L/k}$.

As said above, for any scheme $S$ over $k$, the scheme $S_{L}$ comes with a canonical Galois descent datum and, therefore, induces a pair $(S_{L},h_{\can})$. This construction induces a functor
\begin{align*}
	\Phi: \mathbf{Sch}_{k} &\to \mathbf{Sch}_{L/k} \\
	S &\mapsto (S_{L},h_{\can}).
\end{align*}

A descent datum $h$ on a scheme $Z$ over $L$ is said to be \emph{effective} if the pair $(Z,h)$ lies in the essential image of $\Phi$. We say that an open subscheme $U\subset Z$ is \emph{stable under $h$} if $h$ restricts to an isomorphism $U\times_{L,p_{1}} \Spec(L') \to U\times_{L,p_{2}} \Spec(L')$. Notice that, if $U$ is stable under $h$, the latter induces a Galois descent datum on $U$.  

The following results can be found in \cite[Theorem 14.84]{GW20}, for instance.

\begin{theorem}\label{Theorem14.70}
The functor
	\begin{align*}
		\Phi: \mathbf{Sch}_{k} &\to \mathbf{Sch}_{L/k} \\
		S &\mapsto (S_{L},h_{\can}).
	\end{align*}
is fully faithful. Moreover, if $(Z,h)$ is an object in $\mathbf{Sch}_{L/k}$ and $Z$ is covered by affine open subschemes that are stable under $h$, then $h$ is effective.
\end{theorem} 

Let us denote by $\mathbf{QPSch}_{k}$ the category of quasi-projective schemes over $k$ and by $\mathbf{QPSch}_{L/k}$ the category of quasi-projective schemes over $L$ endowed with a Galois descent datum.

\begin{corollary}\label{Proposition descent quasi-projective}
The functor
	\begin{align*}
		\Psi: \mathbf{QPSch}_{k} &\to \mathbf{QPSch}_{L/k} \\
		S &\mapsto (S_{L},h_{\can}).
	\end{align*}
is an equivalence of categories.
\end{corollary}

This point of view on Galois descent does not allow us to see a proper description of affine normal $T$-varieties for nonsplit torus actions. A better one turns out to be in terms of \emph{Galois semilinear actions}.

\paragraph{From the classical to semilinear actions.} Let $S$ be a scheme over $L$ and $\gamma\in\Gamma_{L}$. The automorphism $\gamma:L\to L$ induces a morphism of schemes $\gamma^{*}:\Spec\, L\to\Spec\, L$. We denote $\gamma S:=S^{\gamma^{-1}}$, the twisting of $S$ by $\gamma^{-1}$, which is the fiber product 
\[\xymatrix{ \gamma S \ar[r]^{\alpha_{\gamma}} \ar[d] & S \ar[d] \\ \Spec\, L \ar[r]_{\gamma^{*}} & \Spec\, L \,,}\]
where $S\to \Spec\, L$ is the structural morphism. Let $S'$ be another scheme over $L$ and $f:S' \to S$ be a morphism of schemes over $L$, we denote by $\gamma f:\gamma S'\to\gamma S$ the pullback of the morphism by $\gamma^{*}$, which satisfies 
\[\alpha_{\gamma}\circ\gamma f=f\circ\alpha_{\gamma}.\]
Notice that $\gamma^{*}=(\gamma^{-1})^{\natural}$, $\alpha_{\gamma}=\beta_{\gamma^{-1}}$ and $\gamma f=f^{\gamma^{-1}}$(see the beginning of \cref{section: semilinear varieties}).

The morphisms $\alpha_{\gamma}$ satisfy 
\[\alpha_{\tau\gamma}=\alpha_{\tau}\circ \tau\alpha_{\gamma},\]
for $\gamma$ and $\tau$ in $\Gamma_{L}$.

Let $(Z,h)$ be an object in $\mathbf{Sch}_{L/k}$. The $L$-algebra isomorphism 
	\begin{align*}
		\beta:L\otimes_{k} L &\to \prod_{\gamma\in\Gamma_{L}}L \\
		  x\otimes y &\mapsto (x\gamma^{-1}(y))_{\gamma\in\Gamma_{L}},
	\end{align*}
where the $L$-algebra structure of $L\otimes_{k}L$ is given by $l\cdot (x\otimes y):=(lx)\otimes y$ and the $L$-algebra structure of $\prod_{\gamma\in\Gamma_{L}} L$ is given by $l\cdot (z_{\gamma})_{\gamma\in \Gamma_{L}}:=(lz_{\gamma})_{\gamma\in\Gamma_{L}}$, induces a family of isomorphisms $h_{\gamma}:\gamma Z\to Z$, of varieties over $L$, satisfying the cocycle condition given by the following commutative diagram
\[ \xymatrix{ 
 \tau\gamma Z \ar[d]_{ \tau h_{\gamma}} \ar[rd]^{h_{ \tau\gamma}} & \\ 
  \tau Z \ar[r]_{h_{ \tau}} & Z\, ,} \]
for every $\gamma, \tau\in\Gamma_{L}$. Moreover, given a family of isomorphisms $h_{\gamma}:\gamma Z\to Z$ of varieties over $L$, satisfying the cocycle condition above, we get a Galois descent datum $h$ on $Z$.

Let $h=\{h_{\gamma}\}_{\gamma\in\Gamma_{L}}$ be a Galois descent datum over a scheme $S$ over $L$. For every $\gamma\in\Gamma_{L}$ we define the following semilinear morphism
\[\xymatrix{ S \ar[r]^{\alpha^{-1}_{\gamma}} \ar@/^2pc/[rr]^{\varrho_{\gamma}} \ar[d] & \gamma S \ar[d] \ar[r]^{h_{\gamma}} & S \ar[d] \\
\Spec\, L \ar[r]_{\gamma^{\natural}} & \Spec\, L \ar[r]_{\id} & \Spec\, L. }\]
This construction induces a map $\varrho:\Gamma_{L}\to \SAut(S)$.

\begin{lemma}\label{lemma group monomorphism Gamma to SAut}
The map defined above
\begin{align*}
\varrho:\Gamma_{L} &\to \SAut(S) \\ 
\gamma  &\mapsto \varrho(\gamma):=\varrho_{\gamma}
\end{align*} 
is a group homomorphism that defines a section of \eqref{Sequence saut aut}. Moreover, it is a monomorphism.
\end{lemma}

\begin{proof}
Notice first that $\varrho$ is a monomorphism, since $\varrho_{\gamma}=\id_{S}$ directly implies that $\gamma$ is the neutral element of $\Gamma_{L}$. 

Let $\gamma$ and $\tau$ be in $\Gamma_{L}$. By definition, we have
\[\varrho_{\tau\gamma}=h_{\tau\gamma}\circ \alpha_{\tau\gamma}^{-1}=h_{\tau}\circ \tau h_{\gamma}\circ (\tau\alpha_{\gamma})^{-1}\circ \alpha_{\tau}^{-1}.\]
Given that $(\tau\alpha_{\gamma})^{-1}=\tau\alpha_{\gamma}^{-1}$, we have that
\[\varrho_{\tau\gamma}=h_{\tau}\circ \tau h_{\gamma}\circ \tau\alpha_{\gamma}^{-1}\circ \alpha_{\tau}^{-1}= h_{\tau}\circ \tau (h_{\gamma}\circ \alpha_{\gamma}^{-1})\circ \alpha_{\tau}^{-1}.\]
Then, by the relation $\alpha_{\tau}\circ\tau f=f\circ\alpha_{\tau}$, it follows that 
\[\varrho_{\tau\gamma}= h_{\tau}\circ \tau (h_{\gamma}\circ \alpha_{\gamma}^{-1})\circ \alpha_{\tau}^{-1}=h_{\tau}\circ\alpha_{\tau}^{-1}\circ h_{\gamma}\circ\alpha_{\gamma}^{-1}=\varrho_{\tau}\circ \varrho_{\gamma}.\]
Finally, since $\varrho_{\gamma}$ is $\gamma$-semilinear, $\varrho$ defines a section. Thus, the assertion holds.
\end{proof}

As a consequence of \cref{lemma group monomorphism Gamma to SAut}, a Galois descent datum induces a Galois semilinear action (see \cref{definition semilinear equivariant action}). Moreover, every Galois semilinear action arises from a unique Galois descent datum.

\begin{lemma}\label{lemma galois sa to descent data}
 Let $S$ be a scheme over $L$ and $\varrho:\Gamma_{L}\to \SAut(S)$ be a $\Gamma_{L}$-semilinear action on $S$. Then, there exists a unique Galois descent datum $\{h_{\gamma}\}_{\gamma\in\Gamma_{L}}$ over $S$, such that $\varrho(\gamma)=\varrho_{\gamma}$ for every $\gamma\in\Gamma_{L}$.
\end{lemma}

\begin{proof}
For every $\gamma\in\Gamma_{L}$, define $h_{\gamma}:=\varrho(\gamma)\circ \alpha_{\gamma}$. Recall that, for $\gamma$ and $\tau$ in $\Gamma_{L}$, we have that
\[\alpha_{\tau\gamma}=\alpha_{\tau} \circ \tau\alpha_{\gamma}.\] 
Hence,
\[h_{\tau \gamma}=\varrho(\tau \gamma)\circ \alpha_{\tau \gamma}=\varrho(\tau)\circ \varrho(\gamma) \circ \alpha_{\tau}\circ \tau\alpha_{\gamma}.\]
The relation $\alpha_{\tau}\circ\tau \varrho(\gamma)=\varrho(\gamma)\circ\alpha_{\tau}$ implies
\[h_{\tau \gamma}=\varrho(\tau)\circ \varrho(\gamma) \circ \alpha_{\tau}\circ \tau\alpha_{\gamma}=\varrho(\tau)\circ \alpha_{\tau}\circ \tau \varrho(\gamma) \circ \tau\alpha_{\gamma}.\]
Then, given that $\tau(\varrho(\gamma)\circ \alpha_{\gamma})=\tau \varrho(\gamma)\circ \tau\alpha_{\gamma}$, we have
\[h_{\tau \gamma}=\varrho(\tau)\circ \alpha_{\tau}\circ \tau \varrho(\gamma) \circ \tau\alpha_{\gamma}=\varrho(\tau)\circ \alpha_{\tau}\circ \tau (\varrho(\gamma)\circ\alpha_{\gamma}) =h_{\tau}\circ \tau h_{\gamma}, \]
which is the cocycle condition. Thus, the set $\{h_{\gamma}\}_{\gamma\in\Gamma_{L}}$ forms a Galois descent system. Moreover, for every $\gamma\in\Gamma_{L}$, we have that $\varrho_{\gamma}=h_{\gamma}\circ \alpha_{\gamma}^{-1}=\varrho(\gamma)$. 

Let us assume now that $\{h_{\gamma}\}_{\gamma\in\Gamma_{L}}$ and $\{h_{\gamma}'\}_{\gamma\in\Gamma_{L}}$ are two Galois descent systems such that $\varrho(\gamma)=\varrho_{\gamma}=\varrho_{\gamma}'$. Hence, $h_{\gamma}\circ\alpha_{\gamma}^{-1}=h_{\gamma}'\circ\alpha_{\gamma}^{-1}$ and, therefore, $h_{\gamma}=h_{\gamma}'$. This proves the uniqueness part and, therefore, the assertion.
\end{proof}

Let $Z_{1}$ and $Z_{2}$ be schemes over $L$ and $\varrho_{1}$ and $\varrho_{2}$ be Galois semilinear actions over $Z_{1}$ and $Z_{2}$, respectively. A \emph{morphism} between $(Z_{1},\varrho_{1})$ and $(Z_{2},\varrho_{2})$ is a morphism of schemes $f:Z_{1}\to Z_{2}$ such that $\varrho_{1}\circ f=f\circ \varrho_{2}$. Let us denote by $\mathbf{Sch}(L/k)$ the category of schemes over $L$ endowed with a Galois semilinear action. There is a natural functor $\mathcal{F}:\mathbf{Sch}_{L/k}\to \mathbf{Sch}(L/k)$ that sends $(Z,h)$ to $(Z,\varrho)$, where $\varrho(\gamma)=h_{\gamma}\circ \alpha_{\gamma}^{-1}$. The following result is a consequence of \cref{lemma group monomorphism Gamma to SAut} and \cref{lemma galois sa to descent data}.

\begin{proposition}\label{Proposition equivalence datum action}
The functor $\mathcal{F}:\mathbf{Sch}_{L/k}\to \mathbf{Sch}(L/k)$, defined above, is an equivalence of categories.
\end{proposition}

We say that a Galois semilinear action on a variety is \textit{effective} if its respective Galois descent datum is effective. The following results are translations of \cref{Theorem14.70} and \cref{Proposition descent quasi-projective}.

\begin{theorem}
The functor
	\begin{align*}
		\mathcal{F}\circ\Phi: \mathbf{Sch}_{k} &\to \mathbf{Sch}(L/k) \\
		S &\mapsto (S_{L},\varrho_{\can}).
	\end{align*}
is fully faithful, where $\varrho_{\can,\gamma}=h_{\can,\gamma}\circ\alpha_{\gamma}^{-1}$. Moreover, if $(Z,\varrho)$ is an object in $\mathbf{Sch}(L/k)$ and $Z$ is covered by affine open subschemes that are stable under $\varrho$, then $\varrho$ is effective.
\end{theorem}

\begin{proposition}\label{Proposition descent quasi-projective semilinear}
The functor
	\begin{align*}
		\mathcal{F}\circ\Psi: \mathbf{QPSch}_{k} &\to \mathbf{QPSch}(L/k) \\
		S &\mapsto (S_{L},\varrho_{\can}).
	\end{align*}
is an equivalence of categories.
\end{proposition}

 Thus, we have the following result.

\begin{proposition}\label{Proposition descent quasi-projective semilinear action}
 The functor $\mathcal{F}\circ\Psi: \mathbf{QPSch}_{k} \to \mathbf{QPSch}(L/k)$ induces an equivalence of categories between the category of quasi-projective varieties over $k$ and the category of quasi-projective varieties over $L$ endowed with a $\Gamma_{L}$-semilinear action.
\end{proposition}

\begin{proof}
The functor $\mathcal{F}\circ\Psi: \mathbf{QPSch}_{k} \to \mathbf{QPSch}(L/k)$ is an equivalence of categories, since it is a composition of $\mathcal{F}$, which is an equivalence of categories, and $\Psi$, which is also an equivalence of categories by \cref{Proposition descent quasi-projective semilinear}. By \cite[Proposition 2.7.1]{EGAIV-II}, a geometrically integral scheme $S$ over $k$ is a variety if and only if $S_{L}$ is a variety over $L$. Then, the assertion holds.
\end{proof}

\begin{definition}
Let $Z$ be a scheme over $L$. A $k$-form of $Z$ is a pair $(S,\nu)$ of a scheme $S$ over $k$ and an isomorphism $\nu:S_{L}\to Z$ of schemes over $L$.
\end{definition}

\paragraph{Equivariant descent via semilinear actions.} An \emph{algebraic group} over a field $k$ is a variety $G$ over $k$ endowed with a morphism $m:G\times G\to G$ such that there exist morphisms $e_{G}:\Spec(k)\to G$ and $\mathrm{inv}:G\to G$ fitting in the following commutative diagrams 
\[\xymatrixcolsep{2.9pc}\xymatrix@R=2pc{ G\times G\times G \ar[r]^{\small{\id_{G}\times m}} \ar[d]_{m\times\id_{G}} & G\times G \ar[d]^{m} & \Spec(k)\times G \ar[r]^{e_{G}\times\id_{G}} \ar[rd]_{\cong}& G\times G \ar[d]^{m} & G\times\Spec(k) \ar[dl]^{\cong} \ar[l]_{\id_{G}\times e_{G}} \\ G\times G \ar[r]_{m} & G & & G & \\ & G \ar[r]^{\mathrm{inv}\times\id_{G}} \ar[d] & G\times G \ar[d]^{m} & \ar[l]_{\id_{G}\times \mathrm{inv}} G \ar[d] &\\ &\Spec(k) \ar[r]_{e_{G}} & G & \ar[l]^{e_{G}} \Spec(k). & }\] 
Let $G$ be an algebraic group over $L$. Given that $G$ is quasi-projective, every Galois descent datum is effective. In this case, we are considering just the Galois descent data given by semilinear group homomorphisms, or equivalently, by \cref{Proposition descent quasi-projective semilinear action}, a Galois semilinear action $\Gamma_{L}\to\SAut_{\mathrm{gr}}(G)$. This is because we are interested in the $k$-forms that are also algebraic groups.

For a $G$-variety $X$ over $L$, an \textit{equivariant Galois descent datum} is a pair of a Galois descent datum $\{\sigma_{\gamma}\}_{\gamma\in\Gamma_{L}}$ over $G$ and a Galois descent datum $\{h_{\gamma}\}_{\gamma\in\Gamma_{L}}$ over $X$ such that the following diagram commutes
\[\xymatrix{ \gamma G\times \gamma X \ar[r]^{\gamma\mu} \ar[d]_{(\sigma_{\gamma},h_{\gamma})} & \gamma X \ar[d]^{h_{\gamma}} \\ 
 G\times X \ar[r]_{\mu} & X\,, }\]
 where $\mu:G\times X\to X$ is the action. We say that an equivariant Galois descent datum is \textit{effective} if both Galois descent data are effective with $k$-forms $G_{0}$ of $G$ and $X_{0}$ of $X$ with $X_{0}$ a $G_{0}$-variety.
 
By \cref{Proposition descent quasi-projective semilinear action}, an equivariant Galois descent datum is equivalent to an equivariant Galois semilinear action as defined in \cref{definition semilinear equivariant action}. In particular, it is a group homomorphism
\[\varrho:\Gamma_{L}\to \SAut(G;X)\subset \SAut_{\mathrm{gr}}(G)\times\SAut(X),\]
such that the following diagram commutes
\[\xymatrix{  G\times X \ar[r]^{\mu} \ar[d]_{\varrho(\gamma):=(\varphi_{\gamma},f_{\gamma})} & X \ar[d]^{f_{\gamma}} \\ 
 G\times X \ar[r]_{\mu} & X\,. }\]

A Galois semilinear equivariant action $\varrho:\Gamma_{L}\to \SAut(G;X)$ naturally induces Galois semilinear actions $\varrho_{G}=\varphi$ on $G$ and $\varrho_{X}=h$ on $X$ under the respective projections.
 
 The Galois descent datum for $G$ is effective, and hence it always has a $k$-form $G_{0}$. Therefore, both pieces of descent data are effective when $X$ is a quasi-projective variety over $L$, which does not directly imply that the equivariant Galois descent datum is effective. However, the action also descends (see for instance: \cite[Lemma 5.4]{Bor20}).
 
\begin{proposition}\label{Proposition descend quasi-projective G-variaties}
Let $G$ be an algebraic group over $L$ and $X$ be a $G$-variety over $L$. Let $\varrho:\Gamma_{L}\to\SAut(G;X)$ be an equivariant $\Gamma_{L}$-semilinear action on $X$ and $G_{0}$ be a $k$-form of $G$. If $X$ is quasi-projective, then the descent is effective as a $G_{0}$-variety over $k$.
\end{proposition}

Let $G$ and $G'$ be algebraic groups over $L$. Let $X$ be a $G$-variety and $X'$ be a $G'$-variety, both over $L$. Let $\varrho$ and $\varrho'$ be effective equivariant $\Gamma_{L}$-semilinear actions on $X$ and $X'$, respectively. Denote by $(G_{0},X_{0})$ a $k$-form of the pair $(G,X)$ and by $(G_{0}',X_{0}')$ a $k$-form of the pair $(G',X')$. An equivariant morphism $(\varphi,f):X\to X'$, satisfying $\varrho(\gamma)\circ (\varphi,f)=(\varphi,f)\circ \varrho'(\gamma)$ for all $\gamma\in\Gamma_{L}$, descends to an equivariant morphism $(\varphi_{0},f_{0}):X_{0}\to X_{0}'$ (see \cite[Proposition 5.6]{Bor20}). Then, we have the following result.

\begin{proposition}\label{proposition effective quasi covering}
Let $G$ be an algebraic group over $L$ and $X$ be a $G$-variety over $L$. Let $\varrho:\Gamma_{L}\to \SAut(G;X)$ be a Galois semilinear equivariant action. If $X$ is covered by $G$-stable and $\Gamma_{L}$-stable quasi-projective open subvarieties, then the Galois semilinear equivariant action descends.
\end{proposition}

\begin{proof}
Let $\mathcal{U}:=\{X_{i}\}$ be a finite $G$-stable and $\Gamma_{L}$-stable quasi-projective open covering, which can be assumed stable under intersections because the intersection of quasi-projective varieties is quasi-projective, $G$-stable and $\Gamma_{L}$-stable. Since each quasi-projective subvariety $X_{i}$ is $G$-stable and $\Gamma_{L}$-stable, the Galois semilinear equivariant action $\varrho:\Gamma_{L}\to \SAut(G;X)$ induces Galois semilinear equivariant actions $\varrho_{i}:\Gamma_{L}\to \SAut(G;X_{i})$. Hence, by \cref{Proposition descend quasi-projective G-variaties}, each triple $(G, X_{i},\varrho_{i})$ has an effective descent $(G_{0,i},X_{0,i},(\varphi_{i},h_{i}))$. Given that each $\varrho_{i}$ induces the same Galois semilinear action on $G$, we have that $G_{0}=G_{0,i}$ and $\varphi=\varphi_{i}$ for each $X_{i}$. Then, the $k$-forms are of the form $(G_{0},X_{0,i},(\varphi,h_{i}))$ for each $(G,X_{i},\varrho_{i})$. 

Let us see that these $G_{0}$-varieties have a gluing data. For the intersection $X_{ij}:=X_{i}\cap X_{j}$, we have canonical $G$-equivariant open embeddings $\iota_{ij}: X_{ij}\to X_{i}$ and $\iota_{ji}:X_{ij}\to X_{j}$ which are compatible with the Galois semilinear equivariant actions $\varrho_{i}$, $\varrho_{j}$ and $\varrho_{ij}$. These morphisms descend to $G_{0}$-equivariant open embeddings $\eta_{ij}:X_{0,ij}\to X_{0,i}$ and $\eta_{ji}:X_{0,ij}\to X_{0,j}$ that satisfy the following commutative diagram
\[\xymatrixcolsep{2cm}\xymatrix{ X_{0,i}\times_{k} L \ar[r]^{(\varphi,h_{i})} & X_{i} \\ X_{0,ij}\times_{k} L \ar[r]_{(\varphi,h_{ij})} \ar[u]^{\eta_{ij}\times_{k}\id_{L}} & X_{ij} \ar[u]_{\iota_{ij}}.  }\]

From the morphisms $\eta_{ij}$ and $\eta_{ji}$, we have $G_{0}$-equivariant isomorphisms $\xi_{ij}:=\eta_{ji}\circ\eta_{ij}^{-1}:\textrm{Im}(\eta_{ij})\to \textrm{Im}(\eta_{ji})$, which satisfy: 
\[\xi_{jl}\circ\varphi_{ij}=\xi_{il} \quad\mathrm{ and }\quad \xi_{ij}^{-1}(X_{0,ji}\cap X_{0,jl})=X_{0,ij}\cap X_{0,il}.\]  

Since the $\xi_{ij}$ and the $X_{0,ij}$ form a glueing data of varieties, by \cite[\href{https://stacks.math.columbia.edu/tag/01JC}{Tag 01JC}]{stacks-project}, the following quotient space: 
\[X_{0}:=\left(\bigsqcup_{X_{0,i}\in\mathcal{U}_{0}} X_{0,i}\right)/\sim,\]
where the relation is given by $x\sim y$ if and only if for some $\xi_{ij}$ we have $\xi_{ij}(x)=y$, is a scheme. The canonical $G_{0}$-equivariant embeddings $X_{0,i}\to X_{0}$ fit into the following commutative diagram
\[\xymatrix{ X_{0,i} \ar[r] & X_{0} \\ X_{0,ij} \ar[r]_{\eta_{ji}} \ar[u]^{\eta_{ij}} \ar[ru] & X_{0,j} \ar[u] \,.} \] 
Also, notice that there is a canonical $G_{0}$-equivariant isomorphism
\[X_{0}\times_{k} L\cong \left(\bigsqcup_{X_{0,i}\in\mathcal{U}_{0}} X_{0,i}\times_{k} L\right)/\sim \]
where the relation is given by $x \sim y$ if and only if for some $\xi_{ij} \times_{k} \id_{L}$ we have $\xi_{ij} \times_{k} \id_{L}(x) = y$. Now, let us take 
\[(\varphi, \tilde{h}) : \bigsqcup_{X_{0,i} \in \mathcal{U}_{0}} X_{0,i} \times_{k} L \to \bigsqcup_{X_{i} \in \mathcal{U}} X_{i},\]
the morphism induced by the $(\varphi, h_{i}) : X_{0,i} \times_{k} L \to X_{i}$. Notice that if for $x$ and $y$ there exists $(\xi_{ij} \times_{k} \id_{L})(x) = y$, then there exists $z \in X_{0,ij}$ such that $(\eta_{ij} \times_{k} \id_{L})(z) = x$ and $(\eta_{ji} \times_{k} \id_{L})(z) = y$. Thus, 
\[\tilde{h}(x) = h_{i}(x) = h_{i}((\eta_{ij} \times_{k} \id_{L})(z)) = h_{j}((\eta_{ji} \times_{k} \id_{L})(z)) = h_{j}(y) = \tilde{h}(y).\]
This implies that $(\varphi, \tilde{h})$ induces a morphism $(\varphi, h) : X_{0} \times_{k} L \to X$, which is indeed an equivariant isomorphism. Hence, $(X_{0}, G_{0}, (\varphi, h))$ is a $k$-form for $(X, G)$. Given that $X$ is a variety over $L$, we have that $X_{0}$ is a variety over $k$ by \cite[Proposition 2.7.1]{EGAIV-II}. Thus, the Galois semilinear equivariant action is effective and the assertion holds.
\end{proof}

This allows us to prove the following result.

\begin{proposition}\label{proposition equivalence normal G-varieties}
Let $G$ be a connected algebraic group over $L$ and $\varphi$ be a Galois semilinear action. Let $G_{0}$ be the algebraic group over $k$ associated with $(G,\varphi)$. Then, the functor $\Phi: \mathbf{Sch}_{k} \to \mathbf{Sch}(L/k)$, given by $S\mapsto (S_{L},\varrho_{\mathrm{can}})$, induces an equivalence of categories between the category of normal varieties with effective $G_{0}$-actions and the category of normal varieties over $L$ with effective $G$-actions endowed with Galois semilinear equivariant actions, which are covered by $G$-stable and $\Gamma_{L}$-stable quasi-projective subvarieties and $\varrho_{G}$ is isomorphic to $\varphi$.  
\end{proposition}

\begin{remark}
The reader should be warned that morphisms in these categories are given by pairs of morphisms $(\varphi,f)$, where $\varphi$ is a morphism of algebraic groups and $f$ is a morphism of varieties. In particular, even if we fix a group $G$, a morphism might not be the identity on $G$, so the latter is not a subcategory of the category of $G$-varieties with $G$-equivariant morphisms. This is actually crucial in order to let $\Gamma_{L}$ act semilinearly on it.
\end{remark}

\begin{proof}
We give the equivalence at the level of objects. The equivalence at the level of morphisms will follow from \cref{Proposition descent quasi-projective semilinear action}. 

Let $X'$ be a normal $G'$-variety over $k$. By \cite[Theorem 1]{Bri17}, $X'$ is covered by $G'$-stable quasi-projective open subvarieties over $L$. Hence, $X_{L}:=X'\times_{k} L$ is a normal $G'_{L}$-variety over $L$ covered by $G'_{L}$ and $\Gamma_{L}$-stable quasi-projective subvarieties. Then, $X_{L}$ has a compatible structure of $G$-variety under the isomorphism $\nu:G'_{L}\to G$.

The other direction is given by \cref{proposition effective quasi covering}.
\end{proof}

\subsection{ Proof of \cref{maintheoremofpaper}}\label{section: affine minimal}
\indent

In this section, before getting into the proof of \cref{maintheoremofpaper}, we define the notion of \textit{Galois semilinear actions on} pp-divisors. Also a new category is presented; the category of pairs $(\D,\varrho)$ of a pp-divisor and a Galois semilinear action. Once the proof of the theorem is achieved, we show how to recover Gillard's Theorem (cf. \cite[Theorem A]{Gil22b}) from \cref{maintheoremofpaper}.

\begin{definition}
Let $H$ be a group. Let $\D\in\PPDiv_{\Q}(Y,\omega)$ be a pp-divisor over $L$. A $H$-\textit{semilinear action on} $\D$ is a group homomorphism $\varrho:H\to\SAut(\D)$.
\end{definition}

Let $H$ be an abstract group. A $H$-semilinear action $\varrho:H\to\SAut(\D)$ induces a $H$-semilinear equivariant action (recall \cref{definition semilinear equivariant action})
\[X(\varrho):H\to \SAut(T;X(\D)),\]
via the functor $X:\mathfrak{PPDiv}(L/k)\to \mathcal{E}(L/k)$. Given that $\D$ is a minimal pp-divisor, every $H$-semilinear equivariant action $\varrho:H\to\SAut(T;X(\D))$ arises from a $H$-semilinear action on $\D$ by \cref{corollary saut D saut XD}. Actually, this defines a bijection between the set of semilinear actions over $\D$ and the set of semilinear equivariant actions over $X(\D)$.

\begin{proposition}\label{proposition bijection semilinear actions}
Let $\D$ be an object in $\mathfrak{PPDiv}(L/k)$ that is minimal. Then, the functor $X(\bullet)$ induces a bijection between the set of semilinear actions over $\D$ and the set of semilinear equivariant actions over $X(\D)$.
\end{proposition}

\begin{proof}
This is consequence of \cref{corollary saut D saut XD}, because it implies that the following commutative diagram can be always completed in a unique way
\[\xymatrix{ & H \ar[dl]_{\varrho_{1}} \ar[dr]^{\varrho_{2}} & \\ \SAut(\D) \ar[rr]_{X}^{\cong} & & \SAut(T;X(\D)) . }\]
Otherwise stated, having $\varrho_{1}$ we can construct a unique $\varrho_{2}$ and having $\varrho_{2}$ there exists a unique $\varrho_{1}$.
\end{proof}

Let $\mathfrak{PPDiv}(\Gamma_{L})$ the category of pairs $(\D,\varrho)$, where $\D$ is a minimal pp-divisor over $L$ and $\varrho:\Gamma_{L}\to\SAut(\D)$ is a Galois semilinear action. A morphism in this category is a morphism of pp-divisors $(\psi,F,\mathfrak{f}):\D\to\D'$ such that 
\[\varrho'_{\gamma}\circ (\psi,F,\mathfrak{f})=(\psi,F,\mathfrak{f})\circ \varrho_{\gamma}\]
 for every $\gamma\in\Gamma_{L}$. Let $(\D,\varrho)$ be an object in $\mathfrak{PPDiv}(\Gamma_{L})$. By \cref{theorem Altmann hausen split}, $X(\D)$ is a normal $T_{\D}$-variety over $L$, where $T_{\D}$ denotes its respective torus action. Moreover, by \cref{proposition bijection semilinear actions}, $X(\D)$ comes with a Galois semilinear equivariant automorphisms 
 \[X(\varrho):\Gamma_{L}\to\SAut(T_{\D};X(\D)).\]
 Then, by \cref{proposition equivalence normal G-varieties}, there exists a normal $T$-variety $X:=X(\D,\varrho)$ over $k$ such that $X_{L}\cong X(\D)$ as $T_{\D}$-varieties over $L$. This proves the first part of the following theorem.

\begin{theorem}\label{theorem main affine minimal}
Let $k$ be a field, $L/k$ be a finite Galois extension with Galois group $\Gamma_{L}$.
	\begin{enumerate}[a)]
            \item\label{theorem main affine minimal part i} Let $(\D_{L},\varrho)$ be an object in $\mathfrak{PPDiv}(\Gamma_{L})$. Then, $X(\D_{L},\varrho)$ is an affine normal variety endowed with an effective action of an algebraic torus $T$ over $k$ such that $T$ splits over $L$ and $X(\D_{L},\varrho)_{L}\cong X(\D_{L})$ as $T_{\D_{L}}$-varieties over $L$.
            \item\label{theorem main affine minimal part ii} Let $X$ be an affine normal variety over $k$ endowed with an effective $T$-action such that $T_{L}$ is split. Then, there exists an object $(\D_{L},\varrho)$ in $\mathfrak{PPDiv}(\Gamma_{L})$ such that $X \cong X(\D_{L},\varrho)$ as $T$-varieties.
        \end{enumerate}
\end{theorem}

\begin{proof}
Let us prove part \eqref{theorem main affine minimal part ii}, the remaining part of the theorem. Let $X$ be a normal variety over $k$ endowed with an effective $T$-action. By \cref{proposition equivalence normal G-varieties}, as a $T$-variety over $k$, $X$ is equivalent to a pair $(X_{L},\varrho')$, where $X_{L}$ is a normal $T_{L}$-variety, with $T_{L}$ split over $L$, and an equivariant $\Gamma_{L}$-semilinear action $\varrho'$. By \cref{proposition existence ah-datum split}, there exists a pp-divisor $\D$ such that $X_{L}\cong X(\D)$ as $T_{L}$-varieties over $L$. This pp-divisor, by the proof of \cref{proposition existence ah-datum split}, can be chosen minimal. Now, by \cref{proposition bijection semilinear actions}, we have that the equivariant $\Gamma_{L}$-semilinear action on $X(\D_{L})$ induces a unique $\Gamma_{L}$-semilinear action $\varrho$ on $\D_{L}$. Then, the pair $(\D,\varrho)$ encodes the pair $(X_{L},\varrho')$. Hence, there exists a pair $(\D,\varrho)$ in $\mathfrak{PPDiv}(\Gamma_{L})$ such that $X\cong X(\D,\varrho)$ as $T$-varieties.
\end{proof}

By \cref{theorem main affine minimal}, every pair $(\D,\varrho)$ corresponds to an affine normal variety $X(\D,\varrho)$ endowed with a torus action on $k$ that is split over $L$. This construction induces a functor 
 \begin{align*}
 X:\mathfrak{PPDiv}(\Gamma_{L}) &\to \mathcal{E}(k,L);\\
  (\D,\varrho) &\mapsto X(\D,\varrho),
 \end{align*}
where $\mathcal{E}(k,L)$ is the category of affine normal varieties over $k$ endowed with an effective action of an algebraic torus over $k$ that is split over $L$. This functor is the composition of the functor $(\D,\varrho)\mapsto (X(\D),X(\varrho))$, from the category $\mathfrak{PPDiv}(\Gamma_{L})$ to the category of affine normal varieties endowed with an effective action of a split algebraic torus over $L$ and an equivariant $\Gamma_{L}$-semilinear action, and the equivalence of categories of \cref{proposition equivalence normal G-varieties}. Given that the first functor is faithful, covariant and essentially surjective, we have the following.

\begin{proposition}\label{theorem equivalence categories descent}
The functor $X:\mathfrak{PPDiv}(\Gamma_{L}) \to \mathcal{E}(k,L)$ is covariant, faithful and essentially surjective.
\end{proposition}

\begin{remark}
Let $X$ be an object in $\mathcal{E}(k,L)$ with torus $T$. By \cref{theorem main affine minimal}, there exists a minimal pp-divisor $\D\in\PPDiv_{\Q}(Y,\omega)$ and a Galois semilinear action $\varrho:\Gamma_{L}\to\SAut(T_{L};X(\D))$ such that $X(\D,\varrho)\cong X$ as $T$-varieties over $k$. Notice that the Galois semilinear action $\varrho$ induces a Galois semilinear action $\psi:\Gamma_{L}\to \SAut(Y)$. Given that $Y$ is semi-projective and hence quasi-projective, the Galois semilinear action $\psi$ is effective. Hence, there exists a semi-projective variety $Z$ over $k$ such that $Z_{L}\cong Y$.
\end{remark}

\textbf{Recovering Gillard's Theorem.} Let $(\D,\varrho)$ be an object in $\mathfrak{PPDiv}(\Gamma_{L})$ such that $\D$ is a minimal pp-divisor. Recall that for every $\gamma\in\Gamma_{L}$, $\varrho_{\gamma}:=(\psi_{\gamma},F_{\gamma},\mathfrak{f}_{\gamma}):\D\to\D$ is a semilinear automorphism of pp-divisors and 
\[\varrho_{\gamma_{2}\gamma_{1}}=(\psi_{\gamma_{2}\gamma_{1}},F_{\gamma_{2}\gamma_{1}},\mathfrak{f}_{\gamma_{2}\gamma_{1}})=(\psi_{\gamma_{2}}\psi_{\gamma_{1}},F_{\gamma_{2}}F_{\gamma_{1}},F_{\gamma_{2}*}(\mathfrak{f}_{\gamma_{1}})\psi_{\gamma_{1}}^{*}(\mathfrak{f}_{\gamma_{2}}))=\varrho_{\gamma_{2}}\varrho_{\gamma_{1}},\] 
for every $\gamma_{1},\gamma_{2}\in\Gamma_{L}$. If we define $h_{\gamma}:=\mathfrak{f}_{\gamma}\circ F_{\gamma^{-1}}^{*}$, where we view $\mathfrak{f}_{\gamma}$ as a morphism $M\to L(Y)^{*}$ and $F_{\gamma}^{*}:M\to M$ is the dual map of $F_{\gamma}$, we have 
\begin{align*}
h_{\gamma_{2}\gamma_{1}}\circ F_{\gamma_{2}\gamma_{1}}^{*} &=\mathfrak{f}_{\gamma_{2}\gamma_{1}} \\
&= F_{\gamma_{2}*}(\mathfrak{f}_{\gamma_{1}})\cdot \psi_{\gamma_{1}}^{*}(\mathfrak{f}_{\gamma_{2}}) \\
&=(\mathfrak f_{\gamma_1}\circ F_{\gamma_2}^*)\cdot \psi_{\gamma_{1}}^{*}(\mathfrak{f}_{\gamma_{2}}) \\
&= (h_{\gamma_{1}}\circ F_{\gamma_{1}}^{*}\circ F_{\gamma_{2}}^{*})\cdot\psi_{\gamma_{1}}^{*}( h_{\gamma_{2}}\circ F_{\gamma_{2}}^{*}) \\
&= (h_{\gamma_{1}}\circ F_{\gamma_{2}\gamma_{1}}^{*})\cdot \psi_{\gamma_{1}}^{*}( h_{\gamma_{2}}\circ F_{\gamma_{1}^{-1}}^{*}\circ F_{\gamma_{2}\gamma_{1}}^{*})\\
&= (h_{\gamma_{1}}\cdot \psi_{\gamma_{1}}^{*}( h_{\gamma_{2}}\circ F_{\gamma_{1}^{-1}}^{*}))\circ F_{\gamma_{2}\gamma_{1}}^{*}.
\end{align*}
Thus, the maps $h_{\gamma}:M\to L(Y)^{*}$ satisfy
\[h_{\gamma_{2}\gamma_{2}}=h_{\gamma_{1}}\cdot \psi_{\gamma_{1}}^{*}(h_{\gamma_{2}}\circ F_{\gamma_{1}^{-1}}^{*}),\] 
for every $\gamma_{1},\gamma_{2}\in\Gamma_{L}$. This condition corresponds to one of \cite[Theorem A]{Gil22b}. The other condition is fulfilled by \cref{corollaryequivariantautaffine}. Then, we recover Gillard's Theorem.

\begin{example}[\cref{exampleAH06 11 1} revisited]
Let $k$ be a field and $L/k$ be a quadratic extension with Galois group $\Gamma_{L}$. The affine threefold $X:=\Spec(L[x,y,z,w]/(x^{3}+y^{4}+zw))$ in $\A_{L}^{4}$ with the action of $\G_{\mathrm{m},L}^{2}$ given by 
\[(\lambda,\mu)\cdot (x,y,z,w)=(\lambda^{4} x,\lambda^{3} y,\mu z,\lambda^{12}\mu^{-1}w)\]
is encoded by the pp-divisor $\D:=\Delta_{0}\otimes \{0\}+\Delta_{1}\otimes \{1\}+\Delta_{\infty}\otimes\{\infty\}$, where 
\[\Delta_{0}=\left(0,\dfrac{1}{3}\right)+\omega, \quad \Delta_{1}=\left(-\dfrac{1}{4},0\right)+\omega, \quad \Delta_{\infty}=\left(\{0\}\times[0,1]\right)+\omega\]
and $\omega=\cone((1,0),(1,12))$.

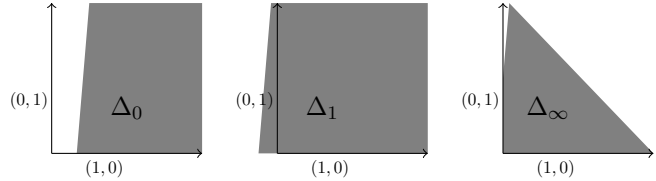
\begin{center}
    \begin{tikzpicture}
        \fill[gray] (1/3,0) -- (1/2,2) -- (2,2) -- (2,0) -- (1/3,0);
        \node[scale=0.6] at (0.7,-0.2) {$(1,0)$};
        \node[scale=0.6] at (-0.3,0.7) {$(0,1)$};
        \draw[->] (0,0) -- (2,0);
        \draw[->] (0,0) -- (0,2);
        \node[scale=0.9] at (1,0.6) {$\Delta_{0}$};
        
        \fill[gray] (11/4,0) -- (35/12,2) -- (5,2) -- (5,0);
        \node[scale=0.6] at (3.7,-0.2) {$(1,0)$};
        \node[scale=0.6] at (2.7,0.7) {$(0,1)$};
        \draw[->] (3,0) -- (5,0);
        \draw[->] (3,0) -- (3,2);
        \node[scale=0.9] at (3.6,0.6) {$\Delta_{1}$};
        
        \fill[gray] (6,0) -- (6,1) --(73/12,2) -- (8,0) -- (6.7,0);
        \node[scale=0.6] at (6.7,-0.2) {$(1,0)$};
        \node[scale=0.6] at (5.7,0.7) {$(0,1)$};
        \draw[->] (6,0) -- (8,0);
        \draw[->] (6,0) -- (6,2);
        \node[scale=0.9] at (6.6,0.6) {$\Delta_{\infty}$};
    \end{tikzpicture}
\end{center} 

We claim that this affine normal $T$-variety has no nontrivial $k$-forms. Let $X'$ be a $k$-form of $X$ as a $T$-variety. It means that $X'$ is endowed with an effective action  of a $k$-form $T'$ of $T$. By \cref{theorem main affine minimal}, there exists a Galois semilinear action $\Gamma_{L}\to \SAut(\D)$ given by $(\psi_{\gamma},F_{\gamma},\mathfrak{f}_{\gamma})$, where $\gamma$ is the nontrivial element of $\Gamma_{L}$. Since $(\psi_{\gamma},F_{\gamma},\mathfrak{f}_{\gamma})$ is a semilinear automorphism of $\D$, it holds that $F_{\gamma}(\omega)=\omega$ and 
\[\psi_{\gamma}^{*}\D=\D+\mathrm{div}(\mathfrak{f}).\] 

Let us prove our claim. It is known that the $k$-forms of $\G_{\mathrm{m},L}^{2}$ are
\[\left(\G_{\mathrm{m},k}\right)^{2},\quad \G_{\mathrm{m},k}\times \mathrm{R}_{L/k}^{1}(\G_{\mathrm{m},L}),\quad \left(\mathrm{R}_{L/k}^{1}(\G_{\mathrm{m},L})\right)^{2}\quad\textrm{and}\quad \mathrm{R}_{L/k}(\G_{\mathrm{m},L}),\]
where $\mathrm{R}_{L/k}(\G_{\mathrm{m},L})$ is the Weil restriction and $\mathrm{R}_{L/k}^{1}(\G_{\mathrm{m},L})$ is its respective norm one subtorus. Their respective Galois descent data $\Gamma_{L}\to\SAut(\G_{\mathrm{m},L}^{2})$ are encoded by one of the following group homomorphisms $F:\Gamma_{L}\to \Aut(N)$:
\[ F_{\gamma}\in\left\{ \left(\begin{matrix} 1 & 0 \\ 0 & 1 \end{matrix}\right), \quad \left(\begin{matrix} 1 & 0 \\ 0 & -1 \end{matrix}\right), \quad \left(\begin{matrix} -1 & 0 \\ 0 & -1 \end{matrix}\right)\quad \mathrm{and} \quad \left(\begin{matrix} 0 & 1 \\ 1 & 0 \end{matrix}\right) \right\}.\]
Since the only $F_{\gamma}$ that preserves the tail cone $\omega=\cone((1,0),(1,12))$ is $F=\id_{N}$. Thus, the Galois semilinear action is given by $(\psi_{\gamma},\id_{N},\mathfrak{f})$. 

Let us now prove that $\psi_{\gamma}:\P^{1}_{L}\to \P^{1}_{L}$ is given by $\psi_{\gamma}([x:y])=[\gamma(x):\gamma(y)]$. We argue by contradiction. Let us assume that $\psi_{\gamma}([x:y])\neq [\gamma(x):\gamma(y)]$. This implies that the induced action of $\psi_{\gamma}$ on $\{\{0\},\{1\},\{\infty\}\}$ must fix a point. Since $\psi_{\gamma}^{*}\D=\D+\mathrm{div}(\mathfrak{f})$, we have that one of the following cases hold
\begin{itemize}
	\item $\Delta_{0}=\Delta_{1}+m$ for some $m\in\Z^{2}$,
	\item $\Delta_{1}=\Delta_{\infty}+m$ for some $m\in\Z^{2}$ or
	\item $\Delta_{\infty}=\Delta_{0}+m$ for some $m\in\Z^{2}$,
\end{itemize}
depending on which prime divisor is fixed by $\psi_{\gamma}$. The first case does not hold. This is because the existence of such an $m\in\Z^{2}$ is equivalent to the equation
	\[\left(0,\frac{1}{3}\right)=\left(-\frac{1}{4},0\right)+m’\]
having a solution in $\Z^{2}$. The remaining two cases do not hold either, since one of the polytope is a segment and the other one is a point. Then, $\psi_{\gamma}([x:y])=[\gamma(x):\gamma(y)]$ in $\P^{1}_{L}$ and, therefore, we have that the only possible Galois semilinear action is $(\psi_{\gamma},\id_{N},\mathfrak{1})$ with $\psi_{\gamma}([x:y])\neq [\gamma(x):\gamma(y)]$. Hence, the only $k$-form of $X$ is the affine threefold $\Spec(k[x,y,z,w]/(x^{3}+y^{4}+zw))$ in $\A_{k}^{4}$ with the action of $\G_{\mathrm{m},k}^{2}$ given by the same action.
 
\end{example}

\begin{example}[\cref{exampleA3} revisited]
Let $k$ be a field and $L/k$ be a quadratic extension with Galois group $\Gamma_{L}$. The affine space $\mathbb{A}_{L}^{3}$ endowed with the action of $\G_{\mathrm{m},L}^{2}$ given by
\[(\lambda,\mu)\cdot (x,y,z)=(\lambda x,\mu y,\lambda\mu z)\]
arises from the pp-divisor $\Delta:=\Delta\otimes\{\infty\}$ over $\mathbb{P}_{L}^{1}$, where $\Delta$ is the polyhedron 
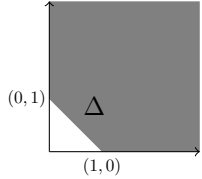
\begin{center}
    \begin{tikzpicture}
        \fill[gray] (0,0.7) -- (0,2) -- (2,2) -- (2,0) -- (0.7,0);
        \node[scale=0.6] at (0.7,-0.2) {$(1,0)$};
        \node[scale=0.6] at (-0.3,0.7) {$(0,1)$};
        \draw[->] (0,0) -- (2,0);
        \draw[->] (0,0) -- (0,2);
        \node[scale=0.9] at (0.6,0.6) {$\Delta$};
    \end{tikzpicture}
\end{center}
 The quotient map $\mathbb{A}_{L}^{3}\dasharrow \mathbb{P}_{L}^{1}$ is given by $(x,y,z)\mapsto(z,xy)$.
 Let us consider the following Galois semilinear equivariant action on $\mathbb{A}_{L}^{3}$:
 \begin{align*}
 \mathbb{A}_{L}^{3} &\to \mathbb{A}_{L}^{3} \\
 (x,y,z) &\mapsto (\gamma(y),\gamma(x),\gamma(z)). 
 \end{align*}
In the torus, the Galois semilinear action is given by $(\lambda,\mu)\mapsto(\gamma(\mu),\gamma(\lambda))$. In terms of the pp-divisor, the Galois semilinear action is given by $(\psi_{\gamma},F,\mathfrak{f})$, with $\psi_{\gamma}([v:w])=[\gamma(w):\gamma(v)]$, $F(a,b)=(b,a)$ and $\mathfrak{f}=\mathfrak{1}$. Notice that 
\[ \Delta\otimes\{\infty\}=\psi_{\gamma}^{*}\Delta=F_{*}\Delta=\Delta\otimes\{\infty\}. \]
Then the descent as a $T$-variety is effective by \cref{theorem main affine minimal}. Now, the semilinear equivariant action on $\A_{L}^{3}$ is given by an equivariant semilinear action in $\A_{L}^{2}$ and another one over $\A_{L}^{1}$. Given that the only separable $k$-form of $\A_{L}^{2}$ is the affine plane by \cite[Theorem 3]{Kam75}, the corresponding $k$-form of $\A_{L}^{3}$ is $\A_{k}^{3}$. For the torus action, the respective $k$-form is $\mathrm{Res}_{L/k}(\G_{\mathrm{m},L})$.
\end{example}

\subsection{ The other T-variety}\label{section: The other T-variety}
\indent

Let $\D\in\PPDiv_{\Q}(Y,\omega)$ be a pp-divisor over $L$. By \cref{proposition ppdiv to variety}, $X(\D)$ is an affine normal variety endowed with an effective action of $T_{\D}$. Also by \cref{proposition ppdiv to variety} we know there is another variety related to $\D$. Recall that from a pp-divisor $\D$ we can construct the $M$-graded sheaf 
\[\mathscr{A}(Y,\D):=\bigoplus_{m\in\omega^{\vee}\cap M}\mathscr{O}_{Y}(\D(m)).\]
The other variety associated to $\D$ is $\tilde{X}(\D):=\Spec_{Y}(\mathscr{A}(Y,\D))$, the relative spectrum of the sheaf $\mathscr{A}(Y,\D)$. This variety is a normal $T_{\D}$-variety whose affinization is $X(\D)$. Moreover, the affinization $r_{\tilde{X}}:\tilde{X}(\D)\to X(\D)$ is proper, birational and it fits into the following commutative diagram
\[\xymatrix{ \tilde{X}(\D) \ar[r]^{r_{\tilde{X}}} \ar[d]_{\sslash T}& X(\D) \ar[d]^{\sslash T} \\ Y \ar[r]_{r_{Y}} & \aff{Y}. }\]
If $Y$ is not affine, then $\tilde{X}(\D)$ and $X(\D)$ are not isomorphic. In particular, in such a case, it follows that $\tilde{X}(\D)$ is not an affine $T$-variety. Thus, somehow pp-divisors are encoding more than just affine normal $T$-varieties. For example, as $\G_{\mathrm{m},k}$-varieties, $\mathrm{Bl}_{0}(\A_{k}^{2})\cong \tilde{X}(\D)$, where $\D=[0]\times\{0\}\in\PPDiv_{\Q}(\P^{1},\{0\})$. Also, another interesting property of the variety $\tilde{X}(\D)$, in contrast to $X(\D)$, is that every orbit closure of a $k$-rational point in $\tilde{X}(\D)$ is normal, whereas $X(\D)$ might have singular orbit closures (\cite[Theorem 10.1]{AH06}). Looking to this last feature, one might ask whether this variety can be used to study the resolution of singularities in positive characteristic in the case of $T$-varieties, since it has been used in characteristic zero \cite{LS13}.

In the last part of this section, we study some properties of $\tilde{X}(\D)$ such as the understanding of morphisms and Galois descent.

Let $(\psi_{\gamma},F,\mathfrak{f}):\D'\to\D$ be a semilinear morphism of pp-divisors, then by definition
\[\psi_{\gamma}^{*}\D\leq F_{*}\D'+\mathrm{div}(\mathfrak{f}).\] 
This triple gives a morphism of sheaves
\begin{align*}
\mathscr{O}_{Y}(\D(m)) &\to \mathscr{O}_{Y'}(\D'(F^{*}(m))) \\
g &\mapsto \mathfrak{f}(m)\psi_{\gamma}^{*}(g),
\end{align*}
which fits into a $M$-graded morphism of algebras 
\[\mathscr{A}(Y,\D):=\bigoplus_{m\in\omega^{\vee}\cap M}\mathscr{O}_{Y}(\D(m))\to \bigoplus_{m\in\omega'^{\vee}\cap M}\mathscr{O}_{Y'}(\D'(m))=\mathscr{A}(Y',\D'). \]
The latter morphism induces a semilinear equivariant morphism of varieties 
\[\tilde{X}(\psi_{\gamma},F,\mathfrak{f}):\tilde{X}(\D')\to\tilde{X}(\D)\] 
that fits into the following commutative diagram
\[\xymatrixcolsep{5pc}\xymatrix{ \tilde{X}(\D') \ar[r]^{\tilde{X}(\psi_{\gamma},F,\mathfrak{f})} \ar[d]_{\sslash T} & \tilde{X}(\D) \ar[d]^{\sslash T} \\ Y' \ar[r]_{\psi_{\gamma}} & Y .}\]

\begin{proposition}\label{proposition other functor}
Let $\D$ and $\D'$ be objects in $\mathfrak{PPDiv}(L/k)$ and $(\psi_{\gamma},F,\mathfrak{f}):\D'\to\D$. Then, the semilinear equivariant morphism $\tilde{X}(\psi_{\gamma},F,\mathfrak{f}):\tilde{X}(\D')\to\tilde{X}(\D)$ satisfies 
\[\tilde{X}(\psi_{\gamma},F,\mathfrak{f})_{\mathrm{aff}}=X(\psi_{\gamma},F,\mathfrak{f}).\]
Moreover, if $(\psi_{\gamma},F,\mathfrak{f}):\D'\to\D$ is a semilinear isomorphism, then $\tilde{X}(\psi_{\gamma},F,\mathfrak{f}):\tilde{X}(\D')\to\tilde{X}(\D)$ is a semilinear equivariant isomorphism.
\end{proposition}

\begin{proof}
Let $(\psi_{\gamma},F,\mathfrak{f}):\D'\to\D$ be a semilinear morphism of pp-divisors. For every $m\in\omega^{\vee}\cap M$, the morphism of sheaves $\mathscr{O}_{Y}(\D(m))\to \mathscr{O}_{Y'}(\D'(F^{*}(m)))$, given by $g\mapsto \mathfrak{f}(m)\psi_{\gamma}^{*}(g)$, induces the morphism 
\begin{align*}
\H^{0}(Y,\mathscr{O}_{Y}(\D(m))) &\to\H^{0}(Y',\mathscr{O}_{Y'}(\D'(F^{*}(m)))) \\
h &\mapsto \mathfrak{f}(m)\psi_{\gamma}^{*}(h),
\end{align*}
between the global sections. Then, the morphism of sheaves $\mathscr{A}\to\mathscr{A}$ induces a morphism of algebras $A[Y,\D]\to A[Y',\D]$, which is the algebraic counterpart of $X(\psi_{\gamma},F,\mathfrak{f})$. Thus, we have that 
\[\tilde{X}(\psi_{\gamma},F,\mathfrak{f})_{\mathrm{aff}}=X(\psi_{\gamma},F,\mathfrak{f}).\]

If $(\psi_{\gamma},F,\mathfrak{f})$ is a semilinear isomorphism of pp-divisors, then $\psi_{\gamma}:Y'\to Y$ is a semilinear isomorphism and, therefore, $\psi_{\gamma}^{*}:L(Y)\to L(Y')$ is an automorphism. Thus, the morphism $\mathscr{A}\to\mathscr{A}'$ is an isomorphism. Hence, $\tilde{X}(\psi_{\gamma},F,\mathfrak{f})$ is a semilinear equivariant isomorphism. 
\end{proof}

\begin{proposition}\label{proposition automor X automor tildeX}
Let $\D$ be an object in $\mathfrak{PPDiv}(L/k)$, which is minimal. Then, for every semilinear equivariant automorphism $(\varphi_{\gamma},f_{\gamma}):X(\D)\to X(\D)$ there exists a semilinear equivariant automorphism $(\tilde{\varphi}_{\gamma},\tilde{f}_{\gamma}):\tilde{X}(\D)\to\tilde{X}(\D)$ such that $(\tilde{\varphi}_{\gamma},\tilde{f}_{\gamma})_{\mathrm{aff}}=(\varphi_{\gamma},f_{\gamma})$.
\end{proposition}

\begin{proof}
Let $(\varphi_{\gamma},f_{\gamma}):X(\D)\to X(\D)$ be a semilinear equivariant isomorphism. Given that $\D$ is minimal, by \cref{theorem semilinear automorphisms}, there exists a semilinear automorphism of pp-divisors $(\psi_{\gamma},F,\mathfrak{f}):\D\to\D$ such that $X(\psi_{\gamma},F,\mathfrak{f})=(\varphi_{\gamma},f_{\gamma})$. Hence, by \cref{proposition other functor}, $(\tilde{\varphi}_{\gamma},\tilde{f}_{\gamma}):=\tilde{X}(\psi_{\gamma},F,\mathfrak{f})$ satisfies $(\tilde{\varphi}_{\gamma},\tilde{f}_{\gamma})_{\mathrm{aff}}=(\varphi_{\gamma},f_{\gamma})$.
\end{proof}

Let $T$ be an algebraic torus over $k$ that splits over $L$ and $X$ be an affine normal $T$-variety over $k$. By \cref{theorem main affine minimal}, there exists a pair $(\D_{L},\varrho)$ in $\mathfrak{PPDiv}(\Gamma_{L})$ such that $X(\D_{L},\varrho)\cong X$ as $T$-varieties. As in \cref{remark diagram affinization split actions}, over $L$, we have the following commutative diagram 
\[\xymatrix{ \tilde{X}(\D_{L}) \ar[r]^{r_{\tilde{X}}} \ar[d]_{\sslash T_{L}}& X(\D_{L}) \ar[d]^{\sslash T_{L}} \\ Y \ar[r]_{r_{Y}} & \aff{Y}. }\]

The Galois semilinear action $\varrho:\Gamma_{L}\to\SAut(\D_{L})$ is equivalent to a Galois semilinear equivariant action $X(\varrho):\Gamma_{L}\to\SAut(T_{L};X_{L})$. By \cref{proposition automor X automor tildeX}, $\varrho:\Gamma_{L}\to\SAut(\D_{L})$ induces a Galois semilinear equivariant action $\tilde{X}(\varrho):\Gamma_{L}\to\SAut(T_{L};\tilde{X}(\D_{L}))$ such that $\tilde{X}(\varrho)_{\mathrm{aff}}=X(\varrho)$. Recall that the Galois semilinear action $\varrho:\Gamma_{L}\to\SAut(\D)$ defines a Galois semilinear action $\psi:\Gamma_{L}\to \SAut(Y)$ and hence $\psi_{\mathrm{aff}}:\Gamma_{L}\to\SAut(\aff{Y})$. If we denote by $\tilde{\pi}:\tilde{X}(\D)\to Y$ and $\pi:X(\D)\to \aff{Y}$ the respective quotients, we have that $\psi_{}\circ\tilde{\pi}=\tilde{\pi}\circ \tilde{X}(\varrho)$ and $\psi_{\mathrm{aff}}\circ\pi=\pi\circ X(\varrho)$. Thus, the diagram \emph{has} a Galois semilinear equivariant action, i.e. the Galois semilinear actions of all the elements of the diagram are compatible with the morphisms of the diagram. Given that $X(\D)$, $\tilde{X}(\D)$, $Y$, and $\aff{Y}$ are all of them quasi-projective, by \cref{Proposition descend quasi-projective G-variaties} and \cref{Proposition descent quasi-projective semilinear action}, the diagram above descends to a diagram 
\[\xymatrix{ \tilde{X}(\D_{L},\varrho) \ar[r]^{r_{\tilde{X}}} \ar[d]_{\sslash T}& X(\D_{L},\varrho) \ar[d]^{\sslash T} \\ Z \ar[r]_{r_{Z}} & Z_{\mathrm{aff}}, }\]
where $Z_{L}\cong Y$.

\Addresses


\begin{thebibliography}{EGAIV-III}


\bibitem[AH06]{AH06} \textsc{Klaus Altmann and J\"{u}rgen Hausen}, \textit{Polyhedral divisors and algebraic torus actions}, Mathematische Annalen \textbf{334}, (2006). 557-607.

\bibitem[AHS08]{AHS08} \textsc{Klaus Altmann; J\"{u}rgen Hausen and Hendrik S\"{u}ss}, \textit{Gluing affine torus actions via divisorial fans}, Transformation Groups \textbf{13}, (2008). 215-242.



\bibitem[BH06]{BH06} \textsc{Florian Berchtold and J\"{u}rgen Hausen}, \textit{G{IT} equivalence beyond the ample cone}, Michigan Math. J. \textbf{54}, (2006). 438-515.

\bibitem[BLR90]{BLR90} \textsc{Siegfried Bosch, Werner L\"utkebohmert and Michel Raynaud}, \textit{N\'eron models} (Ergebnisse der Mathematik und ihrer Grenzgebiete (3) [Results in Mathematics and Related Areas (3)]), Vol. \textbf{21}. Springer-Verlag, Berlin (1990). x+325.

\bibitem[Bor20]{Bor20} \textsc{Mikhail Borovoi}, \textit{Equivariant models of spherical varieties}, Transform. Groups \textbf{25}, (2020). 391-439.

\bibitem[Bou06]{Bou06} \textsc{Nicolas Bourbaki}, \textit{{\'E}l{\'e}ments de math{\'e}matique. {Alg{\`e}bre} commutative. {Chapitres} 1 {\`a} 4.} (Reprint of the 1985 original). Springer, Berlin (2006).

\bibitem[Bri17]{Bri17} \textsc{Michel Brion}, \textit{Algebraic group actions on normal varieties}, Trans. Moscow Math. Soc. \textbf{78}, (2017). 85-107.

\bibitem[Bri24a]{Bri24a} \textsc{Michel Brion}, \textit{Actions of finite group schemes on curves}, Pure Appl. Math. Q. \textbf{20}, (2024). 1065-1095.

\bibitem[Bri24b]{Bri24b} \textsc{Michel Brion}, \textit{Equivariantly normal varieties for diagonalizable group actions}, Preprint, arXiv:2405.12020 [math.AG] (2024).

\bibitem[Bri26]{Bri26} \textsc{Michel Brion}, \textit{Equivariantly normal varieties}, Preprint, [math.AG] (2026).


\bibitem[CLS11]{CLS11} \textsc{David A. Cox; John B. Little and Henry K. Schenck}, \textit{Toric varieties} (Graduate Studies in Mathematics), Vol. \textbf{124}. American Mathematical Society, Providence, RI, (2011).


\bibitem[Dan78]{Dan78} \textsc{I. V. Danilov}, \textit{Geometry of toric varieties}, Russ. Math. Surv. \textbf{33}, (1978). 97--154.

\bibitem[Dem70]{Dem70} \textsc{Michel Demazure}, \textit{Sous-groupes alg\'ebriques de rang maximum du groupe de
              {C}remona}, Ann. Sci. {\'E}c. Norm. Sup{\'e}r. (4) \textbf{3}, (1970). 507--588.

\bibitem[DG70]{DG70} \textsc{Michel Demazure and Pierre Gabriel}, \textit{Groupes alg{\'e}briques. {Tome} {I}: {G{\'e}om{\'e}trie} alg{\'e}brique. {G{\'e}n{\'e}ralit{\'e}s}. {Groupes} commutatifs}. ({Avec} un appendice `{Corps} de classes local' par {Michiel} {Hazewinkel}). {Paris: {Masson} et {Cie}, {\'E}diteur; {Amsterdam}: {North}-{Holland} {Publishing} {Company}. xxvi+700 p. (1970).}



\bibitem[Dun16]{Dun16} \textsc{Alexander Duncan}, \textit{Twisted forms of toric varieties}, Transform. Groups \textbf{21}, (2016). (763-802).


\bibitem[EGAIII-I]{EGAIII-I} \textsc{A. Grothendieck}, \textit{\'El\'ements de g\'eom\'etrie alg\'ebrique. {II}. \'Etude globale \'el\'ementaire de quelques classes de morphismes}, {Inst. Hautes \'Etudes Sci. Publ. Math.} \textbf{8}, (1961). 1--222.


\bibitem[EGAIV-II]{EGAIV-II} \textsc{A. Grothendieck}, \textit{{{\'E}l{\'e}ments de g{\'e}om{\'e}trie alg{\'e}brique. {IV}: {\'E}tude locale des sch{\'e}mas et des morphismes de sch{\'e}mas. ({S{\'e}conde} partie)}}, {Publ. Math., Inst. Hautes {\'E}tud. Sci.} \textbf{24}, (1965). 1--231.



\bibitem[ELFST14]{ELFST14} \textsc{E. Javier Elizondo, Paulo Lima-Filho, Frank Sottile and Zach Teitler}, \textit{Arithmetic toric varieties}, {Math. Nachr.} \textbf{287}, (2014). 216--241.


\bibitem[Fuj83]{Fuj83} \textsc{Takao Fujita}, \textit{Semipositive line bundles}, {J. Fac. Sci., Univ. Tokyo, Sect. I A} \textbf{30}, (1983). 353--378.

\bibitem[Ful93]{Ful93} \textsc{William Fulton}, \textit{Introduction to toric varieties} (Annals of Mathematics Studies) (The William H. Roever Lectures in Geometry), Vol. \textbf{131}. Princeton University Press, Princeton, NJ, (1993). xii+157 pp.

\bibitem[FZ03]{FZ03} \textsc{Hubert Flenner and Mikhail Zaidenberg}, \textit{Normal affine surfaces with {{\(\mathbb{C}^*\)}}-actions}, {Osaka J. Math.} \textbf{40}, (2003). (981-1009).


\bibitem[Gill22a]{Gil22a} \textsc{Pierre-Alexandre Gillard}, \textit{Real torus actions on real affine algebraic varieties}, {Math. Z.} \textbf{301}, (2022). (1507-1536).

\bibitem[Gil22b]{Gil22b} \textsc{Pierre-Alexandre Gillard}, \textit{Torus actions on affine varieties over characteristic zero fields}, arXiv, (2022). https://arxiv.org/abs/2208.01495.

\bibitem[GL73]{GL73} \textsc{Jacob Eli Goodman and Alan Landman}, \textit{Varieties proper over affine schemes}, {Invent. Math.} \textbf{20}, (1973). (267-312).

\bibitem[GW20]{GW20} \textsc{Ulrich G\"ortz and Torsten Wedhorn}, \textit{{Algebraic geometry {I}. {S}chemes---with examples and exercises}} (Springer Studium Mathematik-Master) (Second edition). {Springer Spektrum, Wiesbaden}. vii+625 pp. (2020).




\bibitem[Hur11]{Hur11} \textsc{M. Huruguen}, \textit{Toric varieties and spherical embeddings over an arbitrary field}, J. Algebra \textbf{342}, (2011). 212--234. 




\bibitem[Kam75]{Kam75} \textsc{T. Kambayashi}, \textit{On the absence of nontrivial separable forms of the affine plane}, J. Algebra \textbf{35}, (1975). (449-456).

\bibitem[KKMSD73]{KKMSD73} \textsc{G. Kempf, F. Knudsen, D. Mumford and Bernard Saint-Donat}, \textit{Toroidal embeddings. {I}} (Lecture Notes in Mathematics) (Lect. Notes Math.), Vol. \textbf{339}. Springer, Cham, (1973).

\bibitem[KM08]{KM08} \textsc{J{\'a}nos Koll{\'a}r and Shigefumi Mori}, \textit{Birational geometry of algebraic varieties. {With} the collaboration of {C}. {H}. {Clemens} and {A}. {Corti}}, {(Paperback reprint of the hardback edition 1998)}, {(Camb. Tracts Math.)}, Vol. \textbf{134}. {Cambridge: Cambridge University Press, (2008).}


\bibitem[Lan15]{Lan15} \textsc{Langlois, Kevin}, \textit{Polyhedral divisors and torus actions of complexity one over arbitrary fields}, J. Pure Appl. Algebra \textbf{219}, (2015). (2015-2045)

\bibitem[LS13]{LS13} \textsc{Alvaro Liendo and Hendrik S{\"u}ss}, \textit{Normal singularities with torus actions}, {T{\^o}hoku Math. J. (2)} \textbf{65}, (2013). (105-130)

\bibitem[Liu02]{Liu02} \textsc{Qing Liu}, \textit{Algebraic geometry and arithmetic curves} (Oxford Graduate Texts in Mathematics), Translated from the French by Reinie Ern\'{e}, Oxford Science Publications, Vol. \textbf{6}. Oxford University Press, Oxford, (2002). xvi+576 pp.

\bibitem[LAT26a]{LAT26a} \textsc{Giancarlo Lucchini-Arteche and Ronan Terpereau},
\textit{Equivariant automorphism group and real forms of complexity-one varieties}, Preprint, {arXiv}:2507.18475v3 [math.{AG}] (2026).

\bibitem[LAT26b]{LAT26b} \textsc{Giancarlo Lucchini-Arteche and Ronan Terpereau},
\textit{Equivariant automorphism group and real forms of complexity-one varieties}, Preprint, {arXiv}:2507.18475v4 [math.{AG}] (2026).


\bibitem[MN25b]{MN25b} \textsc{Gary Mart\'inez-N\'unez}, \textit{Divisorial fans and algebraic torus actions over arbitrary fields}, Preprint, arXiv:2512.22432v4 [math.AG] (2025).

\bibitem[MNT26]{MNT26} \textsc{Gary Mart\'inez-N\'unez and Ronan Terpereau}, \textit{Classification of equivariantly normal curves via Altmann--Hausen--S\"uss theory}, Preprint, {arXiv}:2607.06494 [math.{AG}] (2026).

\bibitem[MFK94]{GIT} \textsc{D. Mumford and J. Fogarty and F. Kirwan}, \textit{Geometric invariant theory} (Ergebnisse der Mathematik und ihrer Grenzgebiete (2) [Results in Mathematics and Related Areas (2)]), Vol. \textbf{34}. (Third edition). Springer-Verlag, Berlin, (1994). xiv+292 pp.



\bibitem[Nag61]{Nag61} \textsc{Masayoshi Nagata}, \textit{Complete reducibility of rational representations of a matric group}, J. Math. Kyoto Univ. \textbf{1}, (1961). (87-99).

\bibitem[Nag62]{Nag62} \textsc{Masayoshi Nagata}, \textit{Imbedding of an abstract variety in a complete variety}, J. Math. Kyoto Univ. \textbf{2}, (1962). (1-10).




\bibitem[Oda78]{Oda78} \textsc{Tadao Oda}, \textit{Torus embeddings and applications} (Tata Institute of Fundamental Research Lectures on Mathematics and Physics), Vol. \textbf{57}. Based on joint work with Katsuya Miyake. Tata Institute of Fundamental Research, Bombay; Springer-Verlag, Berlin-New York, (1978). xi+175 pp.



\bibitem[Pin77]{Pin77} \textsc{Henry C. Pinkham}, \textit{Normal surface singularities with {{\(\mathbb{C}^*\)}} action}, Math. Ann. \textbf{227}, (1977). (183-193).









\bibitem[Sta25]{stacks-project} \textsc{The {Stacks Project Authors}}, \textit{Stacks Project}. \url{https://stacks.math.columbia.edu} (2025).


\bibitem[Vol10]{Vol10} \textsc{Robert Vollmert}, \textit{Toroidal embeddings and polyhedral divisors}, Int. J. Algebra \textbf{4}, (2010). (383-388).


\bibitem[Wat79]{Wat79} \textsc{William C. Waterhouse}, \textit{Introduction to affine group schemes} {(Graduate Texts in Mathematics)}, Vol. \textbf{66}. Springer, Cham, (1979).


\end{thebibliography}
\end{document}